\documentclass[11pt,reqno]{amsart}
\usepackage{amsmath,amsfonts,amsthm,amssymb}
\usepackage{graphicx,psfrag,overpic}
\usepackage{mathrsfs}
\usepackage{bm}
\usepackage[all,cmtip,line]{xy}
\usepackage{array}

\usepackage[usenames,dvipsnames]{xcolor}

\newtheorem{lemma}{Lemma}[section]
\newtheorem{problemL}[lemma]{Problem}
\newtheorem{theorem}[lemma]{Theorem}

\newtheorem{corollary}[lemma]{Corollary}
\newtheorem{remark}[lemma]{Remark}
\newtheorem{proposition}[lemma]{Proposition}
\newtheorem{definition}[lemma]{Definition}
\newtheorem{condition}[lemma]{Condition}
\theoremstyle{lemma}

\newtheorem*{note}{Note}

\theoremstyle{definition}

\newcommand \gam{\gamma}
\newcommand \R{\mathbb{R}}
\newcommand \Om{\Omega}
\newcommand \der{\partial}
\newcommand \vphi{\varphi}

\newcommand \mcl{\mathcal}
\newcommand \Gam{\Gamma}
\newcommand \alp{\alpha}
\newcommand \tx{\text}
\newcommand \til{\tilde}

\newcommand \ol{\overline}
\newcommand \eps{\varepsilon}

\newcommand \iu{u_{\infty}}

\newcommand \irho{\rho_{\infty}}
\newcommand \om{\omega}
\newcommand \rx{{\rm{\bf x}}}

\newcommand \Div{\rm{div}}
\newcommand \rhoi{\rho_{\infty}}
\newcommand \ui{u_{\infty}}

\newcommand \lb{P_1}
\newcommand \rb{P_0}

\newcommand \shock{\Gam_{\rm{sh}}}

\newcommand \fshock{f_{\rm sh}}
\newcommand \ipsi{\psi_{\infty}}

\newcommand \tw{\theta_w}

\newcommand \rhosonic{\rho_{\rm{sonic}}}
\newcommand \epsmod{\mathfrak{q}_{\gam}(\eps)}

\numberwithin{equation}{section}
\numberwithin{figure}{section}

\begin{document}

\title[Detached shock past a blunt body]{Detached shock past a blunt body}

\author{Myoungjean Bae}
\address{Myoungjean Bae, Department of Mathematics\\
         POSTECH\\
          San 31, Hyojadong, Namgu, Pohang, Gyungbuk, Republic of Korea 37673;
         Korea Institute for Advanced Study
85 Hoegiro, Dongdaemun-gu,
Seoul 130-722,
Republic of Korea
}
\email{mjbae@postech.ac.kr}

\author{Wei Xiang}
\address{Wei Xiang, Department of Mathematics, City Univerisity of Hong Kong, Hong Kong, China}
\email{weixiang@cityu.edu.hk}

\begin{abstract}
In $\R^2$, a symmetric blunt body $W_b$ is fixed by smoothing out the tip of a symmetric wedge $W_0$ with the half-wedge angle $\theta_w\in (0, \frac{\pi}{2})$. We first show that if a horizontal supersonic flow of uniform state moves toward $W_0$ with a Mach number $M_{\infty}>1$ being sufficiently large depending on $\theta_w$, then the half-wedge angle $\theta_w$ is less than {\emph{the detachment angle}} so that there exist two shock solutions, {\emph{a weak shock solution and a strong shock solution}}, with the shocks being straight and attached to the vertex of the wedge $W_0$. Such shock solutions are given by a shock polar analysis, and they satisfy entropy conditions. The main goal of this work is to construct a detached shock solution of the steady Euler system for inviscid compressible irrotational flow in $\R^2\setminus W_b$. Especially, we seek a shock solution with the far-field state given as the strong shock solution obtained from the shock polar analysis. Furthermore, we prove that the detached shock forms a convex curve around the blunt body $W_b$ if the Mach number of the incoming supersonic flow is sufficiently large, and if the boundary of $W_b$ is convex.
\end{abstract}

\keywords{blunt body, detached shock, Euler system, free boundary problem,  inviscid compressible flow, irrotational,  shock polar, strong shock, transonic shock }
\subjclass[2010]{
35A01, 35J25, 35J62, 35M10, 35Q31, 35R35, 76H05, 76L05, 76N10
}

\date{\today}

\maketitle


\section{Detached shock problem and its motivation}
For a constant $\gam > 1$, called {\emph{an adiabatic exponent}}, we define a function $\mathfrak{h}$ by
\begin{equation}
\label{definition-h}
  \mathfrak{h}(\rho)=
  \frac{\rho^{\gam-1}}{\gam-1}.
\end{equation}
The steady {\emph{Euler system}} of irrotational flow,
\begin{equation}
\label{E-system}
\begin{split}
&\der_{x_1}(\rho u_1)+\der_{x_2}(\rho u_2)=0\\
&\der_{x_1}u_{2}-\der_{x_2}{u_1}=0\\
&\frac 12|{\bf u}|^2+ \mathfrak{h}(\rho)
=B_0 \quad\tx{($B_0>0$: a constant)}
\end{split}
\end{equation}
governs two dimensional isentropic irrotational steady flow of inviscid compressible ideal polytropic gas.
The constant $B_0>0$ is called the {\emph{Bernoulli's constant}}. And, the functions
$(\rho, u_1, u_2)$ represent density, horizontal and vertical components of velocity, respectively.
The velocity ${\bf u}$ is expressed as ${\bf u}=u_1{\hat{\bf e}_1}+u_2{\hat{\bf e}_2}$, where $\hat{\bf e}_i$ is the unit vector in the $x_i$-direction for $i=1, 2$.

The local sound speed $c=c(\rho)$ and the Mach number $M=M(\rho, {\bf u})$ of the system \eqref{E-system} are given by
\begin{equation}
\label{definition-soundsp-mach}
c(\rho)=\rho^{\frac{\gam-1}{2}},\quad M(\rho, {\bf u})=\frac{|{\bf u}|}{c}.
\end{equation}
The flow governed by \eqref{E-system} is called {\emph{subsonic}} if $M(\rho, {\bf u})<1$, {\emph{sonic}} if $M(\rho, {\bf u})=1$, and {\emph{supersonic}} if $M(\rho, {\bf u})>1$.

The goal of this paper is to solve the following problem:\\
{\emph{{\textbf{Problem I.}}
Given an incoming horizontal supersonic flow of uniform state, find an entropy solution of \eqref{E-system} in $\R^2\setminus W_b$ with a detached shock past a blunt body $W_b$ so that the solution satisfies the slip boundary condition
\begin{equation}
\label{slip-bc-Wb}
		{\bf u}\cdot{\bf n}_b=0\quad\tx{on $\der W_b$},
		\end{equation}
where ${\bf n}_b$ is the unit  normal on $\der W_b$ pointing interior to $\R^2\setminus W_b$.
(Fig.\ref{figure-dshock-intro}). }}
\begin{figure}[htp]
\centering
\begin{psfrags}
\psfrag{x}[cc][][0.8][0]{$x_1$}
\psfrag{b}[cc][][0.8][0]{$W_b$}
\psfrag{sh}[cc][][0.8][0]{$\Gam^{\rm det}_{\rm shock}\phantom{aa}$}
\psfrag{ic}[cc][][0.8][0]{$M_{\infty}>1$}
\psfrag{tw}[cc][][0.8][0]{$\theta_w$}
\includegraphics[scale=0.5]{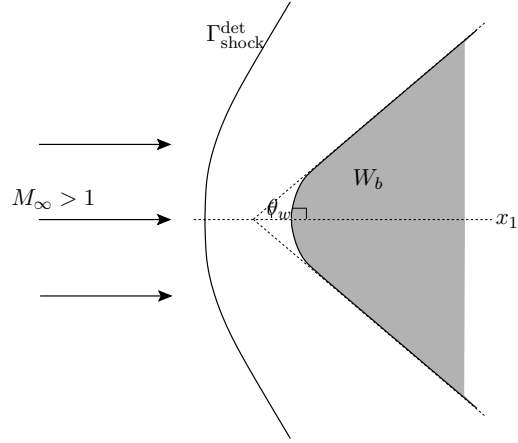}
\caption{Datached shock past a blunt body $W_b$}\label{figure-dshock-intro}
\end{psfrags}
\end{figure}
\smallskip

The rigorous definitions of {\emph{a shock}} and {\emph{an entropy solution}} are given in \S \ref{section-2}.

\begin{figure}[htp]
\centering
\begin{psfrags}
\psfrag{W}[cc][][0.8][0]{$W_0\cap \{x_2\ge 0\}$}
\psfrag{Ub}[cc][][0.8][0]{$\phantom{aaaaaaaa}(\rho_1, u_1, u_1\tan \tw)$}
\psfrag{Uf}[cc][][0.8][0]{$(\irho, \iu,0)$}
\psfrag{s}[cc][][0.8][0]{$\Gam^{\rm att}_{\rm shock}$}
\psfrag{tw}[cc][][0.8][0]{$\theta_w$}
\includegraphics[scale=1.0]{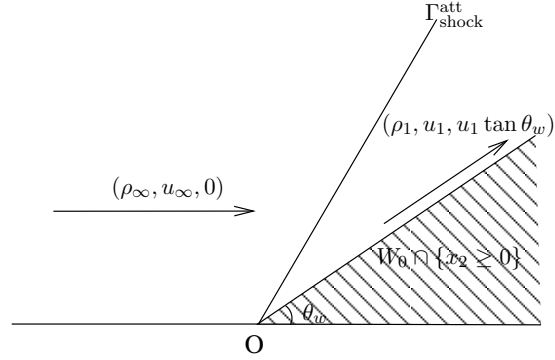}
\caption{Attached oblique shock past $W_0$ in $\R^2\cap\{x_2\ge 0\}$}\label{figure-obshock-intro}
\end{psfrags}
\end{figure}

\smallskip

For a fixed angle $\theta_w\in(0, \frac{\pi}{2})$, let a symmetric wedge $W_0$ in $\R^2$ with the half-angle $\theta_w$ be given by
\begin{equation}
\label{definition-W0}
W_0:=\{{\bf x}=(x_1,x_2)\in \R^2: x_1\ge |x_2|\cot \theta_w\}.
\end{equation}
Given incoming supersonic flow of a uniform state $(\rho, u_1, u_2)=(\irho, \iu,0)$ with $M_{\infty}=\frac{\iu}{\irho^{\frac{\gam-1}{2}}}>1$, it is well known that if $\theta_w$ is less than a critical angle $\theta_{\rm det}$, called {\emph{the detachment angle}}, then {\emph{the shock polar curve}}(\cite{CF, E, K}) yields two entropy solutions of \eqref{E-system} in $\R^2\setminus W_0$, called {\emph{a strong shock soluiton}} and {\emph{a weak shock solution}} so that each of these solutions contains a straight oblique shock $\Gam_{\rm shock}^{\rm att}$ attached to the vertex of the wedge $W_0$ with the downstream state behind the shock given by a uniform state (Fig.\ref{figure-obshock-intro}).
The structural stability of the attached oblique shocks under small perturbations of the incoming supersonic flow,  or under small perturbations of the wedge boundary away from the vertex point are extensively studied in \cite{CCF1, CCJ, ChenFang, F,  FX, YZ} and the references therein. And, the dynamical stability/instability of the attached oblique shock solutions past the symmetric wedge $W_0$ have been studied in \cite{BCF, CF2,  CF3, CFX, Elling2, EL} and the references therein.
\medskip

{\textbf{\emph{Motivation.}}} Our main goal is to find an entropy solution past a blunt body $W_b$(Fig. \ref{figure-dshock-intro}). As we shall see in Definition \ref{definition-bluntbody-ftn}, the boundary $\der W_b$ of the blunt body is given from perturbing the boundary $\der W_0$ of the wedge $W_0$, especially near the vertex $O$. We assume that $\der W_b$ is $C^3$,
 and symmetric about the $x_1$-axis. Therefore, $\der W_b$ is perpendicular to the $x_1$-axis at its $x_1$-intercept. But, we do not assume that $\der W_b$ is necessarily a small perturbation of $\der W_0$. Then it can be checked by using the shock polar and the slip boundary condition \eqref{slip-bc-Wb} that if an entropy solution to \eqref{E-system} contains a shock, then the shock must be completely detached from the blunt body $W_b$, even from its smooth peak. This observation naturally raises a question on the existence of a detached shock solution past the blunt body $W_b$.
\medskip

 In this paper, we give a rigorous proof of the existence of a detached shock solution past $W_b$. Furthermore, we prove that the solution constructed in this work is convex if $M_{\infty}$ is sufficiently large, and if $\der W_b$ is convex (see Fig. \ref{figure-dshock-intro}). To the best of our knowledge, this is the very first rigorous result on a detached shock past a blunt body.

In \cite{CCJ}, the authors proved the existence of a family of detached shock solutions of \eqref{E-system} in a bounded domain past the symmetric wedge $W_0$,where the half-wedge angle $\tw$ of $W_0$ is less than {\emph{the detachment angle}}. In this case, the shock polar analysis yields two entropy solutions with attached straight oblique shocks: {\emph{a strong shock}} and {\emph{a weak shock}} (see Fig. \ref{figure-obshock-intro}). The main focus of \cite{CCJ} is on showing that the value of functional $\int_{\Om} \rho^{\gam}\,d{\rm x}$ evaluated for any detached shock solution is greater than the one evaluated for the attached strong oblique shock solution.
Differently from the case of $W_0$, however, it is impossible to have an attached oblique shock solution past the blunt body $W_b$. See Lemma \ref{lem:3.3} in \S \ref{subsection-intro-to-d-shock} for further detailed analysis. Another different feature of the work in this paper from the one of \cite{CCJ} is that we intend to construct a global-in-space detached shock solution in $\R^2\setminus W_b$. 

\smallskip

For a given $\tw\in(0, \frac{\pi}{2})$, let a symmetric wedge $W_0$ be given by \eqref{definition-W0}. The blunt body $W_b$ considered in this paper is given as a perturbation of $W_0$ as follows:
\begin{definition}
	\label{definition-bluntbody-ftn}
	For a fixed constant $h_0>0$, let a function $b:\R\rightarrow \R$ satisfy the following properties:
\begin{figure}[htp]
\centering
\begin{psfrags}
\psfrag{x}[cc][][0.8][0]{$x_1$}
\psfrag{b}[cc][][0.8][0]{$x_1=b(x_2)\phantom{aaa}$}
\psfrag{Uf}[cc][][0.8][0]{${\bf u}_{\infty}$}
\psfrag{ic}[cc][][0.8][0]{$x_2=s_{\rm st}x_1$}
\psfrag{tw}[cc][][0.8][0]{$\theta_w$}
\includegraphics[scale=0.6]{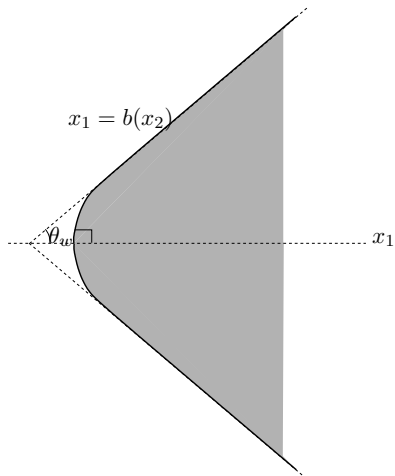}
\caption{Blunt body $W_b$ induced from a symmetric wedge $W_0$}\label{figure-Wb}
\end{psfrags}
\end{figure}
	\begin{itemize}
		\item [($b_1$)] $b(x_2)=b(-x_2)$ for all $x_2\in\R$;
		\item [($b_2$)] $b\in C^3(\R)$;
		\item [($b_3$)] $b'(x_2)>0$ for all $x_2>0$;
		\item [($b_4$)] $b''(x_2)\ge 0$ for all $x_2\ge 0$;
		\item[($b_5$)] $b(x_2)=x_2\cot\theta_w$ for $x_2\ge h_0$.
	\end{itemize}

For such a function $b$, we define a blunt body $W_b$ by
	\begin{equation}\label{definition-Wb}
	W_b:=\{{\rx}=(x_1,x_2)\in \R^2: x_1\ge  b(x_2)\}.
	\end{equation}
For simplicity of notations, we set
\begin{equation}
\label{definition-b0}
  b_0:=b(0).
\end{equation}
\end{definition}

\begin{remark}
\label{remark-b-function}
	It can be directly checked from $(b_1)$--$(b_5)$ stated in Definition \ref{definition-bluntbody-ftn} that the function $b$ satisfies the following properties:
	\begin{align}
	\label{b-property1}
	& b'(0)=0, \\
\label{b-property3}
	& 0\le b'(x_2)\le \cot \theta_w,\quad\tx{and}\quad b(x_2)\ge x_2 \cot \theta_w\quad\tx{for all $x_2\ge 0$},\\
	\label{b-property2}
	& \sum_{k=1}^3 \|\frac{d^k b}{dx_2^k}\|_{C^0(\R)}\le m_b
\end{align}
	for some constant $m_b>0$.
\end{remark}
Let us define
\begin{equation}
\label{definition-half-blunt-body}
W_b^+:=W_b\cap \{x_2 \ge 0\},
\end{equation}
and set
\begin{equation*}
		\Gam_{\rm{sym}}:=\{(x_1,0):x_1\le b_0\},\quad
		\Gam_b:=\{(b(x_2),x_2):x_2>0\}.
		\end{equation*}
The incoming horizontal supersonic flow is given by
\begin{equation}
\label{definition-incoming-flow}
(\rho, {\bf u})=(\irho, \iu,0)\quad \tx{with $\iu>\irho^{\frac{\gam-1}{2}}>0$.}
\end{equation}

Since $W_b$ is symmetric about $x_1$-axis, and the incoming supersonic flow is in a horizontal direction and in a uniform state, we can restate {\emph{Problem I}} as follows:
\quad\\
{\emph{{\textbf{Problem II.}}
For incoming horizontal supersonic flow given by \eqref{definition-incoming-flow}, find an entropy solution of \eqref{E-system} with a detached shock past a blunt body so that the solution satisfies the slip boundary condition
		\begin{equation}
		\label{slip-BC-original}
		{\bf u}\cdot{\bf n}_w=0\quad\tx{on $\Gam_{\rm{sym}}\cup \Gam_b$},
		\end{equation}
where ${\bf n}_w$ is the unit  normal on $\Gam_{\rm{sym}}\cup \Gam_b$ pointing interior to $(\R^2\cap\{x_2>0\})\setminus W_b^+$.
 }}

\smallskip

The rest of the paper is devoted to solving Problem II.
\smallskip

In Section \ref{section-2}, we provide preliminaries and state our main theorem.
In \S \ref{subsec-2-2}, we explain about attached oblique shocks past symmetric wedges in $\R^2$ through the well known {\emph{shock polar analysis}}. Particularly, we show in Lemma \ref{lemma-shock-polar1} that for any symmetric wedge $W_0$ of the half-wedge angle $\theta_w\in(0, \frac{\pi}{2})$, if the Mach number $M_{\infty}$ of incoming horizontal supersonic flow is sufficiently large depending on $\theta_w$, then the half-wedge angle $\theta_w$ is less than {\emph{the detachment angle so that there exist two distinct attached oblique shock solutions (a strong shock solution and a weak shock solution) so that both of which satisfy the entropy condition. This lemma plays an important role in establishing the existence of a detached shock solution past a blunt body. More specifically, Lemma \ref{lemma-shock-polar1} yields a far-field asymptotic state of a detached shock solution past a blunt body $W_b$, where $W_b$ is given from smoothing out the tip of a symmetric wedge $W_0$ with the half-wedge angle $\tw\in(0, \frac{\pi}{2})$. In \S \ref{subsection-intro-to-d-shock}, we explain in detail why a shock past a blunt body $W_b$ must be completely detached from $\der W_b$. 
}}
\smallskip

In Section \ref{section-stream-function}, we formulate a free boundary problem by using a stream function $\psi$ so that any solution to the free boundary problem yields a detached shock solution in $(\R^2\cap\{x_2\ge 0\})\setminus W_b^+$
in the sense of the main theorem. And, the solvability of this free boundary problem is stated as Theorem \ref{main-theorem-psi}.
Then we prove that the solvability of the free boundary problem of the stream function $\psi$ implies the existence of a detached shock solution past the blunt body $W_b$.
Theorem \ref{main-theorem-psi}, the solvability theorem of the free boundary problem of a stream function $\psi$ in $(\R^2\cap\{x_2\ge 0\})\setminus W_b^+$, is proved in Sections \ref{section-fbp-cf} and \ref{section-main-thm-pf}.
In Section \ref{section-fbp-cf}, we introduce a free boundary problem in a cut-off domain of a finite height $L$ from the $x_1$-axis, and prove the existence of a solution to the free boundary problem. In Section \ref{section-main-thm-pf}, we take a strictly increasing sequence $\{L_n\}_{n=1}^{\infty}$ with $\displaystyle{\lim_{n\to \infty}L_n=\infty}$. For each $n\in \mathbb{N}$, we fix a solution to the free boundary problem in the cut-off domain of the height $L_n$ from the $x_1$-axis to get a sequence $\{(f_n, \psi_n)\}_{n=1}^{\infty}$ of solutions to the free boundary problems in cut-off domains. For each $n\in \mathbb{N}$, the graph of $x_1=f_n(x_2)$ for $0\le x_2\le L_n$ represents the detached shock for a cut-off detached shock problem, and $\psi_n$ is the corresponding stream function that satisfies $\nabla^{\perp}\psi=\rho(u_1,u_2)$. In Section \ref{section-main-thm-pf} , we take a subsequence of $\{f_n\}$  so that it converges to a function $f_{*}: \R_+\to \R$ on any bounded interval in an appropriately defined norm. Then, we show that there exists a function $\psi_*$ so that $(f_*, \psi_*)$ yields a solution to the free boundary problem in $(\R^2\cap\{x_2> 0\})\setminus W_b^+$, stated in Section \ref{section-stream-function} thus prove Theorem \ref{main-theorem-psi}.
\smallskip

Finally, we prove the convexity of the detached shock in Section \ref{section-convexity} by employing the method developed in \cite{CFX}. But, one cannot directly apply the result from \cite{CFX} because the detached shock solutions constructed in this paper do not necessarily satisfy all the conditions required to apply \cite[Theorem 2.1]{CFX}. Especially, the detached shock solutions do not necessarily satisfy Condition (A6) in the statement of \cite[Theorem 2.1]{CFX}.
Instead, maximum principles for the speed, horizontal velocity and vertical velocity are additionally taken into account, and a new argument has been developed in Section \ref{section-convexity}. Interestingly, it turns out that the convexity of the blunt body $W_b$, which is stated in Definition \ref{definition-bluntbody-ftn}($b_4$), plays a crucial role in proving the convexity of the detached shock, although it is unclear whether the convexity of the blunt body is a necessary condition for the convexity of the detached shock. This will be investigated further in a future work.

\section{Preliminaries and main theorems}
\label{section-2}

\subsection{Basic definitions}
\label{subsection-prelim}

Let $\Om$ be a domain in $\R^2$. Suppose that a non self-intersecting $C^1$ curve $\mcl{S}$ divides $\Om$ into two open and connected subsets $\Om^-$ and $\Om^+$ so that $\Om^-\cap\Om^+=\emptyset$ and $\Om^-\cup \mcl{S} \cup \Om^+=\Om$.
We define $(\rho, u_1, u_2)\in [L^{\infty}(\Om)\cap C^0(\ol{\Om^{\pm}})\cap C^1_{\rm loc}(\Om^{\pm})]^3 $ to be a {\emph{weak solution of \eqref{E-system}}} with a shock $\mcl{S}$ if the following properties are satisfied:
\begin{itemize}
\item[($s_1$)] $\displaystyle{\rho>0}$ in $\Om$;

\item[($s_2$)] The Bernoulli's law $\displaystyle{\frac 12|{\bf u}|^2+\mathfrak{h}(\rho)=B_0}$ holds in $\Om$;

\item[($s_3$)] For any test function $\phi \in C^{\infty}_{c}(\R^2)$, we have
\begin{equation*}
  \int_{\Om} \rho u_1\phi_{x_1}+\rho u_2\phi_{x_2}\, d\rx=
  \int_{\Om} u_2\phi_{x_1}-u_1\phi_{x_2}\,d\rx=0;
\end{equation*}

\item[($s_4$)] 
For any $\rx_0\in \mcl{S}$, we have
    \begin{equation*}
    \lim_{{\rx\to \rx_0}\atop{\rx\in \Om^-}}{\bf u}\cdot{\bm\nu}\neq 0,\quad \lim_{{\rx\to \rx_0}\atop{\rx\in \Om^+}}{\bf u}\cdot{\bm\nu}\neq 0,\quad \tx{and}\lim_{{\rx\to \rx_0}\atop{\rx\in \Om^-}}{\bf u}\cdot{\bm\nu}\neq \lim_{{\rx\to \rx_0}\atop{\rx\in \Om^+}}{\bf u}\cdot{\bm\nu}
    \end{equation*}
    for a unit normal ${\bm\nu}$ of $\mcl{S}$ at $\rx_0$.
\end{itemize}

In the following, an equivalent definition is given through integration by parts.
\begin{definition}[Shock]
\label{definition-shock-solution}
We define $(\rho, u_1, u_2)\in [L^{\infty}(\Om)\cap C^0(\ol{\Om^{\pm}})\cap C^1_{\rm loc}(\Om^{\pm})]^3 $ to be a {\emph{weak solution to \eqref{E-system}}} with a shock $\mcl{S}$ if the following properties are satisfied:
\begin{itemize}
\item[($S_1$)] The properties {\emph{($s_1$)}} and {\emph{($s_2$)}} hold, and $S$ is $C^1$;

\item[($S_2$)] In $\Om^{\pm}$, $(\rho, u_1, u_2)$ satisfy the equations
\begin{equation*}
  \der_{x_1}(\rho u_1)+\der_{x_2}(\rho u_2)=0,\quad\tx{and}\quad
  \der_{x_1}u_{2}-\der_{x_2}{u_1}=0 \quad \tx{pointwisely;}
\end{equation*}

\item[($S_3$)] For each point ${\rx_*}\in \mcl{S}$, define
		\begin{equation*}
		(\rho^+, u_1^+, u_2^+)(\rx_*):=\lim_{{\rx\to \rx_*}\atop{\rx\in \Om^+}} (\rho, u_1, u_2)(\rx),\quad
		(\rho^-, u_1^-, u_2^-)(\rx_*):=\lim_{{\rx\to \rx_*}\atop{\rx\in \Om^-}} (\rho, u_1, u_2)(\rx).
		\end{equation*}
Then, $(\rho, u_1, u_2)$ satisfy the Rankine-Hugoniot conditions
\begin{equation}
\label{RH-potential}
  \rho^+(u_1^+, u_2^+)\cdot{\bm \nu}=\rho^-(u_1^-, u_2^-)\cdot{\bm \nu},\quad\tx{and}\quad
  (u_1^+, u_2^+)\cdot{\bm \tau}=(u_1^-, u_2^-)\cdot{\bm \tau}\quad\tx{on $\mcl{S}$},
\end{equation}
where ${\bm\nu}$ is a unit normal, and ${\bm \tau}$ is a unit tangential on $\mcl{S}$.

\item[($S_4$)] 
On $\mcl{S}$, we have $(u_1^+, u_2^+)\cdot{\bm \nu}\neq 0$ (or equivalently $(u_1^-, u_2^-)\cdot{\bm \nu}\neq 0$) and $(u_1^+, u_2^+)\cdot{\bm \nu}\neq (u_1^-, u_2^-)\cdot{\bm \nu}$.

\item[($S_5$)] On $\der \Om$, the slip boundary condition
\begin{equation*}
  (u_1, u_2)\cdot {\bf n}=0
\end{equation*}
holds for the inward unit normal vector field ${\bf n}$ on $\der \Om$.
\end{itemize}
\end{definition}

\begin{definition}[Entropy solution]
\label{definition-shock-admsble}

Let $(\rho, u_1, u_2)$ be a weak solution in $\Om$ with a shock $\mcl{S}$ in the sense of Definition \ref{definition-shock-solution}. We call the solution {\emph{an entropy solution}} if
\begin{equation}
\label{condition-admsbl}
  0<\rho^-<\rho^+<\infty,\quad\tx{and}\quad
  0<(u_1^+, u_2^+)\cdot{\bm\nu}<(u_1^-, u_2^-)\cdot{\bm\nu}<\infty
\end{equation}
hold on ${\mcl{S}}$, where the unit normal ${\bm\nu}(=\frac{(u_1^-, u_2^-)-(u_1^+, u_2^+)}{|(u_1^-, u_2^-)-(u_1^+, u_2^+)|})$ on $\mcl{S}$ points interior to $\Om^+$.

\end{definition}

\subsection{Attached oblique shocks past $W_0$}
\label{subsec-2-2}

For fixed constants $\gam > 1$ and $B_0>0$, define
\begin{equation}
\label{definition-Dinfty}
D_{\infty}(\gam, B_0):=\{( \rho_{\infty}, u_{\infty})\in \R^2: \frac 12 u_{\infty}^2+\mathfrak{h}(\rho_{\infty})=B_0,\,\,
\irho>0,\,\,\iu>\irho^{\frac{\gam-1}{2}} \}.
\end{equation}
The set $D_{\infty}(\gam, B_0)$ contains all horizontal supersonic flows with the Bernoulli constant $B_0$.

\begin{lemma}
\label{lemma-shock-polar-general}
For each fixed $(\irho, \iu)\in D_{\infty}(\gam, B_0)$, there exist a unique constant $u_0\in(0, \iu)$ and a unique function $\mathfrak{f}_{\rm polar}:[u_0, \iu]\rightarrow \R_+$ satisfying the following properties:

\begin{itemize}
\item[($p_0$)] $\mathfrak{f}_{\rm polar}\in C^0([u_0, \iu])\cap C^{\infty}((u_0, \iu))$
\item[($p_1$)] $\displaystyle{\mathfrak{f}_{\rm polar}(u_0)=\mathfrak{f}_{\rm polar}(\iu)=0}$;

\item[($p_2$)] $\displaystyle{\mathfrak{f}_{\rm polar}(u)>0}$ for $u_0<u<\iu$;

\item[($p_3$)]If we set
$\displaystyle{
  {\bf u}_{\infty}:=(\iu, 0),\,\,
  {\bf u}:=(u,\mathfrak{f}_{\rm polar}(u)),\,\,
  \rho:=\mathfrak{h}^{-1}(B_0-\frac 12|{\bf u}|^2),}
$
then we have
\begin{equation}
\label{RH1-spolar}
  \rho{\bf u}\cdot {\bm\nu}=\irho {\bf u}_{\infty}\cdot{\bm\nu}\quad\tx{for}\,\,{\bm\nu}=\frac{{\bf u}_{\infty}-{\bf u}}{|{\bf u}_{\infty}-{\bf u}|},
\end{equation}

\begin{equation}
\label{RH2-spolar}
\tx{and}\quad  0<{\bf u}\cdot {\bm\nu}<{\bf u}_{\infty}\cdot{\bm\nu}<\infty\quad\tx{for}\,\,u_0<u<\iu.
\end{equation}

\item[($p_4$)] Any vector ${\bf u}=(u,v)\in \R^2$ satisfying \eqref{RH1-spolar} and \eqref{RH2-spolar} lies either on the curve $v=\mathfrak{f}_{\rm polar}(u)$, or $v=-\mathfrak{f}_{\rm polar}(u)$.

\end{itemize}

Furthermore, such a function $\mathfrak{f}_{\rm polar}$ satisfies
\begin{equation}
\label{spolar-convexity}
  \mathfrak{f}_{\rm polar}''(u)<0\quad\tx{for}\,\,u_0<u<\iu.
\end{equation}
The curve $v=\mathfrak{f}_{\rm polar}(u)$ on the $uv$-plane is called {\emph{the shock polar curve}} of the incoming supersonic flow $(\irho, \iu)$.
\end{lemma}

Note that the shock polar curve discussed in Lemma \ref{lemma-shock-polar-general} is different from the one in \cite[\S121 and \S122]{CF}. The model used in \cite[\S121 and \S122]{CF} is the full Euler system for which the function $ \mathfrak{f}_{\rm polar}$ can be explicitly computed, while we consider irrotational flow in this paper.
To prove Lemma \ref{lemma-shock-polar-general} for the irrotational flow model \eqref{E-system}, a different approach is needed because there is no known explicit formula of $ \mathfrak{f}_{\rm polar}$. In this paper, we briefly demonstrate how to prove Lemma \ref{lemma-shock-polar-general} without details as it is already well known in other references. First, one can refer to \cite[Proposition 2.1]{K} for the existence of the set $\{{\bf u}=(u,v)\}$ that satisfies \eqref{RH1-spolar} and \eqref{RH2-spolar}. Then the implicit function theorem is applied to obtain the unique function $ \mathfrak{f}_{\rm polar}$ that satisfies  the properties ($p_0$)--($p_4$). Finally, by adjusting the proof of \cite[Theorem 1]{E}, one can check that $ \mathfrak{f}_{\rm polar}$ satisfies \eqref{spolar-convexity}. Or, one can refer to \cite[Appendix A]{BCF1} for a detailed proof of \eqref{spolar-convexity}.

\begin{corollary}
\label{corollary-detach-angle}
For each $(\irho, \iu)\in D_{\infty}(\gam, B_0)$, there exists a unique constant $\theta_{\rm det}\in(0, \frac{\pi}{2})$ so that
\begin{itemize}
\item[(i)] if $0\le \theta<\theta_{\rm det}$, then the line $v=u\tan\theta$ intersects $v=\mathfrak{f}_{\rm polar}(u)$ at two distinct points;

\item[(ii)] if $\theta=\theta_{\rm det}$, the line $v=u\tan \theta_{\rm det}$ intersects $v=\mathfrak{f}_{\rm polar}(u)$ at a unique point;

\item[(iii)] if $\theta_{\rm det}<\theta< \frac{\pi}{2}$, then there is no intersection of $v=u\tan\theta$ and $v=\mathfrak{f}_{\rm polar}(u)$.
\end{itemize}
Such $\theta_{\rm det}$ is called the detachment angle.
\end{corollary}

For a symmetric wedge $W_0$ given by \eqref{definition-W0} with the half-wedge angle $\tw\in(0, \frac{\pi}{2})$, let us define
\begin{equation}
\label{definition-upper-wedge}
  \R^2_+:=\{{\rx}=(x_1,x_2)\in \R^2: x_2\ge 0 \},\quad\tx{and}\quad W_0^+:=W_0\cap \R^2_+.
\end{equation}
Given $(\irho, \iu)\in D_{\infty}(\gam, B_0)$, suppose that the half-angle $\theta_w$ of the wedge $W_0$ is less than the detachment angle $\theta_{\rm det}$. By Corollary \ref{corollary-detach-angle}, the curve $v=\mathfrak{f}_{\rm polar}(u)$ intersect the line $v=u\tan \theta_w$ at two distinct points ${\bf u}_{\rm st}=(u_{\rm st}, \mathfrak{f}_{\rm polar}(u_{\rm st}))$ and ${\bf u}_{\rm wk}=(u_{\rm wk}, \mathfrak{f}_{\rm polar}(u_{\rm wk}))$. Without loss of generality, let us assume that $0<u_{\rm st}<u_{\rm wk}<\iu$. And, let us set
\begin{equation*}
  \rho_{\rm st}:=\mathfrak{h}^{-1}(B_0-\frac 12|{\bf u}_{\rm st}|^2),\quad\tx{and}\quad
  \rho_{\rm wk}:=\mathfrak{h}^{-1}(B_0-\frac 12|{\bf u}_{\rm wk}|^2).
\end{equation*}

Let $x_1=s_{\rm st}x_2$ be the line perpendicular to the vector ${\bm\nu}_{\rm st}:=\frac{{\bf u}_{\infty}-{\bf u}_{\rm st}}{|{\bf u}_{\infty}-{\bf u}_{\rm st}|}$. For ${\rm x}=(x_1,x_2)\in \R^2_+$, we define a vector valued function
\begin{equation*}
  {\bf u}(\rm x):=\begin{cases}
  {\bf u}_{\infty}\quad\mbox{for $x_1<s_{\rm st}x_2$},\\
  {\bf u}_{\rm st}\quad\mbox{for $x_1>s_{\rm st}x_2$},
  \end{cases}
\end{equation*}
and we set $\rho({\rm x}):=\mathfrak{h}^{-1}(B_0-\frac 12|{\bf u}|^2)$. Then $(\rho, {\bf u})$ is a shock solution in $\R^2_+\setminus W_0^+$ with the shock $\mcl{S}:x_1=s_{\rm st}x_2$ in the sense of Definition \ref{definition-shock-solution}(Fig.\ref{figure-obshock}). And, it satisfies the entropy condition stated in Definition \ref{definition-shock-admsble}.

\begin{figure}[htp]
\centering
\begin{psfrags}
\psfrag{W}[cc][][0.8][0]{$W_0^+$}
\psfrag{Ub}[cc][][0.8][0]{${\bf u}_{\rm st}$}
\psfrag{Uf}[cc][][0.8][0]{${\bf u}_{\infty}$}
\psfrag{s}[cc][][0.8][0]{$x_1=s_{\rm st}x_2$}
\psfrag{tw}[cc][][0.8][0]{$\theta_w$}
\includegraphics[scale=1.0]{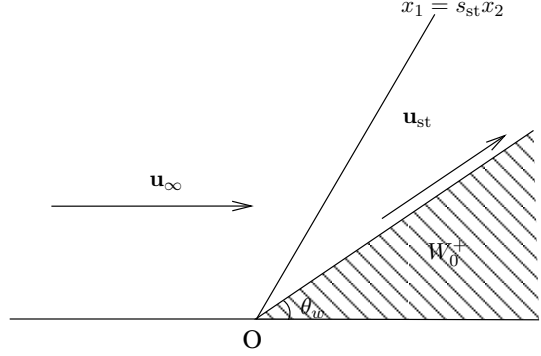}
\caption{Attached oblique shock past $W_0^+$ in $\R^2_+$}\label{figure-obshock}
\end{psfrags}
\end{figure}
By replacing ${\bf u}_{\rm st}$ by ${\bf u}_{\rm wk}$ in the argument right above, another shock solution past $W_0^+$ is given. The shock solution with ${\bf u}_{\rm st}$ is called {\emph{the strong shock solution}}, while the shock solution with ${\bf u}_{\rm wk}$ is called {\emph{the weak shock solution}}.

For $(\irho, \iu)\in D_{\infty}(\gam, B_0)$, set
\begin{equation*}
  M_{\infty}:=\frac{\iu}{\irho^{(\gam-1)/2}}.
\end{equation*}

Next, we show that for any given $\theta_w\in(0, \frac{\pi}{2})$, we have $\theta_w<\theta_{\rm det}$ if $M_{\infty}>1$ is sufficiently large depending on $\theta_w$. Hereafter, we fix the adiabatic exponent $\gam >1$ and the Bernoulli constant $B_0>0$ unless otherwise specified.
\begin{lemma}
	\label{lemma-shock-polar1}
	For any given $\theta_w\in(0, \frac{\pi}{2})$, there exists a small constant $\eps_0\in(0,1)$ depending on $(\gam, B_0, \theta_w)$ so that if $(\irho, \iu)\in D_{\infty}(\gam, B_0)$ satisfies $M_{\infty}\ge \frac{1}{\eps_0}$, then we have
	\begin{equation*}
	\label{inequality-angle-detach}
	\theta_w<\theta_{\rm{det}}.
	\end{equation*}
\end{lemma}

	\begin{proof} The proof is divided into four steps.
		
		{\textbf{1}.} For each $(\irho, \iu)\in D_{\infty}(\gam, B_0)$, let us set
		$
		\eps:=\frac{1}{M_{\infty}}.
		$
Then, $(\irho, \iu)$ can be represented in terms of $\eps$ as follows:
		\begin{equation}
		\label{definition-incoming-epsilon}
\begin{split}
&\iu=\sqrt{\frac{(\gam-1)B_0}{\frac{\gam-1}{2}+\eps^2}},\\
&\irho=\mathfrak{h}^{-1}\left(B_0-\frac{(\gam-1)B_0}{(\gam-1)+2\eps^2}\right)
=
\eps^{\frac{2}{\gam-1}}
\left(\frac{(\gam-1)B_0}{\frac{\gam-1}{2}+\eps^2}\right)^{\frac{1}{\gam-1}}.
\end{split}
		\end{equation}
		Set
		\begin{equation*}
		  \kappa_w:=\tan\theta_w.
		\end{equation*}
		To prove Lemma \ref{inequality-angle-detach}, we need to show that the following nonlinear system for $(\rho, u, s)$
		\begin{equation}
		\label{RH-spolar-Prelim}
		\begin{cases}
		\rho u(1,\kappa_w)\cdot (1, -s)=\rhoi(\ui,0)\cdot (1,-s)\\
		u(1, \kappa_w)\cdot (s, 1)=(\ui, 0)\cdot (s,1)\\
		\frac 12 u^2(1+\kappa_w^2)+\mathfrak{h}(\rho)=B_0
		\end{cases}
		\end{equation}
has two distinct solutions $(\rho, u, s)\in (\irho, \infty)\times (0, \iu) \times (0, \infty)$ so that if we set ${\bf u}=u(1,\kappa_w)$ and ${\bm\nu}=(1,-s)$, then the inequality $0<{\bf u}\cdot {\bm\nu}<{\bf u}_{\infty}\cdot {\bm\nu}$ hold provided that $\eps>0$ is sufficiently small.
Note that ${\bm\nu}$ represents a normal vector of an attached oblique shock: $x_1=sx_2$.

By substituting the expressions given in \eqref{definition-incoming-epsilon} into \eqref{RH-spolar-Prelim}, we get
		\begin{equation}\label{RH-spolar}
		\begin{cases}
		R_1(\rho, u, s,\eps):=\rho u(s\kappa_w-1)
		+ \mathfrak{h}^{-1}\left(B_0-\frac{(\gam-1)B_0}{(\gam-1)+2\eps^2}\right)
\sqrt{\frac{(\gam-1)B_0}{\frac{\gam-1}{2}+\eps^2}}=0\\
		R_2(\rho, u, s,\eps):=u(s+\kappa_w)-s\sqrt{\frac{(\gam-1)B_0}{\frac{\gam-1}{2}+\eps^2}}=0\\
		R_3(\rho, u, s,\eps):=\frac 12 u^2(1+\kappa_w^2)+\mathfrak{h}(\rho)-B_0=0
		\end{cases}.
		\end{equation}

		\medskip
		
		{\textbf{2}.}
		We first solve \eqref{RH-spolar} when $\eps=0(\Rightarrow \irho=0, \iu=\sqrt{2B_0})$. By substituting $\eps=0$ into \eqref{RH-spolar}, we get
		\begin{equation}
		\label{RH-spolar-epsilon-zero}
		\begin{cases}
		\rho u(s\kappa_w-1)=0\\
		u(s+\kappa_w)- s\sqrt{2B_0}=0\\
		\frac 12 u^2(1+\kappa_w^2)+\mathfrak{h}(\rho)-B_0=0
		\end{cases}.
		\end{equation}
		
The first equation in \eqref{RH-spolar-epsilon-zero} holds if either one of the following three cases holds:
		\[\mbox{(i) $\rho=0$},\qquad\mbox{ (ii) $u=0$},\qquad\mbox{ (iii) $s=\frac{1}{\kappa_w}$.}
		\]
We consider these three cases separately.
		\begin{figure}[htp]
			\centering
			\begin{psfrags}
				\psfrag{v}[cc][][1.0][0]{$v$}
				\psfrag{u}[cc][][1.0][0]{$u$}
				\psfrag{np}[cc][][1.0][0]{$\rho=0$}
				\psfrag{w}[cc][][1.0][0]{$s=\frac{1}{\kappa_w}$}
				\psfrag{s}[cc][][1.0][0]{$u=0$}
				\psfrag{ui}[cc][][1.0][0]{$\sqrt{2B_0}$}
				\psfrag{uc}[cc][][1.0][0]{$u_{\rm{sonic}}$}
				\psfrag{tw}[cc][][1.0][0]{$\theta_w$}
				\includegraphics[scale=1.4]{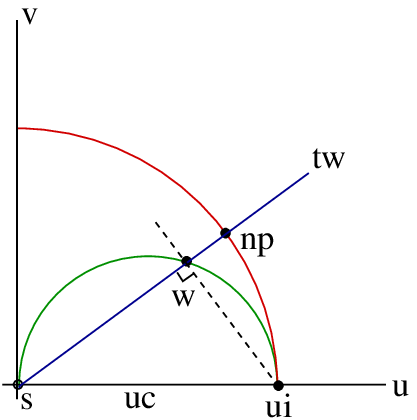}
				\caption{Shock polar for $M_{\infty}=\infty$($\eps=0$)}\label{figure2}
			\end{psfrags}
		\end{figure}

		(Case 1) If $\rho=0$, the last equation in \eqref{RH-spolar-epsilon-zero} becomes $u^2(1+\kappa_w^2)=2B_0$, which implies that the vector ${\bf u}=u(1,\kappa_w)$ lies on the circle of radius $\sqrt{2B_0}$ with the center at the origin (Figure \ref{figure2}). We rewrite the second equation in \eqref{RH-spolar-epsilon-zero} as $u^2(s+\kappa_w)^2=2B_0s^2$ then substitute the expression $u^2=\frac{2B_0}{1+\kappa_w^2}$ into this equation to get
		$$
		\left(\frac{1}{s}\right)^2\kappa_w+2\frac{1}{s}-\kappa_w=0,
		$$
		which yields
		\begin{equation*}
		\frac{1}{s}=\frac{-1\pm \sqrt{1+\kappa_w^2}}{\kappa_w}.
		\end{equation*}
		Then we have $u=\frac{\pm\sqrt{2B_0}}{\sqrt{1+\kappa_w^2}}$. And, this implies that
\begin{equation*}
  u(1,\kappa_w)\cdot(1,-s)=-\sqrt{2B_0}<0,\quad
  \iu(1,0)\cdot (1,-s)=\sqrt{2B_0}>0
  \Rightarrow {\bf u}\cdot{\bm\nu}<0<{\bf u}_{\infty}\cdot{\bm\nu}.
\end{equation*}
Since the signs of ${\bf u}\cdot{\bm\nu}$ and ${\bf u}_{\infty}\cdot{\bm\nu}$ are opposite, we exclude this case from our consideration for the rest of the proof.

		\medskip
		
		(Case 2) If $u=0$, then we obtain from the last two equations in \eqref{RH-spolar-epsilon-zero} that
		\begin{equation}
		\label{downstream-case2}
		(s, \rho)=(0,\mathfrak{h}^{-1}(B_0))=:(0,\rho^{(1)}).
		\end{equation}
Since the function $\mathfrak{h}$ given by \eqref{definition-h} is strictly increasing, we have $\rho^{(1)}>\irho(=0)$. Since $s=0$, the shock solution corresponding to \eqref{downstream-case2} is a normal shock solution, and it corresponds to the origin in the shock polar for $M_{\infty}=\infty$ given in Figure \ref{figure2}.

		\medskip
		(Case 3) If $s=\frac{1}{\kappa_w}$, then the last two equations in \eqref{RH-spolar-epsilon-zero} yield
		\begin{equation}
		\label{downstream-case3}
		(\rho, u)=\left(\mathfrak{h}^{-1}
\left(\frac{B_0\kappa_w^2}{1+\kappa_w^2}\right), \frac{\sqrt{2B_0}}{1+\kappa_w^2}\right):=(\rho^{(2)}, u^{(2)}).
		\end{equation}
Note that $\rho^{(2)}>\irho(=0)$.
The vector ${\bf u}=u(1,\kappa_w)$ satisfies
		\begin{equation*}
		|{\bf u}-(\frac{\sqrt{2B_0}}{2},0)|=\frac{\sqrt{2B_0}}{2}.
		\end{equation*}
This implies that the upper semi-circle $\{|{\bf u}-(\frac{\sqrt{2B_0}}{2},0)|=\frac{\sqrt{2B_0}}{2}:v\ge 0\}$ in Figure \ref{figure2} yields the solution to \eqref{RH-spolar-epsilon-zero} for the case of $s=\frac{1}{\kappa_w}$ for each $\theta_w\in(0, \frac{\pi}{2})$ when $\eps=0$.

\medskip

{\textbf{3}.} We directly differentiate ${\bf R}=(R_1, R_2, R_3)^{T}$, given by \eqref{RH-spolar}, to get
		\begin{equation*}
		D_{(\rho, u, s)}{\bf R}(\rho, u, s, \eps)
		=\begin{pmatrix}
		u(s\kappa_w-1) & \rho(s\kappa_w-1) & \rho u\kappa_w\\
		0 & s+\kappa_w & u-\sqrt{\frac{(\gam-1)B_0}{\frac{\gam-1}{2}+\eps^2}}\\
		\rho^{\gam-2} & u(1+\kappa_w^2) & 0
		\end{pmatrix},
		\end{equation*}
		from which we obtain that
		\begin{equation}
		\label{spolar-nonzero-jacobian}
		\det D_{(\rho, u, s)}{\bf R}(\rho, u, s, 0)=
		\begin{cases}
		(\rho^{(1)})^{\gam-1}\sqrt{\frac{(\gam-1)B_0}{\frac{\gam-1}{2}}}\neq 0, & \mbox{at $(\rho, u, s)=(\rho^{(1)}, 0, 0)$ }, \\
		-(\rho^{(2)})^{\gam-1}u^{(2)}(1+\kappa_w^2)\neq 0, & \mbox{at $(\rho, u, s)=(\rho^{(2)}, u^{(2)}, \frac{1}{\kappa_w})$}.
		\end{cases}
		\end{equation}
Then the implicit function theorem implies that if $\eps$ is sufficiently small depending on $(\gam, B_0, \tw)$, then the system \eqref{RH-spolar} has two solutions $(\rho^{(j)}_{\eps}, u^{(j)}_{\eps}, s^{(j)}_{\eps} )$ for $j=1,2$ with $(\rho^{(1)}_{\eps}, u^{(1)}_{\eps}, s^{(1)}_{\eps} )$ being close to $(\rho^{(1)}, 0, 0)$, and $(\rho^{(2)}_{\eps}, u^{(2)}_{\eps}, s^{(2)}_{\eps} )$ being close to $(\rho^{(2)}, u^{(2)}, \frac{1}{\kappa_w})$, respectively. Since each $(\rho^{(j)}_{\eps}, u^{(j)}_{\eps}, s^{(j)}_{\eps} )$ for $j=1,2$ varies continuously depending on $\eps$, we have $s_{\eps}^{(1)}<s_{\eps}^{(2)}$ for $\eps$ sufficiently small, thus $(\rho^{(1)}_{\eps}, u^{(1)}_{\eps}, s^{(1)}_{\eps} )\neq (\rho^{(2)}_{\eps}, u^{(2)}_{\eps}, s^{(2)}_{\eps} )$ holds.
\medskip

{\textbf{4.}}
By \eqref{definition-incoming-epsilon} and the implicit function theorem, $\irho$ and $\rho^{(j)}_{\eps}$ for $j=1,2$ vary depending continuously  on $\eps$. Therefore, we have $\rho^{(j)}_{\eps}-\irho>0$ for $j=1,2$ and $u_{\eps}^{(2)}>0$ if $\eps$ is sufficiently small because $\rho^{(j)}-\irho>0$ and $u^{(2)}>0$ hold when $\eps=0$. Since $s^{(1)}=0$, we have $s_{\eps}^{(1)}\kappa_w-1<0$ for $\eps$ sufficiently small. Then we solve the equation $R_1(\rho^{(1)}_{\eps}, u^{(1)}_{\eps}, s^{(1)}_{\eps}, \eps)=0$ for $u_{\eps}^{(1)}$, and use $s_{\eps}^{(1)}\kappa_w-1<0$ to obtain that $u_{\eps}^{(1)}>0$ for $\eps>0$ sufficiently small.

To conclude that $\theta_w<\theta_{\rm{det}}$, it remains to check that each $(\rho^{(j)}_{\eps}, u^{(j)}_{\eps}, s^{(j)}_{\eps} )$ for $j=1,2$ satisfies the entropy condition \eqref{condition-admsbl}.
According to \eqref{RH-spolar-Prelim} and \eqref{RH-spolar}, the equations $R_k(\rho^{(j)}_{\eps}, u^{(j)}_{\eps}, s^{(j)}_{\eps}, \eps)=0$ for $k=1,2$ yield
\begin{equation*}
  \rho_{\eps}^{(j)}{\bf u}^{(j)}\cdot {\bm\nu}^{(j)}=\irho {\bf u}_{\infty}\cdot {\bm\nu}^{(j)}\quad\tx{for $j=1,2$}
\end{equation*}
where ${\bf u}^{(j)}$ and ${\bm\nu}^{(j)}$ are given by
\begin{equation*}
 {\bf u}^{(j)}=u_{\eps}^{(j)}(1,\kappa_w),\quad {\bm\nu}^{(j)}:=
 \frac{(1,-s_{\eps}^{(j)})}{\sqrt{1+(s_{\eps}^{(j)})^2}}.
\end{equation*}
Then it follows from $\rho^{(j)}_{\eps}>\irho$ and ${\bf u}_{\infty}\cdot {\bm\nu}^{(j)}>0$ that
\begin{equation}
\label{entp-cond-check}
  0<{\bf u}^{(j)}\cdot {\bm\nu}^{(j)}<{\bf u}_{\infty}\cdot {\bm\nu}^{(j)} \quad\tx{for $j=1,2$}.
\end{equation}
The equation $R_2(\rho^{(1)}_{\eps}, u^{(1)}_{\eps}, s^{(1)}_{\eps}, \eps)=0$ implies that ${\bf u}_{\infty}-{\bf u}^{(j)}$ is parallel to $(1, -s_{\eps}^{(j)})$. Then it follows from \eqref{entp-cond-check} that ${\bm\nu}^{(j)}=\frac{{\bf u}_{\infty}-{\bf u}^{(j)}}{|{\bf u}_{\infty}-{\bf u}^{(j)}|}$ for $j=1,2$. So we conclude that each $(\rho^{(j)}_{\eps}, u^{(j)}_{\eps}, s^{(j)}_{\eps} )$ for $j=1,2$ satisfies the entropy condition \eqref{condition-admsbl}. This completes the proof.

\end{proof}

\begin{remark}
	\label{remark-weakshock-strongshock}
For fixed $\theta_w\in (0, \frac{\pi}{2})$, let the constant $\eps_0>0 $ be from Lemma \ref{lemma-shock-polar1}. For $(\irho, \iu)\in D_{\infty}(\gam, B_0)$, let us set
\begin{equation*}
  \eps:=\frac{1}{M_{\infty}}.
\end{equation*}
For each $\eps\in(0, \eps_0]$, let $(\rho^{(1)}_{\eps}, u^{(1)}_{\eps}, s^{(1)}_{\eps})$ and $(\rho^{(2)}_{\eps}, u^{(2)}_{\eps}, s^{(2)}_{\eps})$ be given from Step 3 in the proof of Lemma \ref{lemma-shock-polar1}. Since $u^{(1)}<u^{(2)}$ when $\eps=0$, we can choose a constant $\eps_1\in(0,\eps_0]$ so that whenever $\eps\in(0,\eps_1]$, we have
\begin{equation*}
  0<u^{(1)}_{\eps}<u^{(2)}_{\eps}<\ui
\end{equation*}
for $\ui$ given by \eqref{definition-incoming-epsilon}. Therefore, $(\rho^{(1)}_{\eps}, u^{(1)}_{\eps}, s^{(1)}_{\eps})$ corresponds to the strong shock, while $(\rho^{(2)}_{\eps}, u^{(2)}_{\eps}, s^{(2)}_{\eps})$ corresponds to the weak shock. Hereafter, we let $(\rho_{\rm st}^{\eps}, u_{\rm st}^{\eps}, s_{\rm st}^{\eps})$ denote $(\rho^{(1)}_{\eps}, u^{(1)}_{\eps}, s^{(1)}_{\eps})$ for each $\eps\in(0, \eps_1]$.

\end{remark}

\subsection{Detached shock past the blunt body $W_b$}
\label{subsection-intro-to-d-shock}

Fix $\gam >1$ and $B_0>0$. Given constants $\theta_w\in(0, \frac{\pi}{2})$ and $h_0>0$, let a function $b:\R\rightarrow \R$ be given by Definition \ref{definition-bluntbody-ftn}. Let $W_b^+$ be given by \eqref{definition-half-blunt-body}, and set
\begin{equation*}
  \R^2_+:=\{{\rx}=(x_1, x_2)\in \R^2: x_2>0\}.
\end{equation*}

It is our goal to construct an entropy solution of \eqref{E-system} in $\R^2_+\setminus W_b^+$ with a shock for an incoming state $(\irho, \iu)\in D_{\infty}(\gam, B_0)$(Fig. \ref{figure-dshock-unbdd}).

\begin{figure}[htp]
	\centering
	\begin{psfrags}
		\psfrag{sh}[cc][][0.8][0]{$\Gam_{{\rm{sh}}}$}
		\psfrag{O}[cc][][0.8][0]{$0$}
		\psfrag{om}[cc][][0.8][0]{$\Om$}
		\psfrag{Uf}[cc][][0.8][0]{$M_{\infty}=\frac{1}{\eps}$}
		\includegraphics[scale=1.0]{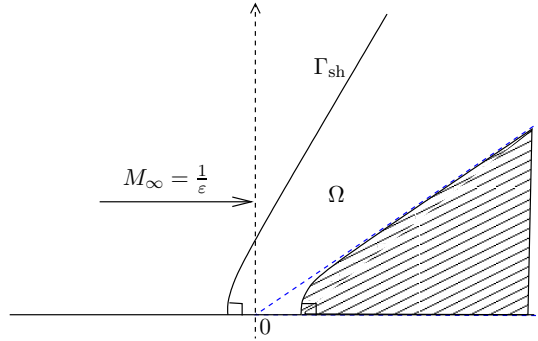}
		\caption{Detached shock past $W_b^+$}\label{figure-dshock-unbdd}
	\end{psfrags}
\end{figure}
\begin{problemL}
  \label{problem-full-system}
  Find a weak solution $(\rho, u_1, u_2)$ to \eqref{E-system} in $\R^2_+\setminus W_b^+$ with a shock
  \begin{equation*}
  \shock=\{x_1=\fshock(x_2):x_2\ge 0\}
  \end{equation*}
  for a $C^1$ function $\fshock: \R_+\rightarrow \R$ so that the solution uniformly converges to a piecewise constant state as $|{\rm x}|\rightarrow \infty$.
\end{problemL}

\begin{lemma}\label{lem:3.3}
For a given incoming state $(\irho, \iu)\in D_{\infty}(\gam, B_0)$, if $(\rho, {\bf u})$ is an entropy solution of \eqref{E-system} in $\R^2_+\setminus W_b^+$ with a shock $\mcl{S}$ given as a $C^1$ curve, then $\mcl{S}$ cannot be attached to the boundary of $W_b^+$.

\begin{proof}
Let $(\rho, {\bf u})$ be a weak solution to \eqref{E-system} in $\R^2_+\setminus W_b^+$ with a shock $\mcl{S}$, and define two domains $\Om^-$ and $\Om^+$ by
\begin{equation*}
  \Om^-:=\{{\rx}\in \R^2_+\setminus W_b^+:
  (\rho, {\bf u})(\rx)=(\irho, \iu, 0)\},\quad
  \Om^+:=(\R^2_+\setminus W_b^+)\setminus \ol{\Om^-}.
\end{equation*}

Suppose that $\mcl{S}$ is attached to the boundary of $W_b^+$ at a point $P_0$.
\smallskip

{\emph{Case 1. $P_0=(b(x_2^*), x_2^*)$ for some $x_2^*>0$}}
\smallskip

Let ${\bm\nu}_0$ be the unit normal vector of $\mcl{S}$ at $P_0$ pointing interior to $\Om^-$ (see Fig.\ref{figure-ashock1}). If ${\bm\nu}_0\not\perp \Gam_b$ at $P_0$, then there exists a point $Q=(b(\hat{x}_2), \hat{x}_2)$ with $\hat{x}_2>0$, and a small constant $r>0$ such that
\begin{equation*}
   (B_r(Q)\setminus W_b^+)\subset \Om^-.
\end{equation*}
\begin{figure}[htp]
	\centering
	\begin{psfrags}
		\psfrag{wb}[cc][][0.8][0]{$W_b^+$}
		\psfrag{us}[cc][][0.8][0]{$\Om^-$}
		\psfrag{b}[cc][][0.8][0]{$B_r(Q)$}
\psfrag{ds}[cc][][0.8][0]{$\Om^+$}
\psfrag{p}[cc][][0.8][0]{$P_0$}
\psfrag{q}[cc][][0.8][0]{$Q$}
		\includegraphics[scale=0.8]{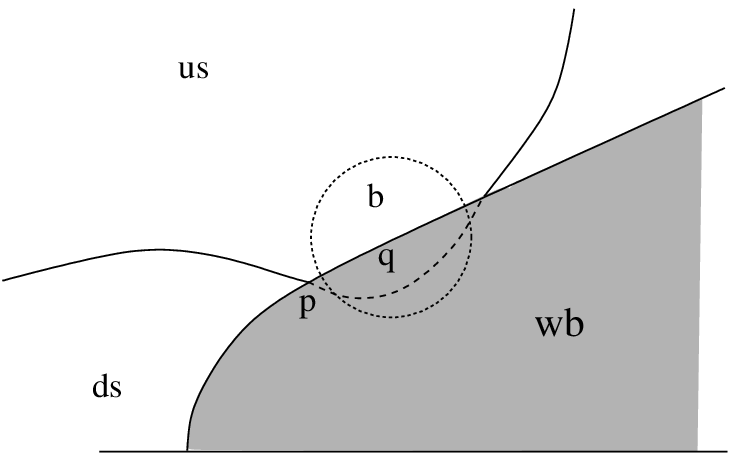}
\caption{}\label{figure-ashock1}
		\end{psfrags}
\end{figure}
Then, by continuation, we have $(\rho, {\bf u})(Q)=(\irho, {\bf u}_{\infty})$.
We define a vector field ${\bf n}(x_2):=(-1, b'(x_2))$ for $x_2>0$.
By Definition \ref{definition-bluntbody-ftn}, ${\bf n}_0:=(-1,b'(\hat{x}_2))$ is a normal vector of $\der W_b$ at $Q$, which points interior to $\R^2_+\setminus W_b$. And, we have ${\bf u}\cdot {\bf n_0}=-\iu\neq 0$ at $Q$ so the slip boundary condition, stated in ($S_5$) of Definition \ref{definition-shock-solution}, does not hold. This is a contradiction.
\smallskip

If ${\bm\nu}_0\perp \Gam_b$ at $P_0$, then we have ${\bm\nu}_0=\frac{{\bf n}_0}{|{\bf n}_0|}$. We compute the value of $(\rho, {\bf u})$ at $P_0$ in the side of $\Om^+$(Fig. \ref{figure-ashock2}) by taking the limit
\begin{equation}
\label{up}
 (\rho, {\bf u})(P_0)= \lim_{\substack{{\rx}\to P_0\\ \rx\in\Om^+}}(\rho, {\bf u})(\rx)=:(\rho^+, {\bf u}^+).
\end{equation}
\begin{figure}[htp]
	\centering
	\begin{psfrags}
		\psfrag{wb}[cc][][0.8][0]{$W_b^+$}
		\psfrag{us}[cc][][0.8][0]{}
		\psfrag{om}[cc][][0.8][0]{$\Om^-$}
\psfrag{ds}[cc][][0.8][0]{}
\psfrag{p}[cc][][0.8][0]{$P_0$}
\psfrag{nu}[cc][][0.8][0]{${\bm \nu}_0$}
\psfrag{op}[cc][][0.8][0]{$\Om^+$}
		\includegraphics[scale=0.8]{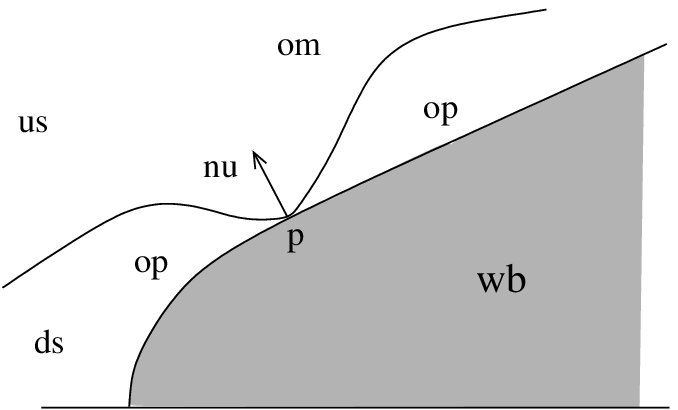}
\caption{}\label{figure-ashock2}
		
	\end{psfrags}
\end{figure}
By the slip boundary condition on $\Gam_b$, and $C^3$ regularity of $b$, we have
\begin{equation*}
  {\bf u}^+\cdot{\bm\nu_0}=\lim_{x_2\to \hat{x}_2}{\bf u}(b(x_2),x_2)\cdot \frac{{\bf n}(x_2)}{|{\bf n}(x_2)|}=0,
\end{equation*}
which implies that $\rho^+{\bf u}^+\cdot {\bm\nu}_0=0$. Then the Rankine-Hugoniot condition \eqref{RH-potential} does not hold at $P_0$ because $\irho {\bf u}_{\infty}\cdot {\bm\nu}_0\neq 0$. This is a contradiction. Therefore, $\mcl{S}$ cannot be attached to $\Gam_b$ away from the point $(b_0,0)$.
\smallskip

{\emph{Case 2. $P_0=(b_0, 0)$}}

Due to \eqref{b-property1} stated in Remark \ref{remark-b-function}, we have
$\displaystyle{\lim_{x_2\to 0+}{\bf n}(x_2)=(-1,0)}$. Let ${\bf u}^+$ be given by \eqref{up}. By continuity of ${\bf u}$ in $\Om^+$(Fig. \ref{figure-ashock3}) up to its boundary, and the slip boundary condition on $\Gam_b$, we obtain that
\begin{equation}
\label{up-hcomp}
  {\bf u}^+\cdot (-1,0)=\lim_{x_2\to 0+} {\bf u}(b(x_2), x_2)\cdot \frac{{\bf  n}(x_2)}{|{\bf n}(x_2)|}=0.
\end{equation}
\begin{figure}[htp]
	\centering
	\begin{psfrags}
		\psfrag{wb}[cc][][0.8][0]{$W_b$}
		\psfrag{us}[cc][][0.8][0]{$(\rho, {\bf u})=(\irho, {\bf u}_{\infty})$}
		\psfrag{om}[cc][][0.8][0]{$\Om^-$}
\psfrag{ds}[cc][][0.8][0]{$\phantom{aaa}(\rho, {\bf u})\neq(\irho, {\bf u}_{\infty})$}
\psfrag{p}[cc][][0.8][0]{$P_0$}
\psfrag{nu}[cc][][0.8][0]{${\bm \nu}_0$}
\psfrag{op}[cc][][0.8][0]{$\Om^+$}
		\includegraphics[scale=0.8]{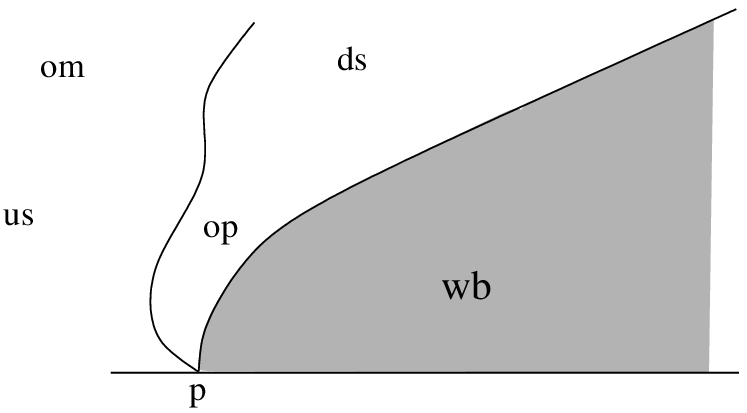}
\caption{}\label{figure-ashock3}
		
	\end{psfrags}
\end{figure}
This implies that if the shock $\mcl{S}$ is attached to $W_b$ at $P_0$, then the horizontal component of ${\bf u}^+$ is 0. By Lemma \ref{lemma-shock-polar-general}, ${\bf u}^+$ must lie on either $v=\mathfrak{f}_{\rm polar}(u)$ or $v=-\mathfrak{f}_{\rm polar}(u)$. Therefore, ${\bf u}^+$ must have a strictly positive horizontal component but this contradicts to \eqref{up-hcomp}. So we conclude that $\mcl{S}$ cannot be attached to $\Gam_b$ at the point $(b_0,0)$. This completes the proof.

\end{proof}
\end{lemma}

According to Lemma \ref{lem:3.3}, any shock solution past the blunt body $W_b^+$ in $\R^2_+$ contains a {\emph{detached shock}} in the sense that $\shock$ is completely separated from the blunt body $W_b^+$. And, this yields the following result on a solution to Problem \ref{problem-full-system}.

\begin{corollary}
  \label{corollary-f-b}
  If $(\rho, u_1, u_2)$ is a solution to Problem \ref{problem-full-system} with a shock
  $\shock=\{x_1=\fshock(x_2):x_2\ge 0\}$, then we have
  \begin{equation}
  \label{location-of-shock}
    f(x_2)<b(x_2)\quad\tx{for all $x_2\ge 0$}.
  \end{equation}
\end{corollary}

In Problem \ref{problem-full-system}, we seek a shock solution past the blunt body $W_b^+$ with piecewise constant asymptotic state at far field. By \eqref{slip-BC-original}, Lemmas \ref{lemma-shock-polar-general} and \ref{lemma-shock-polar1}, and Remark \ref{remark-weakshock-strongshock}, if $(\irho, \ui)\in D_{\infty}(\gam, B_0)$ satisfies $M_{\infty}\ge \frac{1}{\eps_1}$, then we expect either the strong shock solution or the weak shock solution to be asymptotic states of a solution to Problem \ref{problem-full-system}. In this paper, we seek a shock solution to Problem \ref{problem-full-system} with the asymptotic state given by the strong shock solution.

\begin{definition}
\label{notation-strongshock}
Hereafter, the half-wedge angle $\theta_w\in (0, \frac{\pi}{2})$ is fixed in \eqref{definition-Wb} and \eqref{definition-half-blunt-body}. Let $\eps_1\in(0,1)$ be from Remark \ref{remark-weakshock-strongshock}. For each $\eps\in(0,\eps_1]$, let $(\rho_{\rm st}^{\eps}, u_{\rm st}^{\eps}, s_{\rm st}^{\eps})$ be as in Remark \ref{remark-weakshock-strongshock}. We define
\begin{equation*}
  {\bf u}_{\rm st}^{\eps}:=u_{\rm st}(1,\kappa_w)\quad\tx{for $\kappa_w=\tan \theta_w$}.
\end{equation*}
The vector ${\bf u}_{\rm st}^{\eps}$ represents the downstream velocity of the strong shock solution past the wedge $W_0^+$ given by \eqref{definition-upper-wedge} for the incoming supersonic state $(\irho, \ui)\in D_{\infty}(\gam, B_0)$ with $M_{\infty}=\frac{1}{\eps}$. The constant $s_{\rm st}^{\eps}$ is the slope($=\frac{dx_1}{dx_2}$) of the strong shock, and $\displaystyle{\rho_{\rm st}^{\eps}=\mathfrak{h}^{-1}(B_0-\frac 12|{\bf u}_{\rm st}^{\eps}|^2)}$ represents the downstream density of the strong shock solution.

\end{definition}

\begin{lemma}
	\label{lemma-shock-polar2}
	Let $\eps_1>0$ be from Remark \ref{remark-weakshock-strongshock}.  There exists a small constant $\eps_2\in(0, \eps_1]$ depending only on $(\gam, B_0, \theta_w)$ and a continuous function $\mathfrak{q}_{\gam}:(0, \eps_2]\rightarrow (0, \infty)$ with satisfying
\begin{equation}
\label{modulus-of-rhs}
  \lim_{\eps\to 0+}\mathfrak{q}_{\gam}(\eps)=0
\end{equation}
so that whenever $M_{\infty}\ge \frac{1}{\eps_2}$, we have
	\begin{equation}
\label{estimate-bgd}
	|(\rho^{\eps}_{\rm{st}}, u^{\eps}_{\rm{st}}, s^{\eps}_{\rm{st}})-(\rho^0_{\rm{st}}, 0, 0)|\le \mathfrak{q}_{\gam}(\eps),
	\end{equation}
and
\begin{equation}\label{subsonicity}
\left(\rho_{\rm st}^{\eps}\right)^{\gam-1}-(u_{\rm st}^{\eps})^2(1+\kappa_w^2)\ge \frac 12 \left(\rho^0_{\rm{st}}\right)^{\gam-1}>0.
\end{equation}

\end{lemma}
	
	\begin{proof}
		For ${\bf R}(\rho, u, s;\eps)=(R_1, R_2, R_3)^T(\rho, u, s;\eps)$ defined by \eqref{RH-spolar}, we have
		\begin{equation}
		\label{R-difference-expression}
		{\bf R}(\rho_{\rm{st}}^{\eps}, u_{\rm{st}}^{\eps}, s_{\rm{st}}^{\eps};\eps)
		-{\bf R}(\rho_{\rm{st}}^0, 0, 0;\eps)
		={\bf R}(\rho_{\rm{st}}^0, 0, 0;0)-
		{\bf R}(\rho_{\rm{st}}^0, 0, 0;\eps).
		\end{equation}
For ${\mathbb M}\in \R^{3\times 3}$, define a norm $\|{\mathbb M}\|$ by
\begin{equation*}
  \|{\mathbb M}\|:=\sup_{{\,\bm\xi\in \R^3,}\atop{\!|\bm\xi|=1}} |{\mathbb M}{\bm\xi}|.
\end{equation*}		
Since ${\bf R}$ is smooth with respect to $(\rho, u, s, \eps)$, one can find a small constant $\sigma_1>0$ depending on $(\gam, B_0, \theta_w)$ so that if
		\begin{equation}\label{condition-R-invertibility}
		|(\rho, u, s, \eps)-(\rho_{\rm{st}}^0, 0, 0, 0)|\le \sigma_1,
		\end{equation}
		then it follows from \eqref{spolar-nonzero-jacobian} that $D_{(\rho, u, s)}{\bf R}(\rho, u, s;\eps)$ is invertible and satisfies
		\begin{equation}\label{norm-of-inverse}
		\|\left(D_{(\rho, u, s)}{\bf R}(\rho, u, s;\eps)\right)^{-1}\|\le 2
		\|\left(D_{(\rho, u, s)}{\bf R}(\rho_{\rm{st}}^0, 0, 0;0)\right)^{-1}\|=:\mu_0.
		\end{equation}
	Step 3 in the proof of Lemma \ref{lemma-shock-polar1} implies that $(\rho_{\rm{st}}^{\eps}, u_{\rm{st}}^{\eps}, s_{\rm{st}}^{\eps})$ is $C^1$ with respect to $\eps\in(0, \eps_1)$, so there exists a small constant $\eps'\in(0,\eps_1)$ depending only on $(\gam, B_0, \theta_w)$ so that $(\rho, u, s, \eps)=(\rho_{\rm{st}}^{\eps}, u_{\rm{st}}^{\eps}, s_{\rm{st}}^{\eps}, \eps)$ for $\eps\in(0,\eps']$ satisfies \eqref{condition-R-invertibility}.

		Note that we have
		\begin{equation}\label{estimate-R-rhs}
\begin{split}
		|{\bf R}(\rho_{\rm{st}}^0, 0, 0;0)-
		{\bf R}(\rho_{\rm{st}}^0, 0, 0;\eps)|&=
|R_1(\rho_{\rm{st}}^0, 0, 0;0)-
		R_1(\rho_{\rm{st}}^0, 0, 0;\eps)|\\
&=\mathfrak{h}^{-1}
\left(\frac{2B_0}{(\gam-1)+2\eps^2}\eps^2\right)
\sqrt{\frac{(\gam-1)B_0}{\frac{\gam-1}{2}+\eps^2}}\\
&=\left(\frac{2(\gam-1)B_0}{(\gam-1)+2\eps^2}\eps^2\right)^{\frac{1}{\gam-1}}
\sqrt{\frac{(\gam-1)B_0}{\frac{\gam-1}{2}+\eps^2}}
\end{split}
		\end{equation}
	for all $\eps\in(0,\eps_1]$.	
So we finally conclude from \eqref{R-difference-expression}, \eqref{norm-of-inverse} and \eqref{estimate-R-rhs} that if we choose $\eps_2$ as $\eps_2=\eps'$, and if $\eps\in[0, \eps_2]$, then we have
		\begin{equation}
\label{definition-q-gamma}
\begin{split}
		&|(\rho^{\eps}_{\rm{st}}, u^{\eps}_{\rm{st}}, s^{\eps}_{\rm{st}})-(\rho^0_{\rm{st}}, 0, 0)|\\
&\le
		\mu_0 \left(\frac{2(\gam-1)B_0}{(\gam-1)+2\eps^2}\eps^2\right)^{\frac{1}{\gam-1}}
\sqrt{\frac{(\gam-1)B_0}{\frac{\gam-1}{2}+\eps^2}}
=:\mathfrak{q}_{\gam}(\eps)(\approx \eps^{\frac{2}{\gam-1}}).
\end{split}
		\end{equation}
The function $\mathfrak{q}_{\gam}(\eps)$ is continuous with respect to both $\eps>0$ and $\gam > 1$.
And, \eqref{definition-q-gamma} implies that the function $\mathfrak{q}_{\gam}$ satisfies \eqref{modulus-of-rhs}. Finally, one can further reduce $\eps_2$ depending only on $(\gam, B_0, \tw)$ so that \eqref{subsonicity} follows from \eqref{modulus-of-rhs} and \eqref{estimate-bgd} because $u_{\rm st}^0=0$.
	\end{proof}

\subsection{The main theorem}
\label{subsection-main-theorem1}

To state our main theorem, we first define H\"{o}lder norms with weight at infinity.
\begin{definition}
	\label{definition-norms-w-decay}
	Fix constants $m\in \mathbb{Z}^+$, $\mu\in \R$, and $\alp\in(0,1)$.
	\smallskip
	
(i) Let $I\subset (0, \infty)$ be an open interval. For a function $f:\ol{I}\rightarrow \R$, define
	\begin{equation*}
	\begin{split}
	&\|f\|_{m, I}^{(\mu)} :=\sum_{j=0}^m \sup_{x_2\in I} (1+x_2)^{j+\mu} \left|\frac{d^j}{dx_2^j} f(x_2)\right|\\
	&[f]_{m,\alp, I}^{(\mu)} :=
	\sup_{x_2\neq x'_2\in I} (1+\min\{x_2,x'_2\})^{m+\alp+\mu}
	\frac{|\frac{d^m}{dx_2^m}f(x_2)-\frac{d^m}{dx_2^m}f(x'_2)|}{|x_2-x'_2|^{\alp}}\\
	&\|f\|_{m,\alp, I}^{(\mu)}:=
	\|f\|_{m, I}^{(\mu)} + [f]_{m,\alp, I}^{(\mu)} .
	\end{split}
	\end{equation*}
\smallskip

(ii)	Let $D\subset \R^2_+$ be an open and connected domain. For points $\rx, \rx'\in D$, let $x_2$, $x_2'$ denote the $x_2$-coordinates of $\rx,\rx'$, respectively. For a function $\phi: \ol{D}\rightarrow \R$, define
	\begin{equation*}
	\begin{split}
	&\|\phi\|_{m, D}^{(\mu)}  :=
	\sum_{j=0}^m \sup_{{\bf x}\in D} (1+x_2)^{j+\mu}
	\sum_{0\le l\le j}|\der_{x_1}^l\der_{x_2}^{j-l} \phi({\bf x})|\\
	&[\phi]_{m,\alp, D}^{(\mu)} :=
	\sup_{{\bf x}\neq {\bf x}'\in D} (1+\min\{x_2,x_2'\})^{m+\alp+\mu}
	\sum_{0\le l\le m}\frac{|\der_{x_1}^l\der_{x_2}^{m-l}\phi({\bf x})-\der_{x_1}^l\der_{x_2}^{m-l}\phi({\bf x}')|}{|{\bf x}-{\bf x}'|^{\alp}}\\
	&\|\phi\|_{m,\alp, D}^{(\mu)}:=\|\phi\|_{m, D}^{(\mu)}  +[\phi]_{m,\alp, D}^{(\mu)} .
	\end{split}
	\end{equation*}
\end{definition}

\begin{note}
For the rest of the paper, the blunt body $W_{b}^+$ is fixed thus the constants $(\tw, h_0, b_0, m_b)$ are fixed. Here, the constant $m_b$ is from  \eqref{b-property2}. So we will not specify a dependence of any estimate constant on $(\tw, h_0, b_0, m_b)$ hereafter except for the H\"{o}lder exponent $\alp$ for the regularity of solutions to Problem \ref{problem-full-system} because the choice of $\alp$ heavily depends on the half-wedge angle $\tw$. So we will specify the dependence of $\alp$ on $\tw$.
\end{note}

\begin{theorem}
	\label{main-theorem-physical-var}
	Fix $\gam > 1$ and $B_0>0$. And, fix $\beta\in(0,1)$.
\begin{itemize}
\item[(a)]{\emph{(The existence of detached shock solutions)}}
	For a fixed constant $d_0>0$, there exists a small constant $\bar{\eps}>0$ depending on $(\gam, B_0, d_0)$ so that whenever the incoming supersonic state $(\irho, \ui)\in D_{\infty}(\gam, B_0)$ satisfies $M_{\infty}=\frac{1}{\eps}$ for $\eps\in(0, \bar{\eps}]$,
	the system \eqref{E-system} has an entropy solution $(\rho, {\bf u})$ in $\R^2_+\setminus W_b^+$ with a shock $\shock=\{(\fshock(x_2), x_2):x_2\ge 0\}$ in the sense of Definitions \ref{definition-shock-admsble} for the incoming state $(\irho, \iu,0)$. And, the solution satisfies the following properties:
	\begin{itemize}
		\item[(i)]
		$\displaystyle{\fshock(0)=b_0-d_0;}$

\item[(ii)] There exists a constant $\delta>0$ depending only on $(\gam, B_0, d_0)$ such that
		\begin{equation*}
		b(x_2)-\fshock(x_2)\ge \delta\quad\tx{for all $x_2\ge 0$};
		\end{equation*}

		\item[(iii)] Setting as $\displaystyle{
\Om_{\fshock}:=\{{\bf x}=(x_1,x_2)\in \R^2_+\setminus W_b^+: x_1 >\fshock(x_2),\,\,x_2>0\}}$, we have
		\begin{equation*}
		\lim_{{|\bf x|\rightarrow \infty}\atop {{\bf x}\in \Om_{\fshock}}}
		|(\rho, {\bf u})({\bf x})-(\rho^{\eps}_{\rm{st}}, {\bf u}^{\eps}_{\rm{st}})|=0,\quad
		\tx{and}\quad
		\lim_{x_2\to \infty} |\fshock'(x_2)-s^{\eps}_{\rm{st}}|=0;
		\end{equation*}

		\item[(iv)] There exists a constant $\hat{\alp}\in(0,1)$ depending only on $\tw$, and  a constant $C>0$ depending only on $(\gam, B_0, d_0)$ such that
		\begin{equation}
		\begin{split}
		\label{apriori-estimate-physical-var}
		&\|\fshock-f_0\|_{2,\hat{\alp}, \R^+}^{(-\beta)}+
		\|{\bf u}-{\bf u}_{\rm st}^{\eps}\|_{1,\hat{\alp},\Om_{\fshock}}^{(1-\beta)}\le C\mathfrak{q}_{\gam}(\eps)
		\end{split}
		\end{equation}
for the functions $f_0$ defined by
\begin{equation}
\label{definition-background-f}
f_0(x_2):=s_{\rm{st}}^{\eps}x_2+b_0-d_0.
\end{equation}
Here, $\mathfrak{q}_{\gam}(\eps)$ is from Lemma \ref{lemma-shock-polar2};

\item[(v)] There exists a constant $\sigma\in(0,1)$ depending only on $(\gam, B_0, d_0)$ so that the Mach number $M(\rho, {\bf u})$ defined by \eqref{definition-soundsp-mach} satisfies the inequality
    \begin{equation*}
      M(\rho, {\bf u})\le 1-\sigma\quad\tx{in $\ol{\Om_{\fshock}}$}.
    \end{equation*}
    In other words, the flow in $\Om_{\fshock}$ is subsonic, thus $\shock$ is a transonic shock.
	
\end{itemize}

\item[(b)]{\emph{(Convexity of detached shocks)}} For a fixed constant $d_0>0$, let $\bar{\eps}$ be from Theorem \ref{main-theorem-physical-var}(a). Then, there exists a constant
    $\hat{\eps}\in(0, \bar{\eps}]$ depending on $(\gam, B_0, d_0)$ so that if the incoming supersonic state $(\irho, \ui)\in D_{\infty}(\gam, B_0)$ satisfies $M_{\infty}=\frac{1}{\eps}$ for $\eps\in(0, \hat{\eps}]$, then the system \eqref{E-system} has an entropy solution $(\rho, {\bf u})$ in $\R^2_+\setminus W_b^+$ with a shock $\shock=\{(\fshock(x_2), x_2):x_2\ge 0\}$ that satisfies
    \begin{equation*}
      \fshock''(x_2)\ge 0\quad\tx{for $x_2> 0$}
    \end{equation*}
    as well as all the properties (i)--(v) stated in Theorem \ref{main-theorem-physical-var}(a).

\end{itemize}
\end{theorem}

\begin{remark}
\label{remark-main-theorem} Prior to proving Theorem \ref{main-theorem-physical-var}, we first discuss two issues related to the uniqueness of detached shock solutions past $W_b^+$.
\medskip

(1)
Note that Theorem \ref{main-theorem-physical-var} does not guarantee the uniqueness of a detached shock solution for a fixed detached distance $d_0>0$. This is mainly because of the property (iv) stated in Theorem \ref{main-theorem-physical-var} (a). More precisely, the estimate \eqref{apriori-estimate-physical-var} given in the property (iv) of Theorem \ref{main-theorem-physical-var}(a) only provides the asymptotic limit of the tangential slope of $\shock$ at far-fields but the limit of $|\fshock-f_0|$ at $x_2=\infty$ may not exist.
\medskip

(2)
According to Theorem \ref{main-theorem-physical-var} (a), the Mach number $M_{\infty}$ of the incoming supersonic state is required to be sufficiently large depending on {\emph{the detached distance}} $d_0>0$ to have an entropy solution of \eqref{E-system} in $\R^2_+\setminus W_b^+$ with a shock $\shock$ of the detached distance $d_0$ from the tip $(b_0,0)$ of $W_b^+$. Actually, it is possible to  adjust the proof of Theorem \ref{main-theorem-physical-var}, given throughout this paper,  to show that, for a fixed constant $\underline{d}>0$, there exists a small constant $\bar{\eps}^*>0$ depending only on $(\gam, B_0, \underline{d})$ so that whenever the incoming supersonic state $(\irho, \ui)\in D_{\infty}(\gam, B_0)$ satisfies $M_{\infty}=\frac{1}{\eps}$ for $\eps\in(0, \bar{\eps}^*]$ and a detached distance $d_0$ is given on the interval $[\underline{d}, \infty)$, then the system \eqref{E-system} has at least one entropy solution $(\rho, {\bf u})$ in $\R^2_+\setminus W_b^+$ with a shock $\shock$ of the detached distance $d_0$ from the tip $(b_0,0)$ of $W_b^+$. This observation yields  {\emph{a family of detached shock solutions}} of \eqref{E-system} in $\R^2_+\setminus W_b^+$ for a fixed incoming supersonic state. In other words, if we fix a constant $\underline{d}>0$, and if $d_0\ge \underline{d}$ holds, then the system \eqref{E-system} has at least one entropy solution $(\rho, {\bf u})$ in $\R^2_+\setminus W_b^+$ with a shock $\shock$ of the detached distance $d_0$ from the tip $(b_0,0)$ of $W_b^+$ provided that $M_{\infty}$ is sufficiently large depending on $\underline{d}$. So the following questions are naturally raised:
\begin{itemize}
\item[(i)]
Does this family of detached shock solution contains a physically valid detached shock solution? In other words, does there exist a detached shock solution that satisfies all the properties (i)--(v) stated in Theorem \ref{main-theorem-physical-var}, and that becomes a long-time asymptotic limit of a weak solution to unsteady Euler system of irrotational flow?
\item[(ii)]
If so, what is the value of the detached distance $d_0$ for a physically valid detached shock solution?
\end{itemize}
We leave these two questions as open problems to be investigated in the future.
\end{remark}

\section{Stream function formulation}
\label{section-stream-function}

If $(\rho, u_1, u_2)$ is a $C^1$ solution to \eqref{E-system} in a domain, then the first equation in \eqref{E-system} implies that there exists a $C^2$ function $\psi$ to satisfy
\begin{equation}
\label{definition-stream-function}
\nabla^{\perp}\psi=(\rho u_1,\rho u_2)\quad\tx{for $\nabla^{\perp}\psi=(\psi_{x_2}, -\psi_{x_1})$.}
\end{equation}
 Such a function $\psi$ is called a {\emph{stream function}} in the sense that $\psi$ is a constant along each integral curve of the momentum density vector field $\rho{\bf u}=\rho(u_1,u_2)$. With using the function $\psi$, we rewrite the rest of \eqref{E-system} as
\begin{equation}
\label{E-system-str-form}
\begin{split}
&\Div\left(\frac{\nabla\psi}{\rho}\right)=0,\\
&G(\rho, |\nabla\psi|^2):=\frac 12 \frac{|\nabla\psi|^2}{\rho^2}+\mathfrak{h}(\rho)=B_{0}.
\end{split}
\end{equation}
Next, we solve the equation $G(\rho, |\nabla\psi|^2)=B_0$ for $\rho$ to further reduce \eqref{E-system-str-form} into a quasi-linear second order equation for $\psi$.
For $\rho>0$, the equation $G(\rho, |\nabla\psi|^2)=B_0$ is equivalent to
\begin{equation*}
  \rho^2(B_0-\mathfrak{h}(\rho))=\frac 12|\nabla\psi|^2.
\end{equation*}
Set
\begin{equation}\label{definition-H-function}
  H(\rho):=\rho^2(B_0-\mathfrak{h}(\rho)),
\end{equation}
and consider the equation of
\begin{equation}
\label{equation-rho-q}
  H(\rho)=\frac 12|{\bf q}|^2\quad\tx{for ${\bf q}\in \R^2$.}
\end{equation}
A direct computation shows that, for $\gam>1$, the function $H$ satisfies the following properties:
\begin{itemize}
\item[$(H_1)$] $H(0)=0$, $\displaystyle{\lim_{\rho\to \infty} H(\rho)=-\infty}$;

\item [$(H_2)$] We have
\begin{equation}
\label{definition-rho-sonic}
H'(\rho)\begin{cases}
>0\quad &\mbox{if $\rho<\rhosonic$}\\
=0\quad &\mbox{if $\rho=\rhosonic$}\\
<0\quad &\mbox{if $\rho>\rhosonic$}
\end{cases}\quad\tx{for}\quad \rhosonic:=
\left(\frac{2(\gam-1)B_0}{\gam+1}\right)^{\frac{1}{\gam-1}},
\end{equation}
thus the maximum of $H$ is attained at $\rho=\rhosonic$ with
	\begin{equation*}
	H(\rhosonic)=\rhosonic^2\frac{(\gam-1)B_0}{\gam+1};
	\end{equation*}

\item[$(H_3)$] For $\rho_{\rm{max}}=
    \left((\gam-1)B_0\right)^{\frac{1}{\gam-1}}\in (\rhosonic, \infty)$, we have
	\begin{equation}
	\label{rho-max}
	H(\rho)\begin{cases}
	\ge 0 & \mbox{if } \rho\le \rho_{\rm{max}}, \\
	<0 & \mbox{if } \rho>\rho_{\rm{max}}.
	\end{cases}
	\end{equation}
\end{itemize}
The properties $(H_1)$--$(H_3)$ imply that, for each constant $\zeta\in [0,H(\rhosonic))$, the equation $H(\rho)=\zeta$ has exactly two distinct solutions $\rho_{-}(\zeta)$ and $\rho_+(\zeta)$ with
\begin{equation*}
0\le \rho_{-}(\zeta)<\rhosonic<\rho_{+}(\zeta)\le \rho_{\max}.
\end{equation*}
Furthermore, as functions of $\zeta\in [0,H(\rhosonic))$, $\rho_{\pm}(\zeta)$ are smooth. This can be checked by the implicit function theorem.

If we have $H(\rho)=\frac 12|\nabla\psi|^2$ for a stream function $\psi$ given by \eqref{definition-stream-function}, then the corresponding Mach number $M$ defined by \eqref{definition-soundsp-mach} is represented as a function of $\rho$ as follows:
\begin{equation}
\label{definition-mach-rho}
  M(\rho)=\sqrt{\frac{2H(\rho)}{\rho^{\gam+1}}}.
\end{equation}
A direct computation yields that $M(\rhosonic)=1$ and $M'(\rho)=-\frac{2}{\rho^{\gam}}(\gam-1)B_0<0$ for $\rho\in(0,\rho_{\max}]$, and this implies that
\begin{equation*}
  M(\rho)\begin{cases}
  >1\quad &\mbox{for $\rho<\rhosonic$},\\
  <1\quad &\mbox{for $\rho>\rhosonic$}.
  \end{cases}
\end{equation*}
 According to statement (v) of Theorem \ref{main-theorem-physical-var}(a), we seek a shock solution past $W_b^+$ with the Mach number $M<1$ behind the shock. So we use the equation ${\rm div}\left(\frac{\nabla\psi}{\rho_+\left(\frac 12|\nabla\psi|^2\right)}\right)=0$ to determine the downstream state behind a detached shock.

\begin{lemma}
  \label{lemma-rho-expression}
Given constants $\gam >1$ and $B_0>0$, let $(H, \rhosonic, \rho_{\max})$ be given by \eqref{definition-H-function}, \eqref{definition-rho-sonic} and \eqref{rho-max}, respectively. Then, there exists a smooth function $\hat{\rho}: [0, 2H(\rhosonic))\rightarrow (\rhosonic, \rho_{\max}]$ so that if $\psi$ satisfies
\begin{align}
&|\nabla\psi|^2<2H(\rhosonic),\notag\\
\label{equation-psi}
&
  {\rm div}\left(\frac{\nabla\psi}{\hat{\rho}(|\nabla\psi|^2)}\right)=0,
\end{align}
then $(\rho, {\bf u})=\left(\hat{\rho}(|\nabla\psi|^2), \frac{\nabla^{\perp}\psi}{\hat{\rho}(|\nabla\psi|^2)}\right)$ is a solution to the system \eqref{E-system}. Furthermore, the state corresponding to $(\rho, {\bf u})$ is subsonic, that is, we have
\begin{equation}\label{subsonicity-str-form}
  |{\bf u}|^2<\rho^{\gam-1}.
\end{equation}

\begin{proof}
The proof is directly given if we define $\hat{\rho}$ by
\begin{equation}
\label{definition-rho-hat}
  \hat{\rho}(\zeta)=\rho_+\left(\frac{\zeta}{2}\right)
  \quad\tx{for $\zeta<2H(\rhosonic)$.}
\end{equation}

\end{proof}
\end{lemma}

The non-divergence form of the equation \eqref{equation-psi} is
\begin{equation}\label{equation-psi-nondiv}
  (c^2-\frac{\psi_{x_2}^2}{\rho^2})\psi_{x_{1}x_1}
  +2\frac{\psi_{x_1}\psi_{x_2}}{\rho^2}\psi_{x_1x_2}
  +(c^2-\frac{\psi_{x_1}^2}{\rho^2})\psi_{x_2x_2}=0
\end{equation}
for $(c^2,\rho)=(\hat{\rho}^{\gam-1}, \hat{\rho})(|\nabla\psi|^2)$. It follows from \eqref{subsonicity-str-form} that the equation \eqref{equation-psi-nondiv} is elliptic.

\medskip

For the rest of the paper, we fix constants $\gam> 1$ and $B_0>0$. And, we assume that $(\irho, \ui)\in D_{\infty}(\gam, B_0)$ satisfies $M_{\infty}=\frac{1}{\eps}$ for $\eps\le \eps_2$, where $\eps_2$ is from Lemma \ref{lemma-shock-polar2}, unless otherwise specified.

Let us set
\begin{equation}\label{definition-ipsi}
\ipsi(x_1,x_2):=\irho \iu x_2.
\end{equation}
For a given function $f:\R_+\rightarrow \R$ satisfying \eqref{location-of-shock}, define
\begin{equation}
\label{definition-Om-f}
\Om_f:=\{{\bf x}=(x_1,x_2)\in \R^2: f(x_2)<x_1<b(x_2),\,\,x_2>0\}.
\end{equation}
We will prove Theorem \ref{main-theorem-physical-var}(a) by solving the following free boundary problem for a stream function $\psi$ in $\R^2_+\setminus W_b^+$.
\begin{problemL}
  \label{problem-fbp-psi}
  Fix a constant $d_0>0$.
  Find a function $\fshock\in C^1_{\rm{loc}}(\R_+)$ with satisfying \eqref{location-of-shock} and a function $\psi\in C^1_{\rm{loc}}(\ol{\Om_{\fshock}})\cap C^2_{\rm{loc}}(\Om_{\fshock})$ so that the following properties hold:
  \begin{itemize}
    \item [(i)]
    \begin{equation*}
      |\nabla\psi|^2<2H(\rhosonic) \quad\tx{in}\quad \ol{\Om_{\fshock}}
    \end{equation*}
    \item [(ii)] (Equation for $\psi$) For $\hat{\rho}$ from Lemma \ref{lemma-rho-expression},
    \begin{equation*}
     {\rm{div}}\left(\frac{\nabla\psi}{\hat{\rho}(|\nabla\psi|^2)}\right)=0\quad\tx{in}\quad {\Om_{\fshock}}
    \end{equation*}
    \item [(iii)] (Boundary conditions for $\psi$) Define
    \begin{equation*}
      \begin{split}
          & \shock:=\{(\fshock(x_2),x_2):x_2\ge 0\}, \quad \Gam_{\rm{sym}}:=\{(x_1,0):\fshock(0)<x_1<b(0)\}, \\
           & \Gam_b:=\{(b(x_2),x_2):x_2\ge 0\}.
      \end{split}
    \end{equation*}
 Then   $\psi$ satisfies the following boundary conditions:
    \begin{equation}
    \label{BCs-psi}
      \begin{split}
         \psi=\ipsi & \quad\tx{on $\shock$}, \\
          \psi=0 & \quad\tx{on $\Gam_{\rm{sym}}\cup \Gam_b$}.
      \end{split}
    \end{equation}
    (Asymptotic boundary condition) In addition, $\psi$ satisfies
    \begin{equation}
    \label{Asymp-BCs-psi}
      \lim_{{|\bf x|\rightarrow \infty}\atop {{\bf x}\in \Om_{\fshock}}} |\nabla^{\perp}\psi(\rx)-\rho_{\rm st}^{\eps}{\bf u}_{\rm st}^{\eps}|=0.
    \end{equation}

    \item [(iv)] (Free boundary condition) 
    \begin{equation}
    \label{equation-free-bdry}
      \begin{split}
         &\fshock'(x_2)
         =\frac{\left(\psi_{x_1}/\hat{\rho}(|\nabla\psi|^2)\right)(\fshock(x_2),x_2)}
         {\left(\psi_{x_2}/\hat{\rho}(|\nabla\psi|^2)\right)(\fshock(x_2),x_2)-\iu} \quad\tx{for all $x_2>0$}, \\
          &\fshock(0)=b_0-d_0.
      \end{split}
    \end{equation}

  \end{itemize}
\end{problemL}

\begin{definition}[Detached distance of a shock $\shock$ from the blunt body $W_b^+$]
In Problem \ref{problem-fbp-psi}, set
\begin{equation}
\label{definition-P0}
\rb:=(b_0,0),\quad \tx{and}\quad \lb:=(b_0-d_0, 0).
\end{equation}
We call $d_0=|\rb-\lb|$ \emph{the detached distance} of the shock $\shock$ from the blunt body $W_b^+$.
\end{definition}

\begin{theorem}
	\label{main-theorem-psi}
	
	Fix $\gam >1$ and $B_0>0$. Let $\eps_2>0$ be from Lemma \ref{lemma-shock-polar2}.	
	For a fixed constant $d_0>0$, called {\emph{the detached distance of a shock}} from $W_b^+$, and a fixed constant $\beta\in(0,1)$, there exists a small constant $\bar{\eps}\in(0, \eps_2]$ depending on $(\gam, B_0, d_0)$ so that whenever $(\irho, \iu)\in D_{\infty}(\gam, B_0)$ satisfies $\frac{1}{M_{\infty}}=\eps$ for $\eps\in(0, \bar{\eps}]$, Problem \ref{problem-fbp-psi} has a solution $(\fshock, \psi)\in C^1_{\rm{loc}}(\R_+)\times [C^1_{\rm{loc}}(\ol{\Om_{\fshock}})\cap C^2_{\rm{loc}}({\Om}_{\fshock})]$ satisfying  the following properties:
	\begin{itemize}
		\item[(i)](Detached distance)
		\begin{equation}\label{shock-IC}
		\fshock(0)=b_0-d_0;
		\end{equation}

		\item[(ii)](Asymptotic states)
		\begin{equation}
		\label{Asymp-BCs2-psi}
		\lim_{{|\bf x|\rightarrow \infty}\atop {{\bf x}\in \Om_{\fshock}}}
		|\nabla\psi^{\perp}({\bf x})-\rho_{\rm st}^{\eps}{\bf u}_{\rm st}^{\eps}|=0,\quad
		\tx{and}\quad
		\lim_{x_2\to \infty} |\fshock'(x_2)-s_{\rm{st}}^{\eps}|=0;
		\end{equation}
		
		\item[(iii)] There exists a constant $\delta>0$ depending only on $(\gam, B_0, d_0)$ such that
		\begin{equation}\label{uniform-lwrbd-detached-distance}
		b(x_2)-\fshock(x_2)\ge \delta\quad\tx{for all $x_2\ge 0$};
		\end{equation}
		
		\item[(iv)] There exists a constant $\hat{\alp}\in(0,1)$ depending only on $\tw$, and a constant $C$ depending only on $(\gam, B_0, d_0)$ such that
		\begin{equation}\label{apriori-estimate-fbp-solution}
		\|\fshock-f_0\|_{2,\hat{\alp}, \R^+}^{(-\beta)}+
		\|\psi-\psi_0\|_{2,\hat{\alp},\Om_{\fshock}}^{(-\beta)}\le C\mathfrak{q}_{\gam}(\eps)
		\end{equation}
		for the functions $f_0$ and $\psi_0$ given by
\begin{equation}
\label{definition-background}
f_0(x_2):=s_{\rm{st}}^{\eps}x_2+b_0-d_0,\quad \tx{and}\quad
\psi_0({\bf x}):=
\rho_{\rm{st}}^{\eps}u_{\rm{st}}^{\eps}\left(x_2-\kappa_w(x_1-b_0+d_0)\right).
\end{equation}
Here, the function $\mathfrak{q}_{\gam}(\eps)$ is from Lemma \ref{lemma-shock-polar2};

\item[(v)] Set $\rho({\bf x}):=\hat{\rho}(|\nabla\psi|^2({\bf x}))$ for $\hat{\rho}$ from Lemma \ref{lemma-rho-expression}. Then there exists a constant $\sigma\in(0,1)$ depending only on $(\gam, B_0, d_0)$ so that the Mach number $M(\rho)$ defined by \eqref{definition-mach-rho} satisfies the inequality
    \begin{equation*}
      M(\rho({\bf x}))\le 1-\sigma\quad\tx{in $\ol{\Om_{\fshock}}$}.
    \end{equation*}
	\end{itemize}
\end{theorem}

Once Theorem \ref{main-theorem-psi} is proved, then Theorem \ref{main-theorem-physical-var}(a) easily follows from Theorem \ref{main-theorem-psi}. So we first assume that Theorem \ref{main-theorem-psi} holds, then give a short proof of Theorem \ref{main-theorem-physical-var}(a).
\begin{proof}
[Proof of Theorem \ref{main-theorem-physical-var}(a)]
For a fixed constant $d_0>0$, let the constant $\bar{\eps}>0$ be from Theorem \ref{main-theorem-psi}. Suppose that $(\irho, \ui)\in D_{\infty}(\gam, B_0)$ satisfies $M_{\infty}=\frac{1}{\eps}$ for $\eps\in(0,\bar{\eps}]$. By Theorem \ref{main-theorem-psi}, there exists a solution $(\fshock,\psi)$ to Problem \ref{problem-fbp-psi} so that it satisfies \eqref{shock-IC}--\eqref{apriori-estimate-fbp-solution}. Let us set $(\rho, {\bf u}):=\left(\hat{\rho}(|\nabla\psi|^2), \frac{\nabla^{\perp}\psi}{\hat{\rho}(|\nabla\psi|^2)}\right)$. Then it follows from Lemma \ref{lemma-rho-expression} and \eqref{Asymp-BCs2-psi} that $(\rho, {\bf u})$ solves \eqref{E-system} in $\Om_{\fshock}$, and $\rho$ satisfies $\displaystyle{\lim_{{|\bf x|\rightarrow \infty}\atop {{\bf x}\in \Om_{\fshock}}}
		\rho=\rho_{\rm st}^{\eps}}$. By the condition $\psi=\ipsi$ on $\shock$ given in \eqref{BCs-psi}, we have $\der_{\bm\tau}(\psi-\ipsi)=0$ on $\shock$ for a unit tangential vector field ${\bm{\tau}}$ on $\shock$. And, this implies that $(\rho{\bf u}-\irho{\bf u}_{\infty})\cdot{\bm\nu}=0$ on $\shock$ for a unit normal vector field ${\bm{\nu}}$ on $\shock$. Here, ${\bf u}_{\infty}$ represents the vector $(\ui,0)$. The condition \eqref{equation-free-bdry} can be rewritten as $({\bf u}-{\bf u}_{\infty})\cdot{\bm\tau}=0$ on $\shock$. Therefore, $(\rho, {\bf u})$ satisfies the Rankine-Hugoniot condition stated in ($S_3$) of Definition \ref{definition-shock-solution} for the incoming state $(\irho, {\bf u}_{\infty})$. This proves that $(\rho, {\bf u})$ yields a shock solution with the shock $x_1=\fshock(x_2)$ for the incoming supersonic state $(\irho, {\bf u}_{\infty})$ in the sense of Definition \ref{definition-shock-solution}.

Since $(\irho, \iu)\in D_{\infty}(\gam, B_0)$, we have $\irho<\rhosonic$ for $\rhosonic$ given by \eqref{definition-rho-sonic}. On the other hand, \eqref{definition-rho-hat} implies that $\rhosonic<\rho=\hat{\rho}(|\nabla\psi|^2)\le \rho_{\max}$ in $\ol{\Om_{\fshock}}$. Therefore, we have $0<\irho<\rho<\infty$ on $\shock$. Since $\shock$ is the graph of $x_1=\fshock(x_2)$, the unit normal vector field ${\bm\nu}$ pointing interior to $\Om_{\fshock}$ is given by ${\bm\nu}
=\frac{\overrightarrow{(1,-\fshock'(x_2))}}{\sqrt{1+(\fshock'(x_2))^2}}$ so we have ${\bf u}_{\infty}\cdot {\bm\nu}=\frac{\ui}{\sqrt{1+(\fshock'(x_2))^2}}>0$ on $\shock$. Then we obtain from $(\rho{\bf u}-\irho{\bf u}_{\infty})\cdot{\bm\nu}=0$ and $\irho<\rho$ on $\shock$ that $0< {\bf u}\cdot {\bm\nu}< {\bf u}_{\infty}\cdot{\bm\nu}<\infty$ on $\shock$. This proves that $(\rho, {\bf u})$ with a shock $\shock$ satisfies the entropy condition.

Finally, the statements (i)--(v) in Theorem \ref{main-theorem-physical-var}(a) directly follow from the statements (i)--(v) in Theorem \ref{main-theorem-psi}. This completes the proof.

\end{proof}

The proof of Theorem \ref{main-theorem-psi} is given in the next two sections.

\section{Free boundary problems in cut-off domains}
\label{section-fbp-cf}
\numberwithin{equation}{subsection}

In this section, we introduce free boundary problems in cut-off domains.
We will prove Theorem \ref{main-theorem-psi} by finding a sequence of solutions to free boundary problems in cut-off domains then passing to the  limit of the sequence in an appropriately chosen Banach space.


\subsection{Formulation of a free boundary problem in a cut-off domain}
\label{subsection-approximate-shock-set}


Fix a sufficiently large constant $L>1$. For a function $f\in C^1([0,L])$ with satisfying
\begin{equation}
\label{condition-for-f}
 f'(0)=0,\quad\tx{and}\quad f(x_2)<b(x_2)\quad\tx{for $0\le x_2\le L$},
\end{equation}
we introduce a cut-off domain $\Om_{f,L}$ in the following definition.

\begin{figure}[htp]
	\centering
	\begin{psfrags}
		\psfrag{lt}[cc][][0.8][0]{$P_{2,L}^f$}
		\psfrag{rt}[cc][][0.8][0]{$P_{3,L}^f$}
		\psfrag{rb}[cc][][0.8][0]{$\rb$}
		\psfrag{lb}[cc][][0.8][0]{$\lb$}
		\psfrag{sm}[cc][][0.8][0]{$\Gam_{\rm{sym}}$}
		\psfrag{gw}[cc][][0.8][0]{$\Gam_{b,L}^f$}
		\psfrag{cf}[cc][][0.8][0]{$\phantom{aa}\Gam_{{\rm{cutoff}},L}^f$}
		\psfrag{sh}[cc][][0.8][0]{$\Gam_{{\rm{sh}},L}^f$}
		\psfrag{O}[cc][][0.8][0]{$0$}
		\psfrag{om}[cc][][0.8][0]{$\Om_{f,L}$}
		\psfrag{Uf}[cc][][0.8][0]{$M_{\infty}=\frac{1}{\eps}$}
\psfrag{t2}[cc][][0.8][0]{$\phantom{aa}\theta_{P_{2,L}^f}$}
		\includegraphics[scale=1.2]{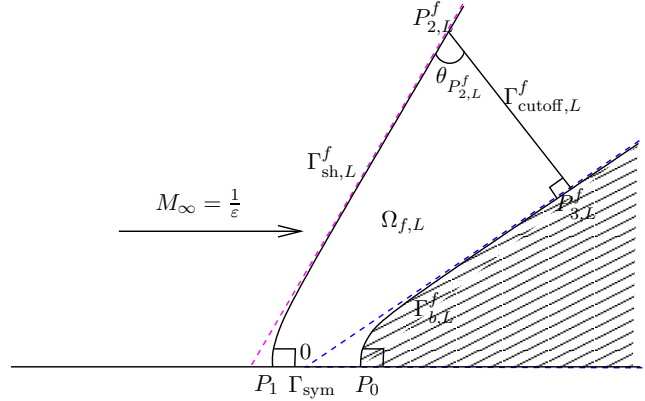}
		\caption{A cut-off domain $\Om_{f,L}$}\label{figure3}
	\end{psfrags}
\end{figure}

\begin{definition}\label{definition-bounded-domain}
For a function $f\in C^1([0,L])$ with satisfying \eqref{condition-for-f}, we set
  \begin{align}
  \label{definition-P2L}
    &P_{2,L}^f:=(f(L), L),\\
\label{definition-approximate-shock}
  \tx{and}\quad
  &\Gam_{\rm{sh},L}^f :=\{(f(x_2),x_2):0 < x_2< L\}.
\end{align}
Let $P_{3,L}^f$ be the point on the line $x_2=x_1\tan\theta_w$ so that the line segment $\ol{P_{2,L}^f P_{3,L}^f}$ is orthogonal to the line $x_1=x_2\cot\theta_w$. A direct computation yields
\begin{equation}\label{definition-P3L}
    P_{3,L}^f=\frac{f(L)+L\kappa_w}{1+\kappa_w^2}(1,\kappa_w),
  \end{equation}
for $\kappa_w=\tan\theta_w$. If $L$ and $f$ are chosen to satisfy
\begin{equation}
\label{condition-for-P3}
  \frac{(f(L)+L\kappa_w)\kappa_w}{1+\kappa_w^2}> h_0,
\end{equation}
then the point $P_{3,L}^f$ lies on $\Gam_b$, and the line segment $\ol{P_{2,L}^f P_{3,L}^f}$ is perpendicular to $\Gam_b$ at $P_{3,L}^f$. See Fig. \ref{figure3}.
We define a cut-off boundary $\Gam_{\rm{cutoff},L}^f$ by
  \begin{equation}\label{definition-cutoff-bdry}
    \Gam_{\rm{cutoff},L}^f:=\ol{P_{2,L}^fP_{3,L}^f},
  \end{equation}
and set
\begin{equation}\label{definition-fbdries}
\Gam^f_{b,L}:=\{(b(x_2), x_2): 0<x_2<x_2^{P^f_{3,L}}\},
\end{equation}
where $x_j^{P}$ represents $x_j$-coordinate of $P\in \R^2$ for $j=1,2$.

Finally, we define $\Om_{f, L}$ by the open region enclosed by $\ol{\Gam_{\rm{sym}}\cup \Gam_{b,L}^f\cup \Gam_{\rm{cutoff},L}^f\cup \Gam_{\rm{sh},L}^f}$.

\end{definition}

\begin{problemL}[Detached shock in a cut-off domain]
\label{fbp-bdd-dmns}
Fix a constant $d_0>0$.
For a fixed constant $L>0$ sufficiently large, find a function $f\in C^1([0,L])$ and a function $\psi\in C^1(\ol{\Om_{f,L}})\cap C^2(\Om_{f,L})$ so that the following properties hold:
\begin{itemize}
\item[(i)] \begin{equation*}
      |\nabla\psi|^2<2H(\rhosonic) \quad\tx{in}\quad \ol{\Om_{f,L}};
    \end{equation*}

 \item[(ii)] The function $\psi$ solves the nonlinear boundary value problem:
     \begin{equation}
     \label{nlbvp-in-fbp-bdd}
       \begin{cases}
        {\rm{div}}\left(\frac{\nabla\psi}{\hat{\rho}(|\nabla\psi|^2)}\right)=0
        \quad\tx{in $\Om_{f,L}$}\\
       \psi=\ipsi\quad\tx{on $\Gam_{\rm{sh},L}^f$}\\
       \psi=0\quad\tx{on $\Gam_{\rm{sym}} \cup \Gam_{b,L}^f$}\\
       \nabla\psi\cdot{\bf n}_c=0\quad\tx{on $\Gam_{\rm{cutoff}, L}^f$ for ${\bf n}_c=\langle \cos \tw, \sin \tw \rangle$}
       \end{cases}
     \end{equation}
     where $\Gam_{\rm{sh},L}^f$, $\Gam_{b,L}^f$ and $\Gam_{\rm{cutoff}, L}^f$ are defined by Definition \ref{definition-bounded-domain}. Here ${\bf n}_c$ is the outward unit normal on $\Gam_{\rm{cutoff}, L}^f$;

\item[(iii)](Free boundary condition)
\begin{equation}
\label{free-bc}
      \begin{split}
         &f'(x_2)
         =\frac{(\psi_{x_1}/\hat{\rho}(|\nabla\psi|^2))(f(x_2),x_2)}
         {(\psi_{x_2}/\hat{\rho}(|\nabla\psi|^2))(f(x_2),x_2)-\iu} \quad\tx{for $0<x_2<L$}\\
         &f(0)=b_0-d_0;
      \end{split}
    \end{equation}

\item[(iv)] The function $f$ satisfies the condition \eqref{condition-for-f}.

\end{itemize}
\end{problemL}

Before stating the solvability of Problem \ref{fbp-bdd-dmns}, we need to define weighted H\"{o}lder norms for functions defined in bounded domains.
\begin{definition}
\label{definition-norms-f-decay-corner}
Fix a constant $L>1$. Let $m, k\in \{0\}\cup \mathbb{N}$ with $k\le m$, $\mu\in \R$, and $\alp\in(0,1)$. For $x_2,x_2'\in (0, L)$, set
\begin{equation*}
  \delta^{(1)}_{x_2}:=\min\{|x_2-L|, 1+x_2\},\quad \delta^{(1)}_{x_2,x_2'}:=\min\{\delta^{(1)}_{x_2}, \delta^{(1)}_{x_2'} \}.
\end{equation*}
For a function $f:[0,L]\rightarrow \R$, define
\begin{equation*}
\begin{split}
&\|f\|_{m, (0,L)}^{(\mu;-(k+\alp),\{L\})} :=
\|f\|_{k,\alp, (0,L)}^{(\mu)}+
\sum_{j>k}^m \sup_{x_2\in (0,L)} (1+x_2)^{j+\mu}
\left(\frac{\delta_{x_2}^{(1)}}{1+x_2}\right)^{ j-(k+\alp)} \left|\frac{d^j}{dx_2^j}f(x_2)\right|\\
&[f]_{m,\alp, (0,L)}^{(\mu;-(k+\alp),\{L\})} :=
\sup_{x_2\neq x'_2\in (0,L)} (1+\min\{x_2,x'_2\})^{m+\alp+\mu}
\left(\frac{\delta_{x_2,x_2'}^{(1)}}{1+\max\{x_2,x_2'\}}\right)
^{ m-k}\\
&\phantom{[f]_{m,\alp, (0,L)}^{(\kappa;\om,\{L\})} :=
\sup_{x_2\neq x'_2\in (0,L)} (1+\min\{x_2,x'_2\})^{m+\alp+\kappa}aaaaaaaa}
\times \frac{|\frac{d^m}{dx_2^m}f(x_2)-\frac{d^m}{dx_2^m}f(x'_2)|}{|x_2-x'_2|^{\alp}}\\
&\|f\|_{m,\alp, (0,L)}^{(\mu;-(k+\alp),\{L\})}:=
\|f\|_{m, (0,L)}^{(\mu;-(k+\alp),\{L\})} + [f]_{m,\alp, (0,L)}^{(\mu;-(k+\alp),\{L\})} .
\end{split}
\end{equation*}
In this definition, the norm $\|\cdot\|_{k,\alp, (0,L)}^{(\mu)}$ is given by Definition \ref{definition-norms-w-decay}.
Let $C^{m,\alp}_{(\mu;-(k+\alp),\{L\})}(0,L)$ denote the set of completion in the norm $\|\cdot\|_{m,\alp, (0,L)}^{(\mu;-(k+\alp), \{L\})}$ for the set of all smooth functions whose $\|\cdot\|_{m,\alp, (0,L)}^{(\mu;-(k+\alp), \{L\})}$-norms are finite.
\end{definition}

\begin{definition}
\label{definition-norms-phi-decay-corner}
Let $\Om$ be an open, bounded and connected domain in $\R^2_+$, and let $P$ be a fixed point on  $\der \Om$.
Let $m, k\in \{0\}\cup \mathbb{N}$ with $k\le m$, $\mu\in \R$, and $\alp\in(0,1)$. For ${\bf x}=(x_1,x_2),{\bf x}'=(x_1', x_2')\in \Om$, set
\begin{equation*}
 \delta^{(2)}_{{\bf x}}:=\min\{ |{\bf x}-P|, 1+x_2\},\quad \delta^{(2)}_{{\bf x},{\bf x}'}:=\min\{\delta^{(2)}_{{\bf x}}, \delta^{(2)}_{{\bf x}'} \}.
\end{equation*}
For a function $\vphi:\Om\rightarrow \R$, define
\begin{equation*}
\begin{split}
&\|\vphi\|_{m, \Om}^{(\mu;-(k+\alp),\{P\})} :=\|\vphi\|_{k,\alp, \Om}^{(\mu)}+
\sum_{j>k}^m \sum_{l=0}^j\sup_{{\bf x}\in \Om} (1+x_2)^{j+\kappa}
\left(\frac{\delta_{x_2}^{(2)}}{(1+x_2)}\right)^{ j-(k+\alp)} \left|\der^l_{x_1}\der^{j-l}_{x_2}\vphi({\bf x})\right|\\
&[\vphi]_{m,\alp, \Om}^{(\mu;-(k+\alp),\{P\})} :=\sum_{l=0}^m
\sup_{{\bf x}\neq {\bf x}'\in \Om} (1+\min\{x_2,x'_2\})^{m+\alp+\mu}
\left(\frac{\delta_{{\bf x},{\bf x}}^{(2)}}{(1+\max\{x_2,x_2'\})}\right)^{m-k}\\
&\phantom{(1+\min\{x_2,x'_2\})^{m+\alp+\kappa}(1+\min\{x_2,x'_2\})^{m+\alp+\kappa}}\times\frac{|\der_{x_1}^l\der_{x_2}^{m-l}\vphi({\bf x})-\der_{x_1}^l\der_{x_2}^{m-l}\vphi({\bf x}')|}{|{\bf x}-{\bf x}'|^{\alp}}\\
&\|\vphi\|_{m,\alp, \Om}^{(\mu;-(k+\alp),\{P\})}:=
\|\vphi\|_{m, \Om}^{(\mu;-(k+\alp),\{P\})} + [\vphi]_{m,\alp, \Om}^{(\mu;-(k+\alp),\{P\})} .
\end{split}
\end{equation*}
In this definition, the norm $\|\cdot\|_{k,\alp, \Om}^{(\mu)}$ is given by Definition \ref{definition-norms-w-decay}.
Let $C^{m,\alp}_{(\mu;-(k+\alp),\{P\})}(\Om)$ denote the set of completion in the norm  $\|\cdot\|_{m,\alp, \Om}^{(\mu;-(k+\alp),\{P\})}$ for the set of all smooth functions whose $\|\cdot\|_{m,\alp, \Om}^{(\mu;-(k+\alp), \{P\})}$-norms are finite.
\end{definition}

\begin{remark}[Compact embedding property]
\label{remark-compact-embedding}
By applying Arzel\`{a}-Ascoli theorem, one can directly check that $C^{m,\alp_1}_{(\mu;-(k+\alp_1),\{P\})}(\Om)$ is compactly embedded into $C^{m,\alp_2}_{(\mu;-(k+\alp_2),\{P\})}(\Om)$ for $m,k\in \{0\}\cup \mathbb{N}$ with $k\le m$ and $0<\alp_2<\alp_1<1$.

\end{remark}

The following proposition states the solvability of Problem \ref{fbp-bdd-dmns} for any $L$ sufficiently large depending on $(\gam, B_0, d_0)$.

\begin{proposition}
  \label{proposition-wp-fbpbd}
Fix $\gam >1$ and $B_0>0$. Let $\eps_2>0$ be from Lemma \ref{lemma-shock-polar2}, respectively.
\smallskip

For fixed constants $d_0>0$ and $\beta\in(0,1)$, there exist a small constant $\eps_3\in (0, \eps_2]$ and a large constant $L_*$ depending only on $(\gam, B_0, d_0)$ so that if
$(\irho, \ui)\in D_{\infty}(\gam, B_0)$ satisfies $M_{\infty}=\frac{1}{\eps}$ for $\eps\in(0, \eps_3]$, and if
$L\ge L_*$ holds, then Problem \ref{fbp-bdd-dmns} has a solution $(\fshock, \psi)$ satisfying the following properties:
\begin{itemize}
		\item[(i)] There exists a constant $\delta>0$ depending only on $(\gam, B_0, d_0)$ such that
		\begin{equation}\label{uniform-lwrbd-detached-distance-bdd}
		b(x_2)-\fshock(x_2)\ge \delta\quad\tx{for all $x_2\in[0,L]$};
		\end{equation}
		
		\item[(ii)] There exist a constant $\hat{\alp}\in(0,1)$ depending only on $\tw$, and a constant $C>0$ depending only on $(\gam, B_0, d_0)$ such that
		\begin{equation}\label{apriori-estimate-fbp-solution-bdd}
		\|\fshock-f_0\|_{2,\hat{\alp}, (0,L)}^{(-\beta;-(1+\hat{\alp}),\{L\})}+
		\|\psi-\psi_0\|_{2,\hat{\alp},\Om_{\fshock,L}}^{(-\beta;-(1+\hat{\alp}),\{P_{2,L}^{\fshock}\})}\le C\mathfrak{q}_{\gam}(\eps)
		\end{equation}
for the functions $f_0$ and $\psi_0$ given by \eqref{definition-background}.
The function $\mathfrak{q}_{\gam}(\eps)$ in \eqref{apriori-estimate-fbp-solution-bdd} is from Lemma \ref{lemma-shock-polar2}. And, the domain $\Om_{\fshock,L}$ is given by Definition \ref{definition-bounded-domain}.

\item[(iii)] Set $\rho({\bf x}):=\hat{\rho}(|\nabla\psi|^2({\bf x}))$ for $\hat{\rho}$ from Lemma \ref{lemma-rho-expression}. Then there exists a constant $\sigma\in(0,1)$ depending only on $(\gam, B_0, d_0)$ so that the Mach number $M(\rho)$ defined by \eqref{definition-mach-rho} satisfy the inequality
    \begin{equation}
    \label{subsonicity-fbp-bdd}
      M(\rho({\bf x}))\le 1-\sigma\quad\tx{in $\ol{\Om_{\fshock,L}}$}.
    \end{equation}
	\end{itemize}
\end{proposition}
The rest of \S \ref{section-fbp-cf} is devoted to proving Proposition \ref{proposition-wp-fbpbd}.

\subsection{Iteration set of approximate shocks on bounded intervals}
\label{subsection-iterset1}
We will prove Proposition \ref{proposition-wp-fbpbd} by the method of iteration. For that reason, we introduce a set of approximate shocks.

\begin{definition}
\label{definition-of-shock-iterset}
Let $f_0$ be given by \eqref{definition-background}.
Fix a constant $L\ge 10(1+h_0)$.
For $\alp\in(0,1)$, $\beta\in(0,1)$, $M_1>0$ and $\eps>0$, define a set:
\begin{equation}
\label{definition-iterset-shocks}
\begin{split}
\mcl{J}_{M_1,L}:=\{f\in \,C^{2,\alp}_{(-\beta;-(1+\alp),\{L\})}(0,L): &\, f(0)=b_0-d_0,\,\,f'(0)=0,\,\, \\
&f(x_2)<b(x_2)\,\,\tx{for all $x_2\in [0,L]$},\\
&\|f-f_0\|_{2,\alp,(0,L)}^{(-\beta;-(1+\alp),\{L\})}\le M_1\mathfrak{q}_{\gam}(\eps)\,\,
\}
\end{split}
\end{equation}
for the function $\mathfrak{q}_{\gam}(\eps)$ from Lemma \ref{lemma-shock-polar2}.
The constants $(\alp, M_1, \eps)$ will be specified later.
\end{definition}
Note that $f_0\not\in \mcl{J}_{M_1,L}$ because $f_0'(0)\neq 0$. So we first check that $\mcl{J}_{M_1,L}$ is nonempty in the following lemma.

\begin{lemma}
\label{lemma-f-iterset-nonempty}
There exists a constant $M_1^*>0$ depending only on $(\gam, B_0)$ so that whenever $M_1\ge M_1^*$ and $\eps\in(0,\eps_2]$ for $\eps_2>0$ be from Lemma \ref{lemma-shock-polar2}, the set $\mcl{J}_{M_1,L}$ is nonempty. Furthermore, the choice of $M_1^*$ is independent of $L\in [ 10(1+h_0),\infty)$.

\begin{proof}
Let $\chi:\R\rightarrow \R$ be a smooth function satisfying the following properties:
\begin{itemize}
\item[(i)] $\chi(x_2)=\begin{cases}1\quad&\mbox{for $x_2\le 5$}\\
    0\quad&\mbox{for $x_2\ge 10$}
\end{cases}$, \quad and $0\le \chi(x_2)\le 1$ for all $x_2\in \R$;

\item[(ii)] $\|\chi\|_{C^3(\R^+)}\le m$ for some constant $m>0$.
\end{itemize}
Set
\begin{equation}
\label{definition-cutoff-function}
  \chi_{h_0}(x_2):=\chi\left(\frac{x_2}{h_0}\right).
\end{equation}
Then we have $\displaystyle{\|\chi_{h_0}\|_{C^3(\R)}\le \|\chi\|_{C^3(\R)} \sum_{k=0}^3\left(\frac{1}{h_0}\right)^k}$.

We define a function $f_*:\R\rightarrow \R$ by
\begin{equation*}
  f_*(x_2)=(b_0-d_0)\chi_{h_0}(x_2)+(1-\chi_{h_0}(x_2))f_0(x_2).
\end{equation*}
Since $b_0-d_0<b(x_2)$ and $f_0(x_2)<b(x_2)$ hold for all $x_2\ge 0$, we have
\begin{equation*}
  f_*(x_2)<b(x_2)\quad\tx{for all $x_2\ge 0$.}
\end{equation*}
By \eqref{definition-background} and the definition of $\chi_{h_0}$, $f_*$ satisfies
\begin{equation*}
\begin{split}
  &f_*(0)=b_0-d_0,\\
  &f_*'(0)=0,\\
  &(f_*-f_0)(x_2)=(b_0-d_0-f_0(x_2))\chi_{h_0}(x_2)
=-s_{\rm{st}}^{\eps}x_2\chi_{h_0}(x_2).
  \end{split}
\end{equation*}
Then, we apply Lemma \ref{lemma-shock-polar2} to conclude that if $\eps\in (0,\eps_2]$, then there exists a constant $C_*>0$ depending only on $(\gam, B_0)$ such that
\begin{equation*}
  \|f_*-f_0\|_{2,\alp,(0,L)}^{(-\beta;-(1+\alp),\{L\})}
  \le C_*\epsmod.
\end{equation*}
Note that the constant $C_*$ is independent of $L\in[10(1+h_0), \infty)$ and $\beta\in(0,1)$.
We choose $M_1^*$ as
\begin{equation*}
  M_1^*=2C_*.
\end{equation*}
Then we have
\begin{equation*}
  f_*\in\mcl{J}_{M_1,L},
\end{equation*}
for $M_1\ge M_1^*$, $\eps\in (0,\eps_2]$, $L\in[10(1+h_0), \infty)$ and $\beta\in(0,1)$. This completes the proof.
\end{proof}
\end{lemma}

Next, we find a sufficient condition for $(M_1, \eps, L)$ so that each $f\in \mcl{J}_{M_1,L}$ satisfies the condition \eqref{condition-for-f}, and that the point $P_{3,L}^f$ given by \eqref{definition-P3L} lies on $\Gam_{b}$, that is, $(L, f)$ satisfies \eqref{condition-for-P3}.

By \eqref{definition-background}, \eqref{definition-iterset-shocks} and Lemma \ref{lemma-shock-polar2}, if $\eps\in (0,\eps_2]$, then each $f\in \mcl{J}_{M_1,L}$ satisfies
\begin{equation}\label{estimate-f-iterset}
  |f'(x_2)|\le s_{\rm st}^{\eps}+M_1\epsmod\le (M_1+1)\epsmod\quad\tx{for $x_2\in [0,L]$.}
\end{equation}
In the several lemmas stated below, we assume that $(\eps, M_1)$ satisfy the following inequalities:
\begin{equation}
	\label{shock-iterset-condtion1}
	\eps\le \eps_2,\quad\tx{and}\quad
	(M_1+1)\epsmod\le \frac 14 \min\{ \cot\tw, \tan\tw\}.
	\end{equation}

First, we find a sufficient condition for $L$ so that, for each $f\in \mcl{J}_{M_1,L}$, $(L, f)$ satisfies \eqref{condition-for-P3}.

\begin{lemma}
\label{lemma-P3L}
Assume that \eqref{shock-iterset-condtion1} holds. Then there exists a constant $\underline{L}>0$ depending only on $d_0$ so that if $L\ge \underline{L}$ , then, for each $f\in \mcl{J}_{M_1,L}$, $(L, f)$ satisfies \eqref{condition-for-P3} so that the point $P_{3,L}^f$ lies on $\Gam_b$.
\end{lemma}
\begin{proof}
Fix $f\in \mcl{J}_{M_1,L}$.
Since $f(0)=b_0-d_0$, we get
\begin{equation*}
  \frac{\kappa_w}{1+\kappa_w^2}(f(L)+L\kappa_w)=
  \frac{\kappa_w}{1+\kappa_w^2}
  \left((b_0-d_0)+\int_0^L f'(t)+\kappa_w\,dt\right).
\end{equation*}
Then we apply \eqref{estimate-f-iterset} and \eqref{shock-iterset-condtion1} to obtain that
\begin{equation*}
\begin{split}
\frac{\kappa_w}{1+\kappa_w^2}(f(L)+L\kappa_w)\ge
\frac{\kappa_w}{1+\kappa_w^2}(b_0-d_0+\frac L2\kappa_w).
\end{split}
\end{equation*}
Therefore, if we choose $\underline{L}$ as
\begin{equation*}
\underline{L}=\frac{4}{\kappa_w}
\left(\frac{(1+\kappa_w^2)h_0}{\kappa_w}+d_0-b_0\right),
\end{equation*}
and if $L\ge \underline{L}$, then $(L, f)$ satisfies \eqref{condition-for-P3}.
\end{proof}

Fix a constant $L\in[\underline{L}, \infty)$ for $\underline{L}$ from Lemma \ref{lemma-P3L}, and assume that \eqref{shock-iterset-condtion1} holds.  For each $f\in \mcl{J}_{M_1,L}$, let $P_{2,L}^f$, $P_{3,L}^f$, $\Gam_{\rm sh,L}^f$, $\Gam_{{\rm cutoff},L}^f$ and $\Gam_{b,L}^f$ be defined by Definition \ref{definition-bounded-domain}. Next, we find a sufficient condition for $(M_1, \eps)$ so that any function $f\in \mcl{J}_{M_1,L}$ satisfies \eqref{condition-for-f}.

For each $f\in \mcl{J}_{M_1,L}$, let us set
\begin{equation*}
d_f^L:={\rm dist}(\Gam^f_{{\rm{sh}},L}, \Gam^f_{b,L})=\inf_{{{\bf x}\in \Gam^f_{{\rm sh},L}}\atop {{\bf x}'\in \Gam_{b,L}^f}}|{\bf x}-{\bf x}'|.
\end{equation*}
And, let $\theta_{P_{2,L}^f}$ be the angle between $\Gam_{{\rm{sh}},L}^f$ and $\Gam_{\rm{cutoff},L}^f$ at $P_{2,L}^f$. See Fig. \ref{figure3}.


\begin{lemma}
\label{lemma-distance-iterset}
For $\underline{L}$ from Lemma \ref{lemma-P3L},
let $L$ satisfy $L\ge \underline{L}+10h_0$. And, assume that the condition \eqref{shock-iterset-condtion1} holds. Then, there exists a small constant
$\sigma_2$ depending only on $d_0$ so that if $(M_1, \eps)$ satisfy
\begin{equation}
\label{shock-iterset-condition4}
(M_1+1)\mathfrak{q}_{\gam}(\eps)\le \sigma_2,
\end{equation}
then any $f\in \mcl{J}_{M_1,L}$ satisfies the following properties:
\begin{itemize}
\item[(a)] $f\in \mcl{J}_{M_1,L}$ satisfies \eqref{condition-for-f}. Furthermore, we have
\begin{equation}
\label{nonzero-detached-dist}
(b-f)(x_2)\ge \frac{3}{4}d_0 \quad\tx{for $x_2\in [0,L]$};
\end{equation}

\item[(b)] There exists $d_*>0$ depending only on $d_0$ to satisfy
\begin{equation*}
d_f^L\ge d_*.
\end{equation*}
Furthermore, the constant $d_*$ is chosen independent of $L$.
\item[(c)] Set $a_w:=\frac{\pi}{2}-\tw (>0)$. Then,
\begin{equation}
\label{estimate-t2}
|\theta_{P_{2,L}^f}-\tw|\le \frac 18\min\{a_w, \tw\}\Rightarrow
\frac 78 \tw \le \theta_{P_{2,L}^f}\le \frac{\pi}{2}-\frac 78 a_w.
\end{equation}
\end{itemize}

\begin{proof}
{\textbf{1.}} Proof of (a): Note that any $f\in \mcl{J}_{M_1,L}$ satisfies $(b-f)(0)=d_0$.

For $0\le x_2\le h_0$, we obtain from the property ($b_3$) in  Definition \ref{definition-bluntbody-ftn},  \eqref{estimate-f-iterset} and \eqref{shock-iterset-condition4} that
\begin{equation*}
(b-f)(x_2)=d_0+\int_0^{x_2}(b-f)'(t)dt
\ge
d_0-(M_1+1)\mathfrak{q}_{\gam}(\eps)h_0\ge d_0-\sigma_2 h_0.
\end{equation*}	
Therefore, if $\sigma_2$ satisfies
\begin{equation}\label{sigma2-cond1}
  \sigma_2h_0\leq d_0/4,
\end{equation}
then we get
\begin{equation}\label{bf-estiamte1}
  (b-f)(x_2)\ge \frac{3}{4}d_0\quad\tx{for $0\le x_2\le h_0$}.
\end{equation}
For $x_2\geq h_0$, it follows from the property ($b_5$) in Definition \ref{definition-bluntbody-ftn} and \eqref{shock-iterset-condtion1} that $(b-f)'(x_2)\ge \frac 34 \cot \tw$. Therefore, if $(M_1, \eps)$ satisfy \eqref{shock-iterset-condition4}, and if the inequality \eqref{sigma2-cond1} holds, then we get
\begin{equation}\label{bf-estiamte2}
(b-f)(x_2)=(b-f)(h_0)+\int_{h_0}^{x_2}(b-f)'(t)dt \geq\frac{3}{4}d_0.
\end{equation}
So we conclude that any $f\in \mcl{J}_{M_1,L}$ satisfies \eqref{condition-for-f} provided that \eqref{sigma2-cond1} holds.

{\textbf{2.}} Proof of (b):
We define a function
\begin{equation*}
  \til{b}(x_2):=b(x_2)-\frac{d_0}{2}.
\end{equation*}
Under the conditions \eqref{shock-iterset-condition4} and \eqref{sigma2-cond1}, it follows from the estimates \eqref{bf-estiamte1} and \eqref{bf-estiamte2} that each $f\in \mcl{J}_{M_1,L}$ satisfies
\begin{equation}
\label{bf-estiamte3}
  (\til{b}-f)(x_2)\ge \frac{d_0}{4}\quad\tx{for $0\le x_2\le L$}.
\end{equation}
Therefore, the graph of $x_1=\til{b}(x_2)$ intersects ${\rm int}\,\Gam_{{\rm{cutoff}},L}^f$ at a unique point $Q_b$. Let $x_2^{Q_b}$ be the $x_2$-coordinate of the point $Q_b$. Let us define
$
\Gam_{\til b,L}:=\{(\til b(x_2),x_2):0\le x_2\le x_2^{Q_b}\}.
$
See Fig. \ref{figure_distance}.
\begin{figure}[htp]
	\centering
	\begin{psfrags}
		\psfrag{p2}[cc][][0.8][0]{$P_{2,L}^f$}
		\psfrag{p3}[cc][][0.8][0]{$P_{3,L}^f$}
		\psfrag{p0}[cc][][0.8][0]{$\rb$}
		\psfrag{p1}[cc][][0.8][0]{$\lb$}
		\psfrag{q1}[cc][][0.8][0]{$Q_b$}
		\psfrag{w}[cc][][0.8][0]{$\Gam_{b,L}^f$}
		\psfrag{it}[cc][][0.8][0]{$\Gam_{\til b, L}$}
		\psfrag{sh}[cc][][0.8][0]{$\Gam_{{\rm{sh}},L}^f$}
		\psfrag{O}[cc][][0.8][0]{$0$}
		\psfrag{om}[cc][][0.8][0]{$\Om_{f,L}$}
		\psfrag{Uf}[cc][][0.8][0]{$M_{\infty}=\frac{1}{\eps}$}
\psfrag{t2}[cc][][0.8][0]{$\phantom{aa}\theta_{P_{2,L}^f}$}
		\includegraphics[scale=0.6]{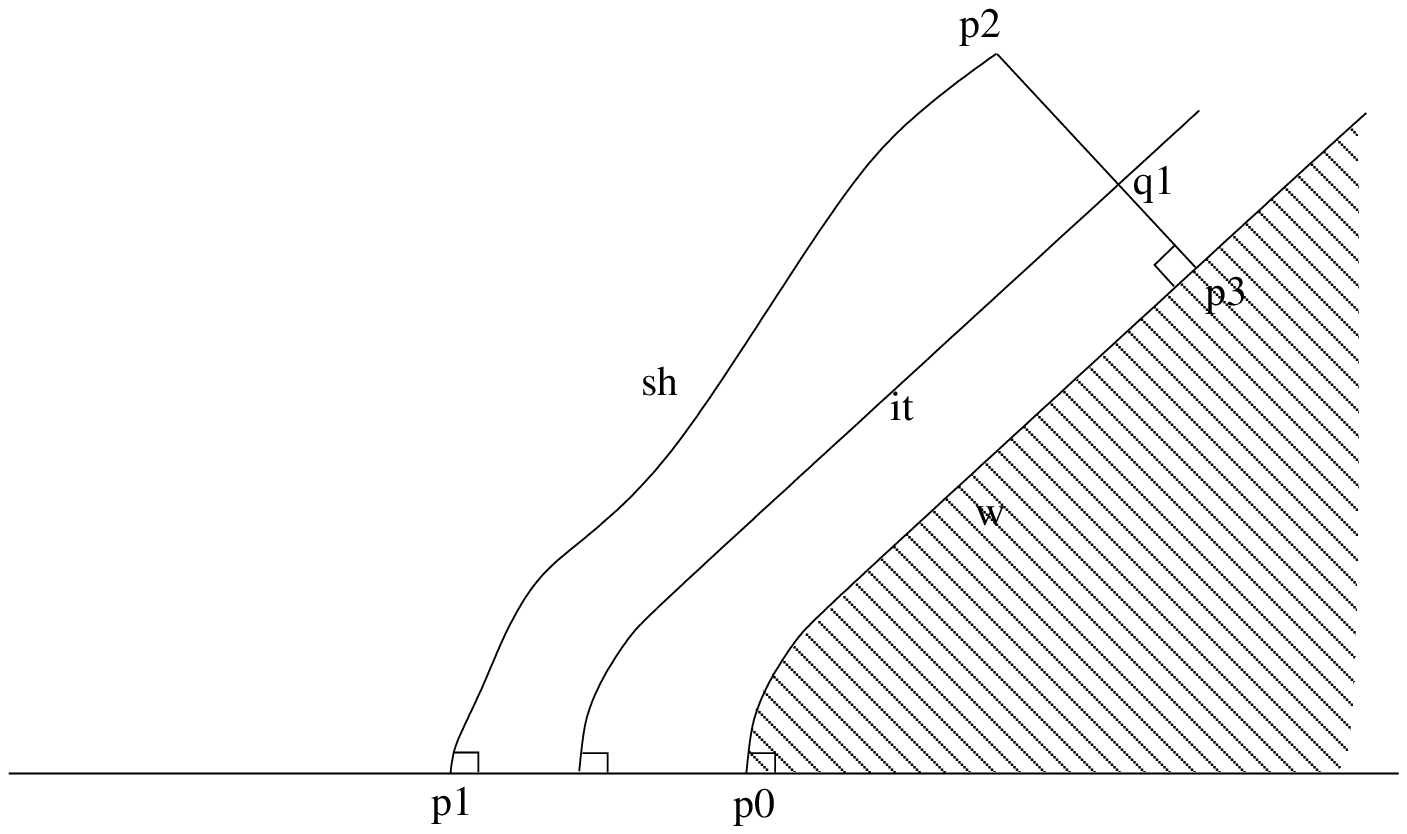}
		\caption{}
\label{figure_distance}
	\end{psfrags}
\end{figure}
By \eqref{bf-estiamte3}, for each $f\in \mcl{J}_{M_1,L}$, we have
\begin{equation}
\label{d-estimate1}
  d_f^L\ge {\rm dist}(\Gam_{\til b,L}, \Gam^f_{b,L}).
\end{equation}
Due to the property ($b_5$) in Definition \ref{definition-bluntbody-ftn}, we have $\til{b}(x_2)=x_2\cot \tw-\frac{d_0}{2}$ for $x_2\ge h_0$, from which it is obtained that
\begin{equation*}
  {\rm dist}\left(\Gam_{\til b,L}\cap\{(x_1,x_2):x_2\ge h_0+\frac{d_0}{2}\cos \tw\}, \Gam_{b,L}^f\right)=\frac{d_0}{2}\sin \tw.
\end{equation*}
It easily follows from \eqref{bf-estiamte3} that there exists a constant $d_1>0$ depending on $d_0$ but independent of $L$ so that we have
\begin{equation*}
  {\rm dist}\left(\Gam_{\til b,L}\cap\{(x_1,x_2):0\le x_2\le h_0+\frac{d_0}{2}\cos \tw\}, \Gam_{b,L}^f\right)\ge d_1.
\end{equation*}
Then \eqref{d-estimate1} yields that
\begin{equation*}
  d_f^L\ge \min\{\frac{d_0}{2}\sin \tw, d_1\}=:d_*.
\end{equation*}
This proves the statement (b).

{\textbf{3.}} Proof of (c):
If $f'(L)=0$, then we have $\theta_{P_{2,L}^f}=\tw$. Therefore, one can choose a small constant $\sigma_{\tw}>0$ depending only on $\tw$ so that if $(M_1+1)\mathfrak{q}_{\gam}(\eps)\le \sigma_{\tw}$,
then \eqref{estimate-f-iterset} implies that
\begin{equation}
\label{estimate-t2}
|\theta_{P_{2,L}^f}-\tw|\le \frac 18\min\{a_w, \tw\}\Rightarrow
\frac 78 \tw \le \theta_{P_{2,L}^f}\le \frac{\pi}{2}-\frac 78 a_w
\end{equation}
for any $f\in \mcl{J}_{M_1,L}$.

{\textbf{4.}} Finally, the proof of Lemma \ref{lemma-distance-iterset} is completed by choosing $\sigma_2$ as
\begin{equation*}
  \sigma_2=\min\{\frac{d_0}{4h_0}, \sigma_{\tw}\}.
\end{equation*}
\end{proof}
\end{lemma}

\begin{condition}
\label{condition-for-parameters}
Hereafter, we assume that
\begin{itemize}
\item[(i)] $L\ge \underline{L}+10(1+h_0)$ for $\underline{L}$ from Lemma \ref{lemma-P3L};

\item[(ii)] $\eps\le \eps_2$ for $\eps_2$ from Lemma \ref{lemma-shock-polar2};

\item[(iii)] $M_1\ge M_1^*$ for $M_1^*$ from Lemma \ref{lemma-f-iterset-nonempty};

\item[(iv)] $(M_1,\eps)$ satisfy
\begin{equation*}
  (M_1+1)\mathfrak{q}_{\gam}(\eps)\le \min\{\sigma_2, \frac{\cot \tw}{4}, \frac{\tan \tw}{4}\}
\end{equation*}
for $\sigma_2$ from Lemma \ref{lemma-distance-iterset}.
\end{itemize}
\end{condition}
Under Condition \ref{condition-for-parameters}, we can conclude from Lemmas \ref{lemma-f-iterset-nonempty}--\ref{lemma-distance-iterset} that the iteration set $\mcl{J}_{M_1,L}$ is nonempty, and that any function $f\in \mcl{J}_{M_1,L}$ yields a simply connected domain $\Om_{f,L}$ given by Definition \ref{definition-bounded-domain}.

\subsection{Linearized boundary value problems in fixed domains $\Om_{f,L}$}
\label{subsection-fixedbd-nlbvp-psi}

We fix $f\in \mcl{J}_{M_1,L}$, and let $\Gam_{\rm sh,L}^f$, $\Gam_{\rm{cutoff},L}^f$, $\Gam_{b,L}^f$ and $\Om_{f,L}$ be given by Definition \ref{definition-bounded-domain}. In this fixed domain $\Om_{f,L}$, we consider the nonlinear boundary value problem \eqref{nlbvp-in-fbp-bdd}. For convenience of computation, we use the equation \eqref{equation-psi-nondiv}, which is the non-divergence form of the equation $ {\rm{div}}\left(\frac{\nabla\psi}{\hat{\rho}(|\nabla\psi|^2)}\right)=0$ from \eqref{nlbvp-in-fbp-bdd}. So we rewrite \eqref{nlbvp-in-fbp-bdd} as
\begin{equation}
\label{fixed-nlbvp-bounded-psi}
\begin{split}
\mcl{N}(\psi)=0 &\qquad\tx{in} \quad \Om_{f,L}\\
\psi=\ipsi &\qquad\tx{on}\quad\Gam^{f}_{{\rm sh},L}\\
\psi=0 &\qquad\tx{on} \quad \Gam_{\rm{sym}}\cup \Gam_{b,L}^f\\
\der_{{\bf n}_c}\psi=0 &\qquad\tx{on}\quad \Gam_{\rm{cutoff},L}^f,
\end{split}
\end{equation}
where ${\bf n}_c$ is the outward unit normal on $\Gam_{\rm{cutoff},L}^f$, that is, ${\bf n}_c=(\cos \tw, \sin \tw)$. Here, the nonlinear differential operator $\mcl{N}(\psi)$ is given by
\begin{equation}
\label{definition-N}
\begin{split}
\mcl{N}(\psi)&:=\sum_{i,j=1}^2
\left(c^2(|\nabla\psi|^2)\delta_{ij}
-\frac{(-1)^{i+j}\psi_{x_{i'}}\psi_{x_{j'}}}{\hat{\rho}^{2}(|\nabla\psi|^2)}
\right)\psi_{x_ix_j},
\end{split}
\end{equation}
with $c^2(|\nabla\psi|^2)=\hat{\rho}^{\gam-1}(|\nabla\psi|^2)$ for $\hat{\rho}$ from Lemma \ref{lemma-rho-expression}, and  $(1',2'):= (2,1)$. And, $\ipsi$ is given by \eqref{definition-ipsi}.



For $\psi_0$ given by \eqref{definition-background}, let us set
\begin{equation*}
  \phi:=\psi-\psi_0.
\end{equation*}
Then function $\psi:\ol{\Om_{f,L}}\rightarrow \R$ solves \eqref{fixed-nlbvp-bounded-psi} if and only if $\phi$ solves
\begin{equation}
\label{fixed-nlbvp-bounded-phi}
\begin{split}
\left(c^2(|\nabla\phi+\nabla\psi_0|^2)\delta_{ij}
-\frac{(-1)^{i+j}(\phi+\psi_0)_{x_{i'}}(\phi+\psi_0)_{x_{j'}}}
{\hat{\rho}^{2}(|\nabla\phi+\nabla\psi_0|^2)}
\right)\phi_{x_ix_j}=0&\qquad\tx{in} \quad \Om_{f,L}\\
\phi=\ipsi-\psi_0&\qquad\tx{on}\quad\Gam_{\rm{sh},L}^f\\
\phi=-\psi_0&\qquad\tx{on} \quad\Gam_{\rm{sym}}\cup \Gam_{b,L}^f \\
\nabla\phi\cdot{\bf n}_c=0&\qquad\tx{on}\quad \Gam_{\rm{cutoff},L}^f.
\end{split}
\end{equation}

To solve the nonlinear boundary value problem \eqref{fixed-nlbvp-bounded-phi} by applying a fixed point theorem, we introduce an iteration set of approximate solutions.
For $\alp\in(0,1)$, $\beta\in(0,1)$, $M_2>0$ and $\eps>0$ to be determined later, we define a set ${\mcl{I}}_{M_2,L}^f$ by
\begin{equation}
\label{definition-iterset-psi}
\begin{split}
{\mcl{I}}_{M_2,L}^f:=\Bigl\{\vphi\in C^{2,\alp}_{(-\beta;-(1+\alp),\{ P_{2,L}^f\})}(\Om_{f,L}):\,\,&
\|\vphi\|_{2,\alp,\Om_{f,L}}^{(-\beta;-(1+\alp),\{P_{2,L}^f\})}\le M_2\mathfrak{q}_{\gam}(\eps),\\
&\phantom{aaaaaaa}\vphi=-\psi_0\,\,\tx{on}\,\,\ol{\Gam_{\rm{sym}}\cup \Gam_{b,L}^f}\,\,
\Bigr\}.
\end{split}
\end{equation}
We remind that the point $P_{2,L}^f$ is given by $P_{2,L}^f=(f(L),L)$ in Definition \ref{definition-bounded-domain}.

In \eqref{definition-iterset-psi}, we intend to fix $(M_2, \eps)$ so that $M_2\mathfrak{q}_{\gam}(\eps)$ is sufficiently small uniformly for all $L\ge \underline{L}+10(1+h_0)$. From this perspective, $-\psi_0$ is not contained in the set ${\mcl{I}}_{M_2,L}^f$ because $\|\psi_0\|
_{2,\alp,\Om_{f,L}}^{(-\beta;-(1+\alp),\{P_{2,L}^f\})}$ diverges to $\infty$ as $L$ tends to $\infty$.
So we need to check that the iteration set ${\mcl{I}}_{M_2,L}^f$ is nonempty.

\begin{lemma}
\label{lemma-vphi-iterset-nonempty}
Under Condition \ref{condition-for-parameters}, there exists a constant $M_2^*>0$ depending only on $(\gam, B_0, d_0)$ so that, for any $f\in \mcl{J}_{M_1,L}$, the set $\mcl{I}_{M_2,L}^f$ given by \eqref{definition-iterset-psi} is nonempty. Note that the constant $M_2^*$ is chosen independently of $L$.

\begin{proof}
By property ($b_5$) in Definition \ref{definition-bluntbody-ftn}, $x_2-\kappa_wx_1=0$ holds on $\Gam_{b,L}^f\cap\{x_2\ge h_0\}$, so it directly follows from \eqref{definition-background} that
\begin{equation}
\label{psi0-on-bdry2}
-\psi_0({\bf x})=-\rho_{\rm{st}}^{\eps}u_{\rm{st}}^{\eps}\kappa_w(b_0-d_0)
\quad\tx{on $\Gam_{b,L}^f\cap\{x_2\ge h_0\}$}.
\end{equation}

For the cut-off function $\chi_{h_0}$ given by \eqref{definition-cutoff-function}, we define $\vphi_*({\bf x}):\ol{\Om_{f,L}}\rightarrow \R$ by
\begin{equation*}
  \vphi_*({\bf x})=-\psi_0({\bf x})\chi_{h_0}(x_2)-\rho_{\rm{st}}^{\eps}u_{\rm{st}}^{\eps}
  \kappa_w(b_0-d_0)(1-\chi_{h_0}(x_2)).
\end{equation*}
By using \eqref{psi0-on-bdry2}, it can be directly checked that $\vphi_*=-\psi_0$ on $\ol{\Gam_{\rm{sym}}\cup \Gam_{b,L}^f}$. By \eqref{definition-background} and Lemma \ref{lemma-shock-polar2}, one can find a constant $\hat{C}>0$ depending only on $(\gam, B_0,d_0)$ such that
\begin{equation*}
  \|\vphi_*\|_{2,\alp, \Om_{f,L}}^{(-\beta;-(1+\alp),\{P_{2,L}^f\})}\le \hat{C}q_{\gam}(\eps).
\end{equation*}
Note that the constant $\hat{C}$ is independent of $(\beta, L, f)$.

If $M_2$ in the definition of $\mcl{I}_{M_2,L}^f$ satisfies $M_2\ge 2\hat{C}$, then we have $\vphi_*\in\mcl{I}_{M_2,L}^f$. Therefore, the proof is completed by choosing $M_2^*$ as $M_2^*=2\hat{C}$.
\end{proof}

\end{lemma}
For each $\vphi\in {\mcl{I}}_{M_2,L}^f$, we define
\begin{equation}
\label{definition-coefficients}
\begin{split}
&a_{11}^{\vphi}:=c^2(|\nabla\vphi+\nabla\psi_0|^2)-
\frac{(\vphi+\psi_0)_{x_2}^2}{\hat{\rho}^2(|\nabla\vphi+\nabla\psi_0|^2)}\\ &a_{22}^{\vphi}:=c^2(|\nabla\vphi+\nabla\psi_0|^2)-
\frac{(\vphi+\psi_0)_{x_1}^2}{\hat{\rho}^2(|\nabla\vphi+\nabla\psi_0|^2)}\\
&a_{12}^{\vphi}=a_{21}^{\vphi}
:=\frac{(\vphi+\psi_0)_{x_1}(\vphi+\psi_0)_{x_2}}
{\hat{\rho}^2(|\nabla\vphi+\nabla\psi_0|^2)}.\\
\end{split}
\end{equation}
And, we define a linear differential operator $\mcl{L}^{\vphi}$ associated with $\vphi\in \mcl{I}_{M_2,L}^f$ by
\begin{equation*}
\mcl{L}^{\vphi}(\phi):=\sum_{i,j=1}^2a_{ij}^{\vphi} \der_{ij}\phi\quad \tx{in $\Om_{f,L}$.}
\end{equation*}


\begin{lemma}
\label{lemma-coefficients}
Under Condition \ref{condition-for-parameters}, there exists a constant $\sigma_{\rm{el}}>0$ depending only on $(\gam, B_0)$ so that if
\begin{equation}\label{condition3-for-epsilon}
 M_2\mathfrak{q}_{\gam}(\eps)\le \sigma_{\rm{el}},
\end{equation}
then, for each $\vphi\in \mcl{I}_{M_2,L}^f$, the coefficient matrix $[a_{ij}^{\vphi}]_{i,j=1}^2$ satisfies the following properties:
\begin{itemize}

\item[(a)] There exists a constant $C$ depending only on $(\gam, B_0)$ such that
\begin{equation*}
\begin{split}
 &\sum_{i,j=1}^2\|a^{\vphi}_{ij}-a^{0}_{ij}\|_{1,\alp, \Om_{f,L}}^{(1-\beta;-\alp,\{P_{2,f}^L\})}\le CM_2\mathfrak{q}_{\gam}(\eps),
\end{split}
\end{equation*}
where $a_{ij}^0=c^2(|\nabla\psi_0|^2)\delta_{ij}
-\frac{(-1)^{i+j}(\psi_0)_{x_{i'}}(\psi_0)_{x_{j'}}}
{\hat{\rho}^2(|\nabla\psi_0|^2)}$ with $(1',2'):=(2,1)$;

\item[(b)] Let us set $c_0^2:=(\gam-1)B_0(=(\rho_{\rm st}^0)^{\gam-1})$. Then, there exists a constant $\lambda\in(0,1)$ depending only on $(\gam, B_0)$ such that, for any $\bm\xi=(\xi_1,\xi_2)\in \R^2$, and ${\bf x}=(x_1, x_2)\in \Om_{f,L}$, we have
\begin{align}
\label{estimate-ellipticity}
&\lambda{c_0^2}|{\bm\xi}|^2 \le \sum_{i,j=1}^2 a^{\vphi}_{ij}({\bf x})\xi_i\xi_j \le \frac{ {c_0}^2}{\lambda}|{\bm\xi}|^2,\\
\label{estimate-ellipticity2}
\tx{and}\quad&\lambda c_0^2|{\bm\xi}|^2 \le \sum_{i,j=1}^2 (-1)^{i+j}a^{\vphi}_{ij}({\bf x})\xi_i\xi_j \le \frac{ {c_0}^2}{\lambda}|{\bm\xi}|^2.
\end{align}

\end{itemize}

\begin{proof}
For ${\bf q}\in \R^2$, let us define
\begin{equation*}
  \varrho({\bf q}):=\hat{\rho}(|{\bf q}|^2)
\end{equation*}
for $\hat{\rho}$ from Lemma \ref{lemma-rho-expression}.
For any small $\sigma\in(0, \frac 12 \sqrt{-2H(\rhosonic)})$, if we set $r_{\sigma}:= \sqrt{-2H(\rhosonic)}-\sigma$, then Lemma \ref{lemma-rho-expression} implies that there exists a constant $\mu_{\sigma}>0$ depending only on $(\gam, B_0, \sigma)$ so that $\varrho$ satisfies
\begin{equation}
\label{estimate-varrho}
  \varrho({\bf q})- \rhosonic\ge \frac{1}{\mu_{\sigma}}\,\,\tx{in $\ol{B_{r_{\sigma}}({\bf 0})}$},\quad\tx{and}\quad
  \|\varrho\|_{C^3(\ol{B_{r_{\sigma}}({\bf 0})})}\le \mu_{\sigma}
\end{equation}
for $H$ and $\rhosonic$ defined by \eqref{definition-H-function} and \eqref{definition-rho-sonic}, respectively.

By Lemmas \ref{lemma-shock-polar2} and \ref{lemma-rho-expression}, there exists $\sigma_0\in (0, \frac 12 \sqrt{-2H(\rhosonic)})$ such that
\begin{equation*}
|\nabla\psi_0|\le \sqrt{-2H(\rhosonic)}-4\sigma_0.
\end{equation*}
Therefore, if $M_2\mathfrak{q}_{\gam}(\eps)\le 2\sigma_0$, then, for any $\vphi\in \mcl{I}^f_{M_2,L}$,  we have
\begin{equation}
\label{uniform-subsonicity}
\sup_{{\rm x}\in \Om_{f,L}}|\nabla\vphi({\rm x})+\nabla\psi_0|\le  \sqrt{-2H(\rhosonic)}-2\sigma_0.
\end{equation}
So one can directly check from \eqref{definition-coefficients}, \eqref{estimate-varrho} and \eqref{uniform-subsonicity} that if $M_2\mathfrak{q}_{\gam}(\eps)\le 2\sigma_0$, then for any $\vphi\in \mcl{I}^f_{M_2,L}$, we have
\begin{equation}
\label{coefficient-variation}
  \|a^{\vphi}_{ij}-a^{0}_{ij}\|_{1,\alp, \Om_{f,L}}^{(1-\beta;-\alp,\{P_{2,f}^L\})}\le C\|\vphi\|_{2,\alp, \Om_{f,L}}^{(-\beta;-(1+\alp),\{P_{2,f}^L\})}\le CM_2\mathfrak{q}_{\gam}(\eps).
\end{equation}
This proves the statement (a) if we choose $\sigma_{\rm el}$ to satisfy $\sigma_{\rm el}\le 2\sigma_0$.

\smallskip

By \eqref{definition-coefficients}, we get
\begin{equation*}
  \sum_{i,j=1}^2  a_{ij}^0\xi_i\xi_j
= (\rho_{\rm st}^{\eps})^{\gam-1}|\bm\xi|^2\quad\tx{for any $\bm\xi=(\xi_1,\xi_2)\in \R^2$}.
\end{equation*}

Note that $(\rho_{\rm st}^0)^{\gam-1}=(\gam-1)B_0$ and that $\rho_{\rm st}^{\eps}$ continuously varies depending on $\eps\in [0,\eps_2]$ with $\rho_{\rm st}^{\eps}<\rho_{\rm st}^0$ for $\eps>0$. So there exists a constant $\lambda_0\in(0,1)$ depending only on $(\gam, B_0)$ such that
$\lambda_0 \rho_{\rm st}^0 \le \rho_{\rm st}^{\eps} \le \rho_{\rm st}^0$ for all $\eps\in[0,\eps_2]$.
By combining this estimate with \eqref{coefficient-variation}, one can find a small constant $\delta_1\in(0, 2\sigma_0]$ depending only on $(\gam, B_0)$ so that if $M_2\mathfrak{q}_{\gam}(\eps)\le \delta_1$ holds, then for any $\vphi\in \mcl{I}_{M_2,L}^f$, the coefficient matrix $[a_{ij}^{\vphi}]_{i,j=1}^2$ satisfies \eqref{estimate-ellipticity} for some constant $\lambda\in(0,1)$ depending only on $(\gam, B_0)$ provided that $\sigma_{\rm el}\le \delta_1$.

Since
$\displaystyle{
\sum_{i,j=1}^2 (-1)^{i+j}a^{\vphi}_{ij}({\bf x})\xi_i\xi_j =
\sum_{i,j=1}^2 a^{\vphi}_{ij}({\bf x})\xi_i\xi_j -4a^{\vphi}_{12}({\bf x})\xi_1\xi_2,
}$ and $a_{12}^0=0$, we can apply \eqref{coefficient-variation} to reduce $\sigma_{\rm el}>0$ and adjust $\lambda\in(0,1)$ depending only on $(\gam, B_0)$ so that the estimate \eqref{estimate-ellipticity2} holds. This proves the statement (b).

\end{proof}

\end{lemma}

Under Condition \ref{condition-for-parameters}, let us fix $\vphi\in \mcl{I}_{M_2, L}^f$, and consider the following linear boundary value problem:
\begin{equation}
\label{lbvp-phi-cutoff}
\begin{split}
\mcl{L}^{\vphi}(\phi)=0&\qquad\tx{in} \quad \Om_{f,L}\\
\phi=\ipsi-\psi_0&\qquad\tx{on}\quad\Gam_{\rm{sh},L}^f\\
\phi=-\psi_0&\qquad\tx{on} \quad\Gam_{\rm{sym}}\cup \Gam_{b,L}^f \\
\nabla\phi\cdot{\bf n}_c=0&\qquad\tx{on}\quad \Gam_{\rm{cutoff},L}^f.
\end{split}
\end{equation}


\begin{proposition}
\label{wellposedness-lbvp-phi-cutoff}
Fix $\beta\in (0,1)$ in \eqref{definition-iterset-shocks} and \eqref{definition-iterset-psi}.
Assume that Condition \ref{condition-for-parameters} holds.
And, assume that constants $(M_2, \eps)$ satisfy
\begin{equation}
\label{wp-conditions}
M_2\ge M_2^*,\quad\tx{and}\quad M_2\mathfrak{q}_{\gam}(\eps)\le \sigma_{\rm el}
\end{equation}
so that Lemmas \ref{lemma-vphi-iterset-nonempty} and  \ref{lemma-coefficients} hold.

Then,
there exist positive constants $L_{\rm wp}>1$ depending on $d_0$, $(\hat{\eps}, \hat{\sigma})$ with $\hat{\sigma}\le \sigma_{\rm el}$ depending only on $(\gam, B_0)$, and $\hat{\alp}\in(0,1)$ depending only on $\tw$ so that
\begin{itemize}
\item[(i)] if
\begin{equation}
\label{condition-for-lbvp-wp}
\begin{split}
L\ge L_{\rm{wp}},\quad 0<\eps<\hat{\eps},\quad M_2\mathfrak{q}_{\gam}(\eps)\le \hat{\sigma};
\end{split}
\end{equation}
\item[(ii)] and if $\mcl{I}_{M_2,L}^f$ is given by \eqref{definition-iterset-psi} with $\alp=\hat{\alp}$,
\end{itemize}
then, for any $\vphi\in \mcl{I}^f_{M_2,L}$, the linear boundary value problem \eqref{lbvp-phi-cutoff} associated with $\vphi$ has a unique solution $\phi\in C^{2,\hat{\alp}}_{(-\beta;-(1+\hat{\alp}),\{ P_{2,L}^f\})}(\Om_{f,L})$.
Furthermore, there exists a constant $C>0$ depending only on $(\gam, B_0, d_0)$ to satisfy
\begin{equation}
\label{apriori-phi-lbvp-cutoff}
\|\phi\|_{2,\hat{\alp}, \Om_{f,L}}^{(-\beta;-(1+\hat{\alp}),\{P_{2,L}^f\})}
\le C\left(1+M_1\mathfrak{q}_{\gam}(\eps)\right)\mathfrak{q}_{\gam}(\eps).
\end{equation}
Most importantly, the constants $(C, \hat{\alp})$ are independent of $(L,\beta)$.
\end{proposition}

The rest of \S \ref{subsection-fixedbd-nlbvp-psi} is devoted to proving Proposition \ref{wellposedness-lbvp-phi-cutoff}. In order to prove this proposition, we first establish various estimates of solutions $\phi$ to \eqref{lbvp-phi-cutoff} uniformly with respect to $\vphi\in \mcl{I}_{M_2,L}^f$ and $f\in \mcl{J}_{M_1, L}$ through Lemmas \ref{lemma-L-infty-estimate}--\ref{lemma-estimate-phi-corner}. Furthermore, all the estimates established below are independent of $L$ provided that $L$ is sufficiently large. Finally, we prove Proposition \ref{wellposedness-lbvp-phi-cutoff} at the end of \S \ref{subsection-fixedbd-nlbvp-psi} by the method of continuity.

\begin{lemma}
\label{lemma-L-infty-estimate}
Under the same assumptions as in Proposition \ref{wellposedness-lbvp-phi-cutoff},
if $\phi\in C^1(\ol{\Om_{f,L}})\cap C^2(\Om_{f,L})$ is a solution to \eqref{lbvp-phi-cutoff}, then it satisfies that
\begin{equation}\label{barrier-function}
  |\phi({\bf x})|\le a_0\mathfrak{q}_{\gam}(\eps) (1+x_2)^{\beta}\quad\tx{in $\Om_{f,L}$}
\end{equation}
for a constant $a_0>0$ depending only on $(\gam, B_0, d_0)$.
Moreover, the constant $a_0$ is chosen independently of $L\in [\underline{L}+10(1+h_0),\infty)$ and $\beta\in(0,1)$.

\begin{proof}
It can be directly checked from \eqref{definition-ipsi} and the first equation in \eqref{RH-spolar-Prelim} that
\begin{equation*}
\frac{d}{dx_2}(\ipsi-\psi_0)(f_0(x_2),x_2)=0 \quad\tx{for all $x_2>0$}.
\end{equation*}
Furthermore, we obtain from \eqref{definition-ipsi} and \eqref{definition-background} that $(\ipsi-\psi_0)(f_0(0),0)=0$ so we get
\begin{equation}\label{background-fbc}
(\ipsi-\psi_0)(f_0(x_2),x_2)=0\quad\tx{for all $x_2\ge 0$}.
\end{equation}
We obtain from \eqref{definition-ipsi}, \eqref{definition-background} and \eqref{background-fbc} that
\begin{equation}
\label{bc-shock-expression}
\begin{split}
   (\ipsi-\psi_0)(f(x_2),x_2)
   &=(\ipsi-\psi_0)(f(x_2),x_2)-(\ipsi-\psi_0)(f_0(x_2),x_2)\\
   &=\rho^{\eps}_{\rm{st}}u^{\eps}_{\rm{st}}\kappa_w (f-f_0)(x_2).
\end{split}
\end{equation}
By Lemma \ref{lemma-shock-polar2} and \eqref{definition-iterset-shocks}, one can find a constant $C_1>0$ depending only on $(\gam, B_0)$ such that
\begin{equation}
\label{lbvp-bdry-estimate-sh}
  |(\ipsi-\psi_0)(f(x_2),x_2)|\le C_1M_1\mathfrak{q}_{\gam}^2(\eps)(1+x_2)^{\beta}\quad\tx{for all $x_2\in [0,L]$}.
\end{equation}
By using Definition \ref{definition-bluntbody-ftn}, Lemma \ref{lemma-shock-polar2}, \eqref{definition-background} and \eqref{psi0-on-bdry2}, one can also find a constant $C_2>0$ depending only on $(\gam, B_0)$ such that
\begin{equation}
\label{lbvp-bdry-estimate-wallsym}
  |\psi_0({\bf x})|\le C_2(b_0+d_0+h_0)\mathfrak{q}_{\gam}(\eps) \quad\tx{for all ${\bf x}\in \Gam_{\rm{sym}}\cup \Gam_{b,L}^f $}.
\end{equation}

For a constant $a_0>0$ to be determined later, let us set
\begin{equation*}
  w({\bf x}):=a_0\mathfrak{q}_{\gam}(\eps)(1+x_2)^{\beta}.
\end{equation*}
Since $0<\beta<1$, Lemma \ref{lemma-coefficients}(b) implies that if \eqref{condition3-for-epsilon} holds, then we have
\begin{equation}
\label{cp-1}
  \mcl{L}^{\vphi}(w)=
  a_{22}^{\vphi}\beta(\beta-1)a_0\mathfrak{q}_{\gam}(\eps)(1+x_2)^{\beta-1}
  \le \lambda {c_0^2}\beta(\beta-1)a_0\mathfrak{q}_{\gam}(\eps)(1+x_2)^{\beta-1}<0
  \quad\tx{in $\Om_{f,L}$}.
\end{equation}
For the rest of the proof, we continue to assume that \eqref{condition3-for-epsilon} hold.

We choose the constant $a_0$ in the definition of $w$ as
\begin{equation}
\label{choice-A}
  a_0=\max\{C_1\sigma_2, C_2(b_0+d_0+h_0)\},
\end{equation}
for the constant $\sigma_2$ from \eqref{shock-iterset-condition4}. Note that the constant $a_0$ has been chosen depending only on $(\gam, B_0, d_0)$. But the choice of $a_0$ is independent of $L\in [\underline{L}+10(1+h_0),\infty)$ and $\beta\in(0,1)$.
 It follows from Condition \ref{condition-for-parameters}(iv), \eqref{lbvp-bdry-estimate-sh} and \eqref{lbvp-bdry-estimate-wallsym} that
\begin{equation}
\label{cp-2}
  |\phi({\bf x})|\le w({\bf x})\quad\tx{for all ${\bf x}\in \der \Om_{f,L}\setminus \Gam_{\rm{cutoff}, L}^f$}.
\end{equation}

On $\Gam_{\rm{cutoff}, L}^f$, we have
\begin{equation}
\label{cp-3}
  \der_{{\bf n}_c}w=a_0\mathfrak{q}_{\gam}(\eps)\beta (1+x_2)^{\beta-1}\sin\tw>0.
\end{equation}

From \eqref{cp-1}--\eqref{cp-3}, the comparison principle and the linearity of the boundary value problem \eqref{lbvp-phi-cutoff}, we finally obtain the estimate \eqref{barrier-function} for the constant $a_0$ given by \eqref{choice-A}.

\end{proof}
\end{lemma}

\begin{lemma}
\label{lemma-estimate-nr-sym}
For each $r>0$, define
\begin{equation*}
\mcl{D}_r:=(\Om_{f,L}\cup \Gam_{\rm{sym}})
\setminus\{{\bf x}\in \Om_{f,L}:x_2\ge r\}.
\end{equation*}
Under the same assumptions as in Proposition \ref{wellposedness-lbvp-phi-cutoff},
if $\phi\in C^1(\ol{\Om_{f,L}})\cap C^2(\Om_{f,L})$ is a solution to \eqref{lbvp-phi-cutoff}, then there exists a constant $C_r>0$ depending only on $(\gam, B_0, d_0,\alp, r)$ such that
\begin{equation}
\label{estimate-phi-nr-sym}
  \|\phi\|_{2,\alp, \mcl{D}_{\frac{3r}{4}}}
  \le C_r\mathfrak{q}_{\gam}(\eps).
\end{equation}

\begin{proof}
Fix a constant $r\in(0,L]$.
We define an extension of the fixed function $f\in \mcl{J}_{M_1,L}$ for $x_2<0$ by
\begin{equation*}
f_{\rm ext}(x_2)=\begin{cases}
f(x_2)\quad\mbox{for $x_2\ge 0$}\\
f(-x_2)\quad\mbox{for $x_2<0$}
\end{cases}.
\end{equation*}
By \eqref{definition-iterset-shocks}, we have $f'(0)=0$. Therefore, the extended function $f_{\rm{ext}}$ is $C^{2,\alp}$ across $x_2=0$, and satisfies
\begin{equation}
\label{estimate-bdries-nr-sym}
\|f_{\rm ext}\|_{2,\alp,[-r,r]}\le 2\|f\|_{2,\alp,[0, r]}.
\end{equation}
Since the blunt body function $b$ is defined on $\R$, and satisfies the properties ($b_1$)--($b_5$) stated in Definition \ref{definition-bluntbody-ftn}, we can define
\begin{equation*}
\begin{split}
\Gam_{{\rm{sh}}, L}^{f,\rm{ext}}:=\{x_1=f_{\rm ext}(x_2):-L\le x_2\le L\},\quad \Gam_{b}^{\rm{ext}}=\{x_1=b(x_2):x_2\in \R\}.
\end{split}
\end{equation*}
And, let $\mcl{D}_r^{\rm{ext}}$ be the domain obtained from extending $\mcl{D}_r$ into $\R\times (-\infty,0)$ by even reflection about $x_1$-axis.
In $\mcl{D}_{r}^{\rm{ext}}$, we define extensions $a_{ij}^{\vphi,\rm ext}$ of the coefficients $a_{ij}^{\vphi}$ by
\begin{equation}
\label{definition-coefficients-extension}
\begin{split}
&a_{ii}^{\vphi,\rm ext}(x_1,x_2)=\begin{cases}
a_{ii}^{\vphi}(x_1,x_2), & \mbox{if } x_2\ge 0\\
a_{ii}^{\vphi}(x_1,-x_2), & \mbox{if } x_2<0
\end{cases} \quad \tx{for $i=1,2$},\\
&a_{12}^{\vphi,\rm ext}(x_1,x_2)=\begin{cases}
a_{12}^{\vphi}(x_1,x_2), & \mbox{if } x_2\ge 0\\
-a_{12}^{\vphi}(x_1,-x_2), & \mbox{if } x_2<0
\end{cases}.
\end{split}
\end{equation}

If $\phi\in C^1(\ol{\Om_{f,L}})\cap C^2(\Om_{f,L})$ solves \eqref{lbvp-phi-cutoff}, then the function $\psi:=\phi+\psi_0$ solves
\begin{equation}
\label{lbvp-phi-cutoff-equiv}
\begin{split}
\mcl{L}^{\vphi}(\psi)=0&\qquad\tx{in} \quad \Om_{f,L}\\
\psi=\ipsi&\qquad\tx{on}\quad\Gam_{\rm{sh},L}^f\\
\psi=0&\qquad\tx{on} \quad\Gam_{\rm{sym}}\cup \Gam_{b,L}^f \\
\nabla\psi\cdot{\bf n}_c=0&\qquad\tx{on}\quad \Gam_{\rm{cutoff},L}^f.
\end{split}
\end{equation}
The following linear boundary value problem in $\mcl{D}_{r}^{\rm{ext}}$ is given as an extension of \eqref{lbvp-phi-cutoff-equiv}.
\begin{equation}
\label{lbvp-reflected-nr-sym}
\begin{split}
&\sum_{i,j=1}^2 a_{ij}^{\vphi, {\rm{ext}}}\der_{ij}\til{\psi}=0\qquad\tx{in $\mcl{D}_{r}^{\rm{ext}}$},\\
&\til{\psi}=0\qquad\tx{on $\der \mcl{D}_{r}^{\rm{ext}}\cap \Gam_b^{\rm{ext}}$},\\
&\til{\psi}=\ipsi \qquad\tx{on\quad $\der \mcl{D}_{r}^{\rm{ext}}\cap \Gam_{\rm{sh},L}^{f,\rm{ext}}$},\\
&\til{\psi}(x_1,x_2)=\begin{cases}\psi(x_1,x_2)&\quad\tx{for $(x_1,x_2)\in\der\mcl{D}_{r}^{\rm{ext}}\cap \Om_{f,L}$}\\
-\psi(x_1,-x_2)&\quad\tx{for $(x_1,-x_2)\in\der\mcl{D}_{r}^{\rm{ext}}\cap \Om_{f,L}$}
\end{cases}.
\end{split}
\end{equation}
Since $\psi$ satisfies $\psi=0$ on $\Gam_{\rm sym}$, the Dirichlet boundary condition of \eqref{lbvp-reflected-nr-sym} is continuous on $\der\mcl{D}_{r}^{\rm{ext}}$. Then the unique existence of a solution $\til{\psi}\in C^0(\ol{\mcl{D}_{r}^{\rm{ext}}})\cap C^{2,\alp}(\mcl{D}_{r}^{\rm{ext}})$ is achieved by Lemma \ref{lemma-coefficients}(b) and the standard elliptic theory(cf. Gilbarg-Trudinger \cite{GT}). Furthermore, if $\til{\psi}(x_1,x_2)$ solves \eqref{lbvp-reflected-nr-sym}, then the uniqueness of a solution implies that $-\til{\psi}(x_1, -x_2)$ solves \eqref{lbvp-reflected-nr-sym} as well, from which it follows that $\til{\psi}$ is an odd function with respect to $x_2$. So, it becomes the solution to \eqref{lbvp-phi-cutoff-equiv}. Therefore, it suffices to estimate $C^{2,\alp}$ norm of $\til{\psi}$ near $\Gam_{\rm sym}$ in order to get $C^{2,\alp}$ estimate of $\psi$ near $\Gam_{\rm sym}$.

Due to Lemma \ref{lemma-coefficients}(b), we have
\begin{equation*}
\lambda c_0^2|\bm\xi|^2\le \sum_{i,j=1}^2 a_{ij}^{\vphi, {\rm ext}}(\rx)\xi_i\xi_j\le \frac{c_0^2}{\lambda}|\bm\xi|^2
\end{equation*}
for any $\rx\in \mcl{D}_{r}^{\rm ext}$, ${\bm\xi}=(\xi_1,\xi_2)\in \R^2$ so that the comparison principle applies to \eqref{lbvp-reflected-nr-sym}. Then we adjust the proof of Lemma \ref{lemma-L-infty-estimate} to obtain that
\begin{equation*}
  |\til{\psi}(\rx)|\le a_0\mathfrak{q}_{\gam}(\eps)(1+x_2)^{\beta}+|\ipsi(\rx)|\quad\tx{in $\mcl{D}_{r}^{\rm ext}$}
\end{equation*}
for $a_0$ given by \eqref{choice-A}.
By \eqref{definition-q-gamma} and \eqref{definition-ipsi}, this estimate yields a constant $C>0$ depending only on $(\gam, B_0, d_0)$ to satisfy
\begin{equation}
\label{Linfty-estimate-psi-nr-sym}
 \sup_{\rx\in \mcl{D}_{r}^{\rm ext}}|\til{\psi}(\rx)|\le C(1+r)\mathfrak{q}_{\gam}(\eps).
\end{equation}



By the standard Schauder estimate with using Lemma \ref{lemma-distance-iterset}, \eqref{estimate-bdries-nr-sym} and \eqref{Linfty-estimate-psi-nr-sym}, one can find a constant $C^*_r>0$ depending only on $(\gam, B_0, d_0, \alp, r)$ such that
\begin{equation}
\label{estimate-tpsi-nr-sym}
\|\til{\psi}\|_{2,\alp, \mcl{D}_{\frac{3r}{4}}^{\rm{ext}}}\le C^*_r\mathfrak{q}_{\gam}(\eps).
\end{equation}
Since $\til{\psi}=\psi$ in $\mcl{D}_{\frac{3r}{4}}$, it follows from \eqref{definition-q-gamma}, \eqref{definition-ipsi} and \eqref{estimate-tpsi-nr-sym} that there exists a constant $C_r>0$ depending only on $(\gam, B_0, d_0, \alp, r)$ to satisfy the estimate \eqref{estimate-phi-nr-sym}.

\end{proof}
\end{lemma}

In order to obtain a priori weighted $C^{2,\alp}$-estimate of $\phi$ away from the symmetric boundary $\Gam_{\rm sym}$ up to the cut-off boundary $\Gam^f_{{\rm cutoff}, L}$ independently of $L$, we introduce a new coordinate system.
Let us define ${\bm\eta}=(\eta_1,\eta_2)$ by
\begin{equation}
\label{definition-cov}
  \begin{pmatrix}
    \eta_1 \\
    \eta_2
  \end{pmatrix}
  =\begin{pmatrix}
\sin \tw & -\cos \tw \\
     \cos \tw & \phantom{-}\sin \tw
   \end{pmatrix}\begin{pmatrix}
                  x_1 \\
                  x_2
                \end{pmatrix}=:\mcl{R}\begin{pmatrix}
                  x_1 \\
                  x_2
                \end{pmatrix}.
\end{equation}
In ${\bm\eta}$-coordinates, the $\eta_2$-axis is along the line $x_2=x_1\tan\theta_w$(Figure \ref{figure_coordinate}).
\begin{figure}[htp]
	\centering
	\begin{psfrags}
		\psfrag{lt}[cc][][0.8][0]{$P_{2,L}^f$}
		\psfrag{rt}[cc][][0.8][0]{${\bf{P_{3,L}^f}}$}
		\psfrag{rb}[cc][][0.8][0]{$\rb$}
		\psfrag{lb}[cc][][0.8][0]{$\lb$}
		\psfrag{sm}[cc][][0.8][0]{$\Gam_{\rm{sym}}$}
		\psfrag{gw}[cc][][0.8][0]{$\Gam_{b,L}^f$}
		\psfrag{cf}[cc][][0.8][0]{$\phantom{aa}\Gam_{{\rm{cutoff}},L}^f$}
		\psfrag{sh}[cc][][0.8][0]{$\Gam_{{\rm{sh}},L}^f$}
		\psfrag{O}[cc][][0.8][0]{$0$}
		\psfrag{om}[cc][][0.8][0]{$\Om_{f,L}$}
		\psfrag{Uf}[cc][][0.8][0]{$M_{\infty}=\frac{1}{\eps}$}
\psfrag{t2}[cc][][0.8][0]{$\phantom{aa}\theta_{P_{2,L}^f}$}
\psfrag{a1}[cc][][0.8][0]{$\eta_1$}
\psfrag{a2}[cc][][0.8][0]{$\eta_2$}
\psfrag{h}[cc][][0.8][0]{\phantom{aaaa}${\bf{\eta_2=\frac{h_1}{\sin \tw}}}$}
		\includegraphics[scale=1.2]{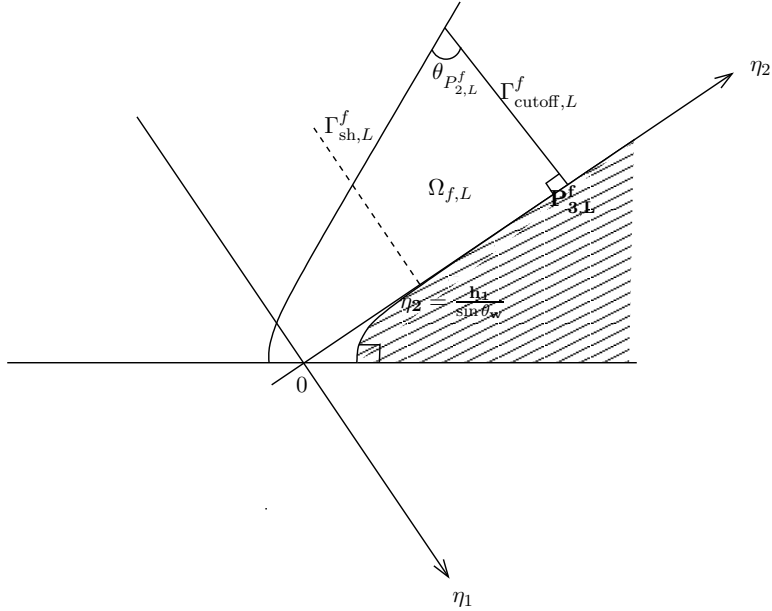}
		\caption{$(\eta_1,\eta_2)$-coordinates}\label{figure_coordinate}
	\end{psfrags}
\end{figure}

By \eqref{definition-iterset-shocks}, \eqref{estimate-f-iterset} and \eqref{shock-iterset-condtion1},  we have
\begin{equation}
\label{estimate-x2}
\begin{split}
&f(x_2)=f(0)+\int_0^{x_2}f'(s)ds>b_0-d_0-\frac{x_2}{4}\tan \tw.
\end{split}
\end{equation}
Therefore, for $h_1$ defined by
\begin{equation}
\label{choice-h1}
  h_1=\begin{cases}
  -\frac{4(b_0-d_0)}{\tan \tw}\quad&\mbox{if $b_0-d_0<0$}\\
  h_0\quad&\mbox{if $b_0-d_0\ge 0$}
  \end{cases},
\end{equation}
we have
\begin{equation}
\label{condition-for-h1}
-\frac{x_2}{2}\tan \tw  < f(x_2) \le x_1 \le b(x_2)=x_2\cot \tw\quad\tx{for ${\bf x}\in \ol{\Om_{f,L}}\cap\{x_2\ge h_1\}$},
\end{equation}
where we apply the property ($b_5$) stated in Definition \ref{definition-bluntbody-ftn} to get the last equality in \eqref{condition-for-h1}.
By a direct computation with using \eqref{definition-cov} and \eqref{condition-for-h1}, we get
\begin{equation}
\label{equiv-x-eta}
\frac{\sin \tw}{2}x_2   \le  \eta_2\le \frac{2}{\sin \tw}x_2\quad\tx{for ${\bf x}\in \ol{\Om_{f,L}}\cap\{x_2\ge h_1\}$}.
\end{equation}

Later, we will choose the constant $L_*\in [\underline{L}+10(1+h_0+h_1), \infty)$
so that the set $\Om_{f,L}\cap\{x_2\ge h_1\}$ is nonempty whenever $L\ge L_*$.
\smallskip

Let us set
\begin{equation*}
h_1^*:=\frac{h_1}{\sin \tw},
\end{equation*}
and define
\begin{equation}
\Lambda'_{f,L}:= \{(\eta_1,\eta_2)^T=\mcl{R}\begin{pmatrix}
                  x_1 \\
                  x_2
                \end{pmatrix}: (x_1, x_2)\in \Om_{f,L}\}\cap \{\eta_2\ge h_1^*\}. \quad {\tx{(see Fig. \ref{figure_coordinate})}}
\end{equation}
In $\Lambda'_{f,L}$, we define
\begin{equation}
\label{definition-tildes}
  \til{\phi}({\bm\eta}):=\phi({\bf x}),\quad
  \til{\phi}_{\infty}({\bm\eta}):=(\psi_{\infty}-\psi_0)({\bf x})
\end{equation}
for the functions $\ipsi$ and $\psi_0$ given by \eqref{definition-ipsi} and \eqref{definition-background}, respectively.
Fix a point $P\in \Lambda'_{f,L}$. For each $j=1,2$, let $\eta_j^P$ represent the $\eta_j$-coordinate of the point $P$. Similarly, let $x_j^P$ represent the $x_j$-coordinate of the point $P$. By \eqref{definition-cov}, we have
\begin{equation}
\label{x2p}
  x_2^P=-\eta_1^P\cos\theta_w+\eta_2^P\sin\theta_w.
\end{equation}
Since we have
\begin{equation}
\label{negativity-eta1}
 \eta_1^P\leq 0\quad\tx{and}\quad \eta_2^P\ge h_1^*\quad\tx{for $P\in \Lambda'_{f,L}$},
\end{equation}
\eqref{x2p} implies that $x_2^P\ge h_1$ (see Fig. \ref{figure_coordinate}). Then we obtain from \eqref{equiv-x-eta} that
\begin{equation*}
\frac{1}{m_0} x_2^P \le  \eta_2^P\le m_0 x_2^P
\quad\tx{with $m_0=\frac{2}{\sin \tw}$.}
\end{equation*}
By combining this inequality with \eqref{barrier-function}, we get
\begin{equation}
\label{Linfty-estimate-tphi}
  |\til{\phi}({\bm\eta})|\le a_0m_0^{\beta}\mathfrak{q}_{\gam}(\eps)(1+\eta_2)^{\beta}\quad \tx{in $\Lambda'_{f,L}$}.
\end{equation}

\smallskip

Let us define
\begin{equation}
\label{definition-Lambda-fL}
\Lambda_{f,L}:= \Lambda'_{f,L}\cap \{\eta_2\ge 2h_1^*\}.
\end{equation}
By using \eqref{definition-cov}, we rewrite the linear boundary value problem \eqref{lbvp-phi-cutoff} restricted in $\Lambda_{f,L}$ in $(\eta_1, \eta_2)$-coordinates as follows:
\begin{equation}
\label{lbvp-cov}
  \begin{split}
      A^{\vphi}_{11}\til{\phi}_{\eta_1\eta_1}+
      2A^{\vphi}_{12}\til{\phi}_{\eta_1\eta_2}
     +A^{\vphi}_{22}\til{\phi}_{\eta_2\eta_2}=0 \quad & \tx{in $\Lambda_{f,L}$} \\
     \til{\phi}=\til{\phi}_{\infty} \quad & \tx{on $\Gam_{{\rm{sh}},L}^f\cap\der \Lambda_{f,L}$}\\
      \til{\phi}=-\rho_{\rm{st}}^{\eps}u_{\rm{st}}^{\eps}\kappa_w(b_0-d_0) \quad & \tx{on $\Gam_{b,L}^f\cap\der \Lambda_{f,L}$} \\
      \til{\phi}_{\eta_2}=0 \quad & \tx{on $\Gam_{\rm{cutoff},L}^f$}
  \end{split}
\end{equation}
where we use the fact that $x_2-\kappa_wx_1=0$ on $\Gam_{b,L}^f\cap\{\eta_2>2h_1^*\}$.
In \eqref{lbvp-cov}, the coefficients $A^{\vphi}_{ij}$ for $i,j=1,2$ are given by
\begin{equation}
\label{coeff-eta-coord}
A^{\vphi}_{ij}({\bm\eta})
=c^{2}(|\nabla\vphi+\nabla\psi_0|^2)\delta_{ij}
-(-1)^{i+j}\frac{(\vphi+\psi_0)_{\eta_{i'}}(\vphi+\psi_0)_{\eta_{j'}}}
{\hat{\rho}^2(|\nabla\vphi+\nabla\psi_0|^2)}
\end{equation}
for $\eta_{1'}=\eta_2$ and $\eta_{2'}=\eta_1$.

The following lemma is directly obtained from Lemma \ref{lemma-coefficients}.

\begin{lemma}
\label{lemma-A-vphi}
For each $\vphi\in \mcl{I}_{M_2,L}^f$, the coefficients $\{A_{ij}^{\vphi}\}_{i,j=1}^2$ given by \eqref{coeff-eta-coord} satisfy the following properties:
\begin{itemize}
\item[(a)] There exists a constant $C>0$ depending only on $(\gam, B_0,  d_0)$ so that for each $\vphi\in \mcl{I}_{M_2,L}^f$, we have
    \begin{equation}\label{estimate-A-ij}
      \|A_{ij}^{\vphi}-A_{ij}^0\|_{1,\alp, \Lambda_{f,L}}^{(1-\beta;-\alp,\{P_{2,L}^f\})}\le CM_2\mathfrak{q}_{\gam}(\eps)
    \end{equation}
    where the norm $\|\cdot\|_{1,\alp, \Lambda_{f,L}}^{(1-\beta;-\alp,\{P_{2,L}^f\})}$ is defined by Definition \ref{definition-norms-phi-decay-corner} with $(x_1,x_2)$ being replaced by $(\eta_1, \eta_2)$, respectively;

\item[(b)] For any $\bm\xi=(\xi_1,\xi_2)\in \R^2$, and ${\bm\eta}\in \Lambda_{f,L}$, we have
\begin{equation}\label{ellipticity-A-ij}
\lambda c_0^2|{\bm\xi}|^2 \le \sum_{i,j=1}^2 A^{\vphi}_{ij}({\bm\eta})\xi_i\xi_j \le \frac{ {c_0}^2}{\lambda}|{\bm\xi}|^2\quad\tx{with $ c_0^2=(\gam-1)B_0$}
\end{equation}
for some $\lambda\in(0,1)$ depending only on $(\gam, B_0)$.
\end{itemize}
\end{lemma}

In ${\bm\eta}$-coordinates, we have
\begin{equation*}
\Gam_{\rm{sh}, L}^f \cap \der \Lambda'_{f,L}
=\{(\eta_1,\eta_2):F(\eta_1,\eta_2)=0\,\,\tx{for $2h_1^* \le \eta_2\le |P_{3,L}^f|$}\},
\end{equation*}
for
\begin{equation*}
 F(\eta_1,\eta_2)
 :=(\eta_1\sin\tw+\eta_2\cos\tw)-f(-\eta_1\cos\tw+\eta_2\sin\tw).
\end{equation*}
Since $F_{\eta_1}=(\tan \tw+f')\cos \tw\ge \frac{3}{4}\sin \tw>0$ by \eqref{estimate-f-iterset} and Condition \ref{condition-for-parameters}(iv), the implicit function theorem implies that there exists a unique function $\til{f}:[h_1^*, |P_{3,L}^f|]\rightarrow \R$ satisfying
\begin{equation}
\label{definition-shock-etacoord}
\Gam_{\rm{sh}, L}^f \cap \der \Lambda'_{f,L}=\{(\til{f}(\eta_2), \eta_2): h_1^*\le \eta_2\le |P_{3,L}^f|\}.
\end{equation}

\begin{lemma}
There exists a constant $m_1>1$ depending only on $d_0$ such that
\begin{equation}
\label{bounds-ftilde}
-m_1\eta_2\le \til{f}(\eta_2) \le -\frac{1}{m_1}\eta_2 \quad \tx{for $ 2h_1^* \le \eta_2 \le |P_{3,L}^f|$}.
\end{equation}

\begin{proof}
By differentiating $F(\til{f}(\eta_2), \eta_2)\equiv0$ with respect to $\eta_2$, and using \eqref{definition-cov} and \eqref{definition-shock-etacoord} , we get
\begin{equation}
\label{estimate1-gf}
\til{f}'(\eta_2)= -\frac{\cot \tw-f'(x_2)}{1+f'(x_2)\cot \tw}\,\,\tx{for $x_2=-\til f(\eta_2)\cos \tw+\eta_2\sin \tw$}.
\end{equation}
And, we use \eqref{estimate-f-iterset} and Condition \ref{condition-for-parameters}(iv) to obtain from
\eqref{estimate1-gf} that
\begin{equation}
\label{estimate-ftilde-derivative}
-\frac 53 \cot \tw\le \til{f}'(\eta_2) \le -\frac 35 \cot \tw\quad\tx{for $h_1^* \le \eta_2\le |P_{3,L}^f|$}.
\end{equation}
This estimate implies that
\begin{equation}
\label{estimate-ftilde-zerorder}
  -\frac 53 \cot \tw(\eta_2-h_1^*) \le \til{f}(\eta_2)-\til{f}(h_1^*)\le -\frac 35\cot \tw(\eta_2-h_1^*)\quad
  \tx{for $h_1^* \le \eta_2\le |P_{3,L}^f|$.}
\end{equation}
By \eqref{negativity-eta1}, we have $\til{f}(h_1^*)\le 0$, thus \eqref{estimate-ftilde-zerorder} implies that
\begin{equation}
\label{estimate-ftilder-lwrbd}
  \til{f}(\eta_2)<-\frac 35\cot \tw(\eta_2-h_1^*) \le -\frac{3}{10} \eta_2\cot \tw \quad\tx{for $2h_1^*\le \eta_2 \le |P_{3,L}^f|$}.
\end{equation}

A direct computation with using \eqref{condition-for-h1} shows that
\begin{equation*}
  \til{f}(h_1^*)\ge -(\tan \tw+\frac{2}{\tan \tw})h_1^*=:-\om.
\end{equation*}
Since $h_1$ is given by \eqref{choice-h1}, the constant $\om>0$ depends only on $d_0$. By combining this estimate with the first inequality in \eqref{estimate-ftilde-zerorder}, it is obtained that
\begin{equation}
\label{estimate-ftilder-uprbd}
\til{f}(\eta_2)\ge -\left(\frac 53 \cot \tw+\frac{\om}{2h_1^*}\right)\eta_2=:-\om_1 \eta_2 \quad \tx{for $ 2h_1^*\le \eta_2 \le |P_{3,L}^f|$}.
\end{equation}

From \eqref{estimate-ftilder-lwrbd} and \eqref{estimate-ftilder-uprbd}, one can fix a constant $m_1>1$ depending only on $d_0$ so that \eqref{bounds-ftilde} holds. This proves the lemma.

\end{proof}
\end{lemma}

To simplify a notation, set
\begin{equation}
\label{definition-Ltilde}
\til{L}:=|P_{3,L}^f|.
\end{equation}
Then $\Lambda_{f,L}$ given by \eqref{definition-Lambda-fL} is represented as
\begin{equation*}
\Lambda_{f,L}:=\Om_{f,L}\cap\{{\bm{\eta}}=(\eta_1,\eta_2): \til{f}(\eta_2)<\eta_1<0,\,\,\eta_2\in(2h_1^*, \til{L})\}.
\end{equation*}

\smallskip

For each $R\in[4h_1^*, \til{L}]$, define a domain $Q_R$ by
\begin{equation*}
Q_R=\{{\bm \zeta}=(\zeta_1,\zeta_2)\in \R^2:\frac 12<\zeta_2<1,\,\, R{\bm \zeta}\in \Lambda_{f, L}\}.
\end{equation*}
Then, $Q_R$ can be represented as
\begin{equation}
Q_R=\{(\zeta_1,\zeta_2):\in \R^2: \frac 12 <\zeta_2<1,\,\, {F}^R(\zeta_2)<\zeta_1<0\}.
\end{equation}
for
\begin{equation}\label{definition-scaled-f}
  F^R(\zeta_2):=\frac 1R \til{f}(R\zeta_2).
\end{equation}
By \eqref{bounds-ftilde}, for any $R\in[4h_1^*, \til{L}]$, $F_R$ satisfies
\begin{equation}
\label{estimate-gfr-lwrbd}
-m_1<{F}^R(\zeta_2)<-\frac{1}{2m_1}\quad\tx{for all $\frac 12<\zeta_2<1$.}
\end{equation}

For a priori estimates of solutions to \eqref{lbvp-cov} up to $\Gam_{{\rm{cutoff}},L}^f$ uniform with respect to $L\ge L_*$, we define another weighted H\"{o}lder norms of functions in $(\zeta_1,\zeta_2)$-variables.

\begin{definition}
\label{Definition-weighted-Holder-zeta}
Fix $m,k\in \{0\}\cup \mathbb{N}$ with $k\le m$, and $\alp\in (0,1)$.
\begin{itemize}
\item[(i)] Let $(s,t)$ be an interval with $-\infty<s<t<\infty$. For $\zeta_2, \zeta_2'\in (s,t)$, set
    \begin{equation*}
      d_{\zeta_2}^{(1)}:=t-\zeta_2,\quad
      d_{\zeta_2, \zeta_2'}^{(1)}:=\min\{d_{\zeta_2}^{(1)}, d_{\zeta_2'}^{(1)}\}.
    \end{equation*}
    For a function $F:(s,t)\rightarrow \R$, define
    \begin{equation*}
      \begin{split}
      &\|F\|_{m,(s,t)}^{(-(k+\alp), \{t\})}
      :=\|F\|_{k,\alp, (s,t)}
      +\sum_{j>k}\sup_{\zeta_2\in(s,t)}
      (d_{\zeta_2}^{(1)})^{j-(k+\alp)}
      |\frac{d^j}{d\zeta_2^j}F(\zeta_2)|,\\
      &[F]_{m,\alp,(s,t)}^{(-(k+\alp), \{t\})}
      :=\sup_{\zeta_2\neq \zeta_2'\in (s,t)}
      (d_{\zeta_2,\zeta_2'}^{(1)})^{m-k}
      \frac{|\frac{d^m}{d\zeta_2^m}F(\zeta_2)-\frac{d^m}{d\zeta_2^m}F(\zeta_2')|}
      {|\zeta_2-\zeta_2'|^{\alp}},\\
      &\|F\|_{m,\alp, (s,t)}^{(-(k+\alp),\{t\})}
      :=\|F\|_{m,(s,t)}^{(-(k+\alp), \{t\})}+[F]_{m,\alp,(s,t)}^{(-(k+\alp), \{t\})}.
      \end{split}
    \end{equation*}

    \item[(ii)] Let $Q$ be an open, bounded and connected domain in $\R^2$, and let ${\bf q}$ be a fixed point on  $\der Q$. For ${\bm \zeta}=(\zeta_1,\zeta_2),{\bm \zeta}'=(\zeta_1', \zeta_2')\in Q$, set
\begin{equation*}
 d^{(2)}_{{\bm \zeta}}:= |{\bm\zeta}-{\bf q}|,\quad d^{(2)}_{{\bm \zeta},{\bm\zeta'}}:=\min\{d^{(2)}_{{\bm \zeta}}, d^{(2)}_{{\bm \zeta}'} \}.
\end{equation*}
For a function $u:Q\rightarrow \R$, define
\begin{equation*}
\begin{split}
&\|u\|_{m, Q}^{(-(k+\alp),\{{\bf q}\})} :=\|u\|_{k,\alp, Q}+
\sum_{j>k}^m \sup_{{\bm \zeta}\in Q}(d^{(2)}_{{\bm \zeta}})^{ j-(k+\alp)}\sum_{l=0}^j
|\der^l_{\zeta_1}\der^{j-l}_{\zeta_2}u({\bm \zeta})|,\\
&[u]_{m,\alp, Q}^{(-(k+\alp),\{{\bf q}\})} :=
\sup_{{\bm \zeta}\neq {\bm \zeta}'\in Q}
(d^{(2)}_{{\bm \zeta},{\bm\zeta'}})^{m-k}\sum_{l=0}^m
\frac{|\der_{\zeta_1}^l\der_{\zeta_2}^{m-l}u({\bm \zeta})-\der_{\zeta_1}^l\der_{\zeta_2}^{m-l}u({\bm \zeta}')|}{|{\bm \zeta}-{\bm \zeta}'|^{\alp}},\\
&\|u\|_{m,\alp,Q}^{(-(k+\alp),\{{\bf q}\})}:=
\|u\|_{m, Q}^{(-(k+\alp),\{{\bf q}\})} + [u]_{m,\alp, Q}^{(-(k+\alp),\{{\bf q}\})} .
\end{split}
\end{equation*}
\end{itemize}
\end{definition}

\begin{lemma}
\label{lemma-estimate-Fr}
There exists a constant $h_2^*\in[h_1^*, \infty)$ and $C>0$ depending only on $d_0$ so that if $\til L\ge 10(1+h_2^*)$, then $F^R$ satisfies the following estimates:
\begin{itemize}
\item[(a)] For $4h_2^* \le  R \le \frac 45 \til{L}$,
\begin{equation}
\label{estimate-Fr-scaling-1}
\begin{split}
  \|{F}^R\|_{C^{2,\alp}([\frac 12, 1])}&\le C(1+M_1q_{\gam}(\eps));
  \end{split}
\end{equation}
\item[(b)] For $R=\til{L}$,
\begin{equation}
\label{estimate-Fr-scaling-2}
  \|F^{\til L}\|_{2,\alp, (\frac 12 , 1)}^{(-(1+\alp),\{1\})}\le C(1+M_1q_{\gam}(\eps)).
\end{equation}
\end{itemize}

\begin{proof}
{\textbf{1.}}
It directly follows from \eqref{estimate-f-iterset} and \eqref{definition-scaled-f} that if $4h_1^*\le R\le \til{L}$, then we have
\begin{equation}
\label{estimate-FR-1st-deriv-new}
\|(F^R)'\|_{C^0([\frac 12, 1])}\le 2\cot \tw.
\end{equation}

By \eqref{equiv-x-eta}, there exists a constant $h_2^*\in[h_1^*, \infty)$  depending only on $d_0$ so that
\begin{equation}\label{assumption-equiv-weight}
  \tx{if $\til{L}\ge 10(1+h_2^*)$, $\eta_2\in [h_2^*, \til{L}]$, and if $x_2$ is given by \eqref{estimate1-gf},}
\end{equation}
then we have
\begin{equation}
\label{equiv-weight1}
  \om_0 \le \frac{1+\eta_2}{1+x_2} \le \frac{1}{\om_0}
\end{equation}
for some constant $\om_0\in(0,1)$ depending only on $\tw$. By \eqref{definition-cov}, we have $|\til L-\eta_2|=|L-x_2|\sin\tw$ for $x_2$ given by \eqref{estimate1-gf}. For $\eta_2, \eta_2'\in [h_2^*, \til{L}]$, let $x_2, x_2'$ be given by  \eqref{estimate1-gf} corresponding to $\eta_2, \eta_2'$, respectively. Let us define
\begin{equation*}
\begin{split}
  &\delta_{\eta_2}^{(1)}:=\min\{1+\eta_2, \til{L}-\eta_2\},\quad \delta_{\eta_2, \eta_2'}^{(1)}:=\min\{ \delta_{\eta_2}^{(1)}, \delta_{\eta_2'}^{(1)}\},\\
  &\hat{\delta}_{x_2}^{(1)}:=\min\{1+x_2, L-x_2\},\quad \hat{\delta}_{x_2, x_2'}^{(1)}:=\min\{ \delta_{x_2}^{(1)}, \delta_{x_2'}^{(1)}\}.
  \end{split}
\end{equation*}
Then, we use \eqref{equiv-weight1} to obtain
\begin{equation}
\label{equiv-weight2}
\begin{split}
&\om_1  \le \frac{\delta_{\eta_2}^{(1)}}{1+\eta_2}\cdot \frac {1+x_2}{\hat{\delta}_{x_2}^{(1)}}\le \frac{1}{\om_1},\\
&\om_1  \le \frac{\delta_{\eta_2,\eta_2'}^{(1)}}{1+\max\{\eta_2, \eta_2'\}}\cdot \frac {1+\max\{x_2, x_2'\}}{\hat{\delta}_{x_2, x_2'}^{(1)}}\le \frac{1}{\om_1}
\end{split}
\end{equation}
for some constant $\om_1\in(0,1)$ depending only on $\tw$ under the condition of \eqref{assumption-equiv-weight}.

Since $(F^R)''(\zeta_2)=R\til f''(R\zeta_2)$ due to \eqref{definition-scaled-f},
we compute $\til f''$ in terms of $f$ by differentiating \eqref{estimate1-gf} with respect to $\eta_2$. A direct computation yields that
\begin{equation}
\label{f-tilde-2nd-deriv}
\til f''(\eta_2)=f''(x_2)(\sin \tw-f'(x_2)\cos \tw)=(f''(x_2)-f_0''(x_2))(\sin \tw-f'(x_2)\cos \tw)
\end{equation}
for $x_2=-\til f(\eta_2)\cos \tw+\eta_2\sin \tw$, because $f''_0\equiv 0$.
\smallskip

{\textbf{2.}}
Suppose that $4h_2^* \le  R \le \frac 45 \til{L}$. Since $f\in \mcl{J}_{M_1, L}$ for $\mcl{J}_{M_1, L}$ given by \eqref{definition-iterset-shocks}, a direct computation with using \eqref{equiv-weight2} and \eqref{f-tilde-2nd-deriv} yields that
\begin{equation}
\label{estimate-FR-away-cutoff}
\|(F^R)''\|_{C^0([\frac 12, 1])}\le CM_1q_{\gam}(\eps)R^{-1+\beta},\quad\tx{and}\quad
[(F^{R})'']_{\alp, (\frac 12, 1)}\le CM_1q_{\gam}(\eps)R^{-(1+\alp)+\beta}
\end{equation}
for $C>0$ depending only on $d_0>0$. So the estimate \eqref{estimate-Fr-scaling-1} is obtained by combining \eqref{estimate-FR-away-cutoff} with \eqref{estimate-gfr-lwrbd} and \eqref{estimate-FR-1st-deriv-new}.
\smallskip

{\textbf{3.}}
Suppose that $R=\til {L}$.
By the definition of $f_0$ given in \eqref{definition-background-f}, $f_0'$ is a constant. Then, we use \eqref{definition-iterset-shocks}, \eqref{estimate1-gf}, \eqref{definition-scaled-f} and \eqref{equiv-weight1} to obtain that, for any $\zeta_2, \zeta_2'\in(\frac 12, 1)$ with $\zeta_2\neq \zeta_2'$,
\begin{equation}
\label{estimate-FL-1new}
\begin{split}
\frac{|(F^{\til{L}})'(\zeta_2)-(F^{\til{L}})'(\zeta_2')|}{|\zeta_2-\zeta_2'|^{\alp}}
&=\frac{|\til f'(\til L\zeta_2)-\til f'(\til{L}\zeta_2')|}{|\til{L}(\zeta_2-\zeta_2')|^{\alp}}\til L^{\alp}\\
&\le  C\|f-f_0\|_{1,\alp, (0,L)}^{(-\beta)}\til L^{-1+\beta}\\
&\le CM_1q_{\gam}(\eps)\til L^{-1+\beta}.
\end{split}
\end{equation}
 By using \eqref{equiv-weight1}--\eqref{f-tilde-2nd-deriv}, it can be directly checked that for any $\zeta_2, \zeta_2'\in(\frac 12, 1)$ with $\zeta_2\neq \zeta_2'$,
\begin{equation}
\label{estimate-FL-2new}
(1-\zeta_2)^{1-\alp}|(F^{\til L})''(\zeta_2)|
+\min\{1-\zeta_2, 1-\zeta_2'\}\frac{|(F^{\til L})''(\zeta_2)-(F^{\til L})''(\zeta_2')|}{|\zeta_2-\zeta_2'|^{\alp}}
\le CM_1q_{\gam}(\eps)\til L^{-1+\beta}.
\end{equation}
In the estimates \eqref{estimate-FL-1new} and \eqref{estimate-FL-2new}, the constant $C>0$ is given depending only on $d_0>0$. The estimate \eqref{estimate-Fr-scaling-2} is obtained by combining \eqref{estimate-FL-1new} and \eqref{estimate-FL-2new} with \eqref{estimate-gfr-lwrbd} and \eqref{estimate-FR-1st-deriv-new}.
\end{proof}
\end{lemma}
\smallskip

For $R\in [4h_2^*, \til{L}]$, let us define
\begin{equation}
\label{definition-scaled-functions}
u^R({\bm \zeta})=\frac{\til{\phi}(R{\bm \zeta})}{R^{\beta}}\quad\tx{for $\bm\zeta=(\zeta_1, \zeta_2)\in Q_R$.}
\end{equation}
It directly follows from \eqref{Linfty-estimate-tphi} and \eqref{lbvp-cov} that, for each $R\in [4h_2^*, \til{L}]$, $u^R$ satisfies
\begin{equation}\label{estimate-uR-Linfty}
  \sup_{{\bm\zeta}\in Q_R}|u^R(\bm\zeta)|\le a_0(2m_0)^{\beta}\mathfrak{q}_{\gam}(\eps)
\end{equation}
for $m_0=\frac{2}{\sin \tw}$ and $a_0$ from Lemma \ref{lemma-L-infty-estimate}, and
\begin{equation}
\label{lbvp-scaled-problem1}
  \begin{split}
\mcl{L}^{\vphi}_{R} (u^R):=\sum_{i,j=1}^2 A_{ij}^{\vphi, R}u^R_{\zeta_i\zeta_j}=0&\quad\mbox{in $Q_R$}\\
u^R=-\frac{\rho_{\rm{st}}^{\eps}u_{\rm{st}}^{\eps}
\kappa_w(b_0-d_0)}{R^{\beta}}&\quad\mbox{on $\der Q_R\cap\{\zeta_1=0\}=:\Sigma_{b}^R$}\\
u^R=\frac{1}{R^{\beta}}\til{\phi}_{\infty}(\til{f}(R\zeta_2), R\zeta_2)=:g_{R}(\zeta_2)&\quad\mbox{on $\der Q_R\cap\{\zeta_1={F}^R(\zeta_2)\}=:\Sigma^R_{\rm{sh}}$}
  \end{split}
\end{equation}
for $A_{ij}^{\vphi, R}({\bm\zeta}):=A_{ij}^{\vphi}(R{\bm\zeta})$.

We will estimate (weighted) $C^{2,\alp}$ norms of $u^R$ for two cases:  (i) $4h_2^*\le R\le \frac 45 \til{L}$, (ii) $R=\til{L}$.

By Lemma \ref{lemma-A-vphi}, we have
\begin{equation}
\label{estimate-coeff-scaling-2}
\begin{split}
&\lambda c_0^2|{\bm\xi}|^2\le \sum_{i,j=1}^2A_{ij}^{\vphi,R}({\bm\zeta})\xi_i\xi_j\le \frac{c_0^2}{\lambda} |\bm{\xi}|^2\quad\mbox{for all ${\bm\zeta}=(\zeta_1,\zeta_2)\in \ol{Q_R}$, ${\bm\xi}=(\xi_1,\xi_2)\in \R^2$},\\
&\|A_{ij}^{\vphi, R}-A_{ij}^{0, R}\|_{\alp, Q_R}\le \frac{\lambda_0}{ R^{1-\beta}}M_2\mathfrak{q}_{\gam}
(\eps) \quad\tx{\quad\tx{for all $i,j=1,2$}}
\end{split}
\end{equation}
for the constants $(\lambda, c_0, C)$ from Lemma \ref{lemma-A-vphi}.
This implies that, for each $R\in[4h_1^*, \til{L}]$, the equation $\mcl{L}_R^{\vphi}(u^R)=0$ is uniformly elliptic, and its coefficients $\{A_{ij}^{\vphi, R}\}_{i,j=1}^2$ are in $C^{\alp}(\ol{Q_R})$.

\begin{lemma}
\label{lemma-estiamte-uR-1}
Fix $\beta\in(0,1)$ and $\alp\in(0,1)$ in \eqref{definition-iterset-shocks} and \eqref{definition-iterset-psi}. Assume the same conditions as Proposition \ref{wellposedness-lbvp-phi-cutoff}. Suppose that $\til{L}\ge 10(1+h_2^*)$ for $h_2^*$ from Lemma \ref{lemma-estimate-Fr}. If $\phi\in C^1(\ol{\Om_{f,L}})\cap C^2(\Om_{f,L})$ is a solution to \eqref{lbvp-phi-cutoff}, then there exists a constant $C>0$ depending only on $(d_0,\alp)$ so that, for each $R\in [4h_2^*, \frac 45 \til{L}]\cup \{\til{L}\}$, $u_R$ satisfies
\begin{equation}
\label{estimate-ur-away-cutoff-final}
  \|u^R\|_{2,\alp, Q_R\cap\{\frac{11}{20}<\zeta_2<\frac{19}{20}\}}
  \le C(1+M_1\mathfrak{q}_{\gam}(\eps))\mathfrak{q}_{\gam}(\eps)
  .
\end{equation}

\begin{proof}
By \eqref{bc-shock-expression},
we can rewrite the function $g_R(\zeta_2)$ in \eqref{lbvp-scaled-problem1} as
\begin{equation}
\label{gR_representation}
g_R(\zeta_2)=\frac{1}{R^{\beta}}\rho_{\rm st}^{\eps}u_{\rm st}^{\eps}\kappa_w(f-f_0)(x_2)\quad\tx{for $x_2=-\til{f}(R\zeta_2)\cos \tw+R\zeta_2 \sin \tw$}.
\end{equation}
By directly differentiating this representation with using \eqref{estimate1-gf} and \eqref{f-tilde-2nd-deriv}, we get
\begin{equation}
\label{gR_derivatives}
\begin{split}
&g_R'(\zeta_2)=R^{1-\beta}\rho^{\eps}_{\rm st}u_{\rm st}^{\eps}\kappa_w\frac{(f-f_0)'(x_2)}{\sin \tw+f'(x_2)\cos \tw},\\
&g_R''(\zeta_2)=R^{2-\beta}\rho^{\eps}_{\rm st}u_{\rm st}^{\eps}\kappa_w
\frac{1-(\sin \tw+f'(x_2)\cos \tw)^3(f-f_0)'(x_2)}{(\sin \tw+f'(x_2)\cos \tw)^2}
(f-f_0)''(x_2)
\end{split}
\end{equation}
 for $x_2$ given by \eqref{gR_representation}. By applying Lemma \ref{lemma-shock-polar2}, \eqref{definition-iterset-shocks} and \eqref{equiv-weight1}--\eqref{equiv-weight2}, it can be directly derived from \eqref{gR_derivatives} that
 \begin{itemize}
\item[(i)] For $4h_2^* \le  R \le \frac 45 \til{L}$,
\begin{equation}
\label{estimate-gr-away-cf}
  \|g_R\|_{C^{2,\alp}([\frac 12, 1])}\le CM_1\mathfrak{q}_{\gam}^2(\eps);
\end{equation}
\item[(ii)] For $R=\til{L}$,
\begin{equation}
 \label{estimate-gr-nr-cf}
 \|g_R\|_{2,\alp, (\frac 12, 1)}^{(-(1+\alp),\{1\})}
  \le CM_1\mathfrak{q}_{\gam}^2(\eps)
\end{equation}
\end{itemize}
for a constant $C>0$ depending only on $d_0>0$.

For each $R\in [4h_2^*, \frac 45\til{L}]\cup\{\til{L}\}$, $u^R$ satisfies \eqref{lbvp-scaled-problem1}, and $\mcl{L}^{\vphi}_R$ is uniformly elliptic in $Q_R$ due to \eqref{estimate-coeff-scaling-2}. Then, by the standard Schauder estimate theory with using \eqref{estimate-uR-Linfty}, \eqref{estimate-coeff-scaling-2}, Lemma \ref{lemma-estimate-Fr}, and the estimates \eqref{estimate-gr-away-cf}--\eqref{estimate-gr-nr-cf}, the estimate \eqref{estimate-ur-away-cutoff-final} is obtained.

\end{proof}
\end{lemma}

Set
\begin{equation*}
 {\bf{q}}_{*}:= (F^{\til{L}}(1),1).
\end{equation*}
In addition to \eqref{lbvp-scaled-problem1}, $u^{\til{L}}$ satisfies the boundary condition
\begin{equation}
\label{lbvp-u-tilde-L}
\der_{\zeta_2}u^{\til{L}}=0\quad\tx{on $\der Q_{\til{L}}\cap \{\zeta_2=1\}:=\Sigma_{\rm{cutoff}}^{\til{L}}$}
\end{equation}
due to the slip boundary condition $\nabla \phi\cdot{\bf n}_c=0$ on $\Gam_{\rm{cutoff}, L}^f$ in \eqref{lbvp-phi-cutoff}. Since the boundary $\Sigma_b^{\til{L}}$ is perpendicular to $\Sigma_{\rm{cutoff}}^{\til{L}}$ at ${\bm\zeta}=(0,1)$, and $u^{\til{L}}$ is a constant on $\Sigma_b^{\til{L}}$, we can apply the method of reflection and apply standard Schauder estimate theory to $u^{\til{L}}$ up to $\Sigma_{\rm{cutoff}}^{\til{L}}$ away from ${\bf q}_*$. So the following lemma is obtained.
\begin{lemma}
\label{lemma-estiamte-uR-2}
Assume the same conditions as in Lemma \ref{lemma-estiamte-uR-1}. If $\phi$ is in $C^{2,\alp}$ up to $\Gam^f_{{\rm cutoff}, L}$ away from $P_{2, L}^f$, then, for any constant $r_0\in(0, \frac{1}{20})$, there exists a constant $C_{r_0}>0$ depending only on $(d_0,\alp, r_0)$ so that $u^{\til{L}}$ satisfies the estimate
\begin{equation}\label{estimate-uL-reflection}
\|u^{\til{L}}\|_{2,\alp, (Q_{\til{L}}\cap\{\frac{11}{20}<\zeta_2<1\})\setminus B_{r_0/4}( {\bf{q}}_{*})}\le
C_{r_0}(1+M_1\mathfrak{q}_{\gam}(\eps))
\mathfrak{q}_{\gam}(\eps).
\end{equation}

\end{lemma}

To estimate a weighted $C^{2,\alp}$ norm of $u^{\til{L}}$ near  ${\bf{q}}_{*}$, we first find $\hat{\alp}\in(0,1)$ so that a priori $C^{1,\hat{\alp}}$ estimate of $u^{\til{L}}$ can be achieved up to ${\bf q}_*$, then we get a priori weighted $C^{2,\hat{\alp}}$ estimate of $u^{\til{L}}$ near ${\bf q}_*$ by using a scaling argument. As we shall see below, the choice of $\hat{\alp}$ essentially depends on $\tw$.

\begin{lemma}
\label{lemma-estiamte-uR-3}
Under the same conditions as in Lemma \ref{lemma-estiamte-uR-1}, there exist constants $h_3^*\in[h_2^*, \infty)$, $r_1\in(0, \frac 12)$,  $\hat{\eps}>0$, $\hat{\sigma}>0$ and $\hat{\alp}\in(0,1)$ so that if
\begin{itemize}
\item[(i)] $\til{L}\ge 10(1+h_3^*)$;
\item[(ii)] $\alp=\hat{\alp}$ in \eqref{definition-of-shock-iterset} and \eqref{definition-iterset-psi};
\item[(iii)] $
0<  \eps\le \hat{\eps}$;
\item[(iv)] $M_2\mathfrak{q}_{\gam}(\eps)\le \hat{\sigma}$,
\end{itemize}
then $u^{\til{L}}$ satisfies the estimate
\begin{equation}
\label{uL-corner-estimate}
\|u^{\til{L}}\|_{2,\hat{\alp}, Q_{\til{L}} \cap B_{r_1/2}({\bf q}_*)}^{(-1-\hat{\alp}, \{{\bf q}_*\})}\le
C(1+M_1\mathfrak{q}_{\gam}(\eps))\mathfrak{q}_{\gam}(\eps)
\end{equation}
for some constant $C>0$. Here, $(h_3^*, r_1, C)$ depend only on $(\gam, B_0, d_0)$. The choice of $(\hat{\eps}, \hat{\sigma})$ depends only on $(\gam, B_0)$, and the choice of $\hat{\alp}$ depends only on $\tw$. But, the choice of $(\hat{\eps}, \hat{\sigma},\hat{\alp})$ is independent of $d_0$ and $L$.

\begin{proof} The proof is divided into 5 steps.

{\textbf{1.}}
Back to $(x_1,x_2)$-coordinates, the unit tangential $\bm\tau_{f}=(\tau_1, \tau_2)$ to $\Gam_{{\rm sh},L}^f$ at $P_{2,L}^f$ with $\bm\tau_{f}\cdot {\bf e}_2>0$ is expressed as
\begin{equation*}
  \bm\tau_{f}:=\frac{(f'(L),1)}{\sqrt{(f'(L))^2+1}}.
\end{equation*}

Set
\begin{equation}
\label{definition-a-vector}
{\bf a}:=(\der_{x_1},\der_{x_2}) \phi(P_{2,L}^f).
\end{equation}
By differentiating the boundary condition $\phi=\ipsi-\psi_0$ on $\Gam_{{\rm sh}, L}^f$  along the tangential direction at $P_{2,L}^f$, we get $\bm\tau_{f} \cdot {\bf a}= \der_{ \bm\tau_{f}}(\ipsi-\psi_0)(P_{2,L}^f)$. And, the slip boundary condition $\nabla\phi\cdot{{\bf n}_c}=0$ on $\Gam_{{\rm{cutoff}}, L}^f$ can be rewritten as $(\cos \tw, \sin \tw)\cdot {\bf a}=0$.
We rewrite these two equations for ${\bf a}$ in the form of
\begin{equation}
\label{system-for-a}
  \begin{pmatrix}
  \tau_1 & \tau_2\\
  \cos \tw &\sin \tw
  \end{pmatrix}
 {\bf a}^T=
  \begin{pmatrix}
\der_{{\bm \tau}_f}(\ipsi-\psi_0)(P_{2,L}^f)\\ 0
  \end{pmatrix}.
\end{equation}
By \eqref{estimate-f-iterset} and \eqref{shock-iterset-condtion1}, the  matrix $\begin{pmatrix}
  \tau_1 & \tau_2\\
  \cos \tw &\sin \tw
  \end{pmatrix}$ is invertible, so we get
  \begin{equation}
  \label{expression-afvec}
    {\bf a}^{T}=\begin{pmatrix}
  \tau_1 & \tau_2\\
  \cos \tw &\sin \tw
  \end{pmatrix}^{-1}
  \begin{pmatrix}
  \der_{{\bm \tau}_f}(\ipsi-\psi_0)(P_{2,L}^f)\\ 0
  \end{pmatrix}.
  \end{equation}
It directly follows from Lemma \ref{lemma-shock-polar2}, \eqref{definition-iterset-shocks} and \eqref{bc-shock-expression} that
\begin{equation}
\label{estimate-bc-deriv}
  |\der_{{\bm \tau}_f}(\ipsi-\psi_0)(P_{2,L}^f)|\le CM_1\mathfrak{q}_{\gam}^2(\eps)(1+L)^{-1+\beta}
\end{equation}
for some constant $C>0$ depending only on $(\gam, B_0, )$. Hereafter, we regard any estimate constant $C$ to be depending only on $(\gam, B_0)$ unless otherwise specified.

We combine \eqref{estimate-bc-deriv} with \eqref{expression-afvec} to get
\begin{equation}\label{estimate-phi-grad-cutoffpt}
  |  {\bf a}|\le CM_1\mathfrak{q}_{\gam}^2(\eps)(1+L)^{-1+\beta}.
\end{equation}

Let us define
\begin{equation}
\label{definition-Phi}
  \Phi({\bf x})=\phi({\bf x})-\phi(P_{2,f}^L)-{\bf a}\cdot ({\bf x}-P_{2,f}^L).
\end{equation}
It directly follows from \eqref{lbvp-phi-cutoff} that $\Phi$ satisfies
\begin{equation*}
  \begin{split}
&\mcl{L}^{\vphi}(\Phi)=0 \quad\tx{in} \quad \Om_{f,L}\\
&\Phi=(\ipsi-\psi_0)(f(x_2),x_2)
-(\ipsi-\psi_0)(P_{2,L}^f)-{\bf a}\cdot ((f(x_2),x_2)-P_{2,L}^f)=:\til{g}(x_2)
\quad\tx{on} \quad\Gam_{\rm{sh},L}^f\\
&\nabla\Phi\cdot{\bf n}_c=0 \quad\tx{on}\quad \Gam_{\rm{cutoff},L}^f.
  \end{split}
\end{equation*}
\smallskip

{\textbf{2.}}
For $(\eta_1, \eta_2)$-coordinates given by \eqref{definition-cov}, set
\begin{equation*}
\begin{split}
&\Psi(\eta_1, \eta_2):=\Phi(x_1,x_2),\\
&g(\eta_2)=\til{g}(x_2)\quad\tx{for $x_2=-\til{f}(\eta_2)\cos \tw+\eta_2 \sin \tw$},\\
&U^{\til{L}}(\zeta_1,\zeta_2):=\frac{1}{\til{L}^{\beta}}\Psi(\til{L}\zeta_1, \til{L}\zeta_2)\quad \tx{for ${\bm\zeta}=(\zeta_1,\zeta_2)\in Q_{\til{L}}$}.
\end{split}
\end{equation*}
By \eqref{Linfty-estimate-tphi}, \eqref{equiv-weight1}, and \eqref{estimate-phi-grad-cutoffpt}, we have
\begin{equation}
\label{estimate-Linfty-utilde}
|U^{\til{L}}(\bm\zeta)|\le C(1+M_1\mathfrak{q}_{\gam}(\eps))\mathfrak{q}_{\gam}(\eps)\quad\tx{for all $\bm\zeta\in Q_{\til{L}}$}.
\end{equation}
 And, $U^{\til{L}}$ satisfies
\begin{equation}
\label{lbvp-U-L}
  \begin{split}
&\mcl{L}^{\vphi}_{\til{L}} (U^{\til{L}})=\sum_{i,j=1}^2 A_{ij}^{\vphi, \til{L}}U^{\til{L}}_{\zeta_i\zeta_j}=0 \quad\mbox{in $Q_{\til{L}}$},\\
&\der_{\zeta_2}U^{\til{L}}=0 \quad\mbox{on $\Sigma_{\rm{cutoff}}^{\til{L}}$},\\
&U^{\til{L}}(F^{\til{L}}(\zeta_2), \zeta_2)=\frac{g(\til{L}\zeta_2)}{\til{L}^{\beta}}
=:G^{\til{L}}(\zeta_2)
\quad\mbox{on $\Sigma^{\til{L}}_{\rm{sh}}$}
  \end{split}
\end{equation}
for $F^{\til{L}}$ given by \eqref{definition-scaled-f}.
The boundary condition of $U^{\til{L}}$ on $\Sigma_{\rm{cutoff}}^{\til{L}}$ stated in \eqref{lbvp-U-L} holds because ${\bf a}\cdot {\bf n}_c=0$ for ${\bf n}_c=(\cos \tw, \sin \tw)$ according to \eqref{system-for-a}.

By \eqref{bc-shock-expression}, we have
\begin{equation}
\label{dG-dzeta}
\frac{d G^{\til{L}}}{d\zeta_2}=
\til{L}^{1-\beta}\left(\rho_{\rm st}^{\eps}u_{\rm st}^{\eps}\kappa_w(f-f_0)'(x_2)-{\bf a}\cdot (f'(x_2),1)\right)
\end{equation}
for $x_2=-\til{f}(\til{L} \zeta_2)\cos \tw+\til{L}\zeta_2 \sin \tw$.
Note that $G^{\til{L}}(1)=(G^{\til{L}})'(1)=0$.
And, a direct computation with using Lemma \ref{lemma-shock-polar2}, \eqref{definition-iterset-shocks}, \eqref{estimate-ftilde-derivative},
\eqref{equiv-weight1} and \eqref{estimate-phi-grad-cutoffpt} yields that
\begin{equation}
\label{estimate-Gl-1}
 \left \|\frac{d G^{\til{L}}}{d\zeta_2} \right\|_{C^{\alp}([\frac 12, 1])}
  \le  CM_1\mathfrak{q}_{\gam}^2(\eps).
\end{equation}
So we obtain that
\begin{equation}
\label{estimate-Gl-2}
  |G^{\til{L}}(\zeta_2)|\le CM_1\mathfrak{q}_{\gam}^2(\eps)|\zeta_2-1|^{1+\alp}\quad\tx{for all $\zeta_2\in[\frac 12, 1]$.}
\end{equation}

Let $(r,\theta)$ be the polar coordinates centered at ${\bf q}_*$ with $\Sigma_{\rm{cutoff}}^{\til{L}}\subset \{\theta=0\}$. By \eqref{estimate-t2}, \eqref{estimate-FL-1new}, \eqref{estimate-ftilder-lwrbd} and \eqref{estimate-Gl-2}, one can find a sufficiently large constant $h_3^*\in[h_2^*, \infty)$, and a small constant $r_1\in(0, \frac 12)$ depending on $d_0>0$ so that if
\begin{equation}
\label{condition-for-Ltilde-h3}
  \til{L}\ge 10(1+h_3^*),
\end{equation}
then the following properties hold:
\begin{itemize}
\item[(i)] For $a_w=\frac{\pi}{2}-\tw$,
\begin{equation}
\label{condition3-for-L-in-wp}
\Sigma_{{\rm sh}}^{\til{L}}\subset \{-\frac{\pi}{2}+\frac 34 a_w<\theta<-\frac 34 \theta_w\};
\end{equation}
\item[(ii)]
 \begin{equation}
 \label{condition4-for-L-in-wp}
 \{(r,\theta)\in[0, r_1)\times (-\frac 34 \theta_w,0) \}
\subset \left(Q_{\til{L}}\cap B_{r_1}({\bf q}_*)\right)
\subset \{(r,\theta)\in [0, r_1)\times (-\frac{\pi}{2}+\frac 34 a_w, 0)\};
 \end{equation}

\item[(iii)]
 \begin{equation}\label{UL-on-shock}
   |U^{\til{L}}(r,\theta)|\le CM_1\mathfrak{q}^2_{\gam}(\eps) r^{1+\alp} \quad \tx{on $\Sigma^{\til{L}}_{\rm{sh}}\cap B_{r_1}({\bf q}_*)$}.
 \end{equation}
\end{itemize}
\smallskip

{\textbf{3.}}
Set
\begin{equation}
\label{choice-of-mus}
  \mu_1:=\frac{\frac 12(\pi-a_w)}{\frac{\pi}{2}-\frac 34 a_w},\quad\tx{and}\quad
  \mu_0:=\frac{a_w}{4}.
\end{equation}
Then, we have
\begin{equation}
  \label{estimate-lwrbd-V-cosonly}
  \cos(\mu_1\theta-\mu_0)\ge \sin \frac{a_w}{4}\quad \tx{for all $\theta\in(-\frac{\pi}{2}+\frac 34 a_w, 0)$}.
\end{equation}
Next, we define a function $V(r,\theta)$ by
\begin{equation}
\label{definition-V}
  V(r,\theta)=Kr^{1+\alp}\cos(\mu_1\theta-\mu_0)
\end{equation}
for positive constants $K>0$ and $\alp\in(0,1)$ to be determined later.
By \eqref{condition4-for-L-in-wp} and \eqref{estimate-lwrbd-V-cosonly}, $V$ satisfies
\begin{equation}\label{estimate-lwrbd-V}
  V(r,\theta)\ge Kr^{1+\alp}\sin \frac{a_w}{4}\quad\tx{in $Q_{\til{L}}\cap B_{r_1}({\bf q}_*)$}.
\end{equation}
Since $\Sigma_{\rm{cutoff}}^{\til{L}}\subset \{\theta=0\}$, we have
\begin{equation}
  \label{V-on-cutoff-bdry}
  \der_{{\bf n}_c}V=\frac 1r\der_{\theta} V=Kr^{\alp}\mu_1\sin \frac{a_w}{4}>0\quad \tx{on $\Sigma_{\rm{cutoff}}^{\til{L}}\cap B_{r_1}({\bf q}_*)$.}
\end{equation}

By Lemma \ref{lemma-shock-polar2}, \eqref{definition-background}, \eqref{definition-cov}, \eqref{coeff-eta-coord}, \eqref{estimate-coeff-scaling-2} and \eqref{estimate-lwrbd-V-cosonly}, we have
\begin{equation}
\label{expression-Lv}
\begin{split}
  \mcl{L}^{\vphi}_{\til{L}}(V)&=(\rho_{\rm{st}}^{\eps})^{\gam-1}\Delta V
  +\sum_{i,j=1}^2(A_{ij}^{\vphi,\til{L}}-A_{ij}^{0,\til{L}})
  \der_{\zeta_i\zeta_j}V
  -(u_{\rm{st}}^{\eps}\sec \tw)^2\der_{\zeta_2\zeta_2}V\\
  &\le K r^{\alp-1}\left( (\rho_{\rm{st}}^{\eps})^{\gam-1}((\alp+1)^2-\mu_1^2)\cos (\mu_1\theta-\mu_0)+C(M_2+1)\mathfrak{q}_{\gam}(\eps)\right)
  \end{split}
\end{equation}
in $Q_{\til{L}}\cap B_{r_1}({\bf q}_*)$.

Since $\mu_1>1$ by \eqref{choice-of-mus}, we can express $\mu_1$ as
\begin{equation*}
  \mu_1=1+\delta_1\quad\tx{for $\delta_1>0$.}
\end{equation*}
In \eqref{definition-V}, we choose $\alp$ as
\begin{equation}
\label{choice-of-alpha}
\alpha=\frac 12 \min\{1,\delta_1\}=:\hat{\alp}.
\end{equation}
Then we obtain from \eqref{estimate-lwrbd-V-cosonly} and \eqref{expression-Lv} that
\begin{equation*}
   \mcl{L}^{\vphi}_{\til{L}}(V)\le Kr^{\hat{\alp}-1}
   \left(-(\rho_{\rm{st}}^{\eps})^{\gam-1}\delta_1(1+\hat{\alp})\sin \frac{a_w}{4}+C(M_2+1)\mathfrak{q}_{\gam}(\eps)\right)\quad
   \tx{in $Q_{\til{L}}\cap B_{r_1}({\bf q}_*)$. }
\end{equation*}
Since $(\delta_1, \hat{\alp}, a_w, C)$ are fixed independently of $\eps(=\frac{1}{M_{\infty}})>0$, and since $\rho_{\rm st}^{\eps}$ is a continuous function with respect to $\eps$, it follows from \eqref{downstream-case2} and Lemma \ref{lemma-shock-polar2} that
\begin{equation*}
  \lim_{\eps\to 0+} \left(-(\rho_{\rm{st}}^{\eps})^{\gam-1}\delta_1(1+\hat{\alp})\sin \frac{a_w}{4}+C\mathfrak{q}_{\gam}(\eps)\right)
  =-(\mathfrak{h}^{-1}(B_0))^{\gam-1}
  \delta_1(1+\hat{\alp})\sin \frac{a_w}{4}<0
\end{equation*}
Therefore, there exist small constants $\hat{\eps}>0$ and $\hat{\sigma}>0$ depending only on $(\gam, B_0)$ so that if
\begin{equation}
\label{condition7-for-epsilon}
0<  \eps\le \hat{\eps}, \quad\tx{and}\quad
 M_2\mathfrak{q}_{\gam}(\eps)\le \hat{\sigma},
\end{equation}
then we have
\begin{equation}
\label{V-supersol}
  \mcl{L}^{\vphi}_{\til{L}}(V)<0\quad\tx{in $Q_{\til{L}}\cap B_{r_1}({\bf q}_*)$.}
\end{equation}

 By \eqref{estimate-Linfty-utilde}, \eqref{UL-on-shock} and \eqref{estimate-lwrbd-V-cosonly}, we can fix a constant $C_*>0$ so that if $K$ in \eqref{definition-V} is given by
\begin{equation}
\label{choice-K}
K=\frac{C^*(1+M_1\mathfrak{q}_{\gam}(\eps))\mathfrak{q}_{\gam}(\eps)}
{r_1^{1+\hat{\alp}}\sin \frac{a_w}{4}},
\end{equation}
then $U^{\til{L}}$ satisfies
\begin{equation}\label{V-comparion-int-bdry}
  U^{\til{L}}\le V\quad \tx{on $(Q_{\til{L}}\cap \der B_{r_1}({\bf q}_*))\cup (\Sigma_{\rm{sh}}^{\til{L}}\cap B_{r_1}({\bf q}_*))$}.
\end{equation}

For $(\mu_1, \mu_0)$, $\hat{\alp}$ and $K$ given by \eqref{choice-of-mus}, \eqref{choice-of-alpha} and \eqref{choice-K}, respectively, we obtain from \eqref{V-on-cutoff-bdry}, \eqref{V-supersol} and \eqref{V-comparion-int-bdry} that
\begin{equation*}
  \begin{split}
  \mcl{L}_{\til{L}}^{\vphi}(U^{\til{L}}-V)>0\quad &\tx{in $Q_{\til{L}}\cap B_{r_1}({\bf q}_*)$},\\
  \der_{{\bf n}_c}(U^{\til{L}}-V)<0\quad &\tx{on $\Sigma_{\rm{cutoff}}^{\til{L}}\cap B_{r_1}({\bf q}_*)$},\\
  U^{\til{L}}-V\le 0\quad &\tx{on $(Q_{\til{L}}\cap \der B_{r_1}({\bf q}_*))\cup (\Sigma_{\rm{sh}}^{\til{L}}\cap B_{r_1}({\bf q}_*))$}
  \end{split}
\end{equation*}
under the same conditions as in Proposition \ref{wellposedness-lbvp-phi-cutoff}, provided that
\eqref{condition-for-Ltilde-h3} and \eqref{condition7-for-epsilon} hold.
Then, it follows from the maximum principle and the Hopf's lemma that
\begin{equation*}
  U^{\til{L}}\le V\quad\tx{in $\ol{Q_{\til{L}}\cap B_{r_1}({\bf q}_*)}$}.
\end{equation*}
By using \eqref{estimate-Linfty-utilde} and the linearity \eqref{lbvp-U-L}, one can similarly check that  $U^{\til{L}}\ge -V$ in $\ol{Q_{\til{L}}\cap B_{r_1}({\bf q}_*)}$.
Therefore we obtain that
\begin{equation}
\label{estimate-UL-Linfty-growth}
  |U^{\til{L}}(r,\theta)|\le Kr^{1+\hat{\alp}}\quad \tx{in $\ol{Q_{\til{L}}\cap B_{r_1}({\bf q}_*)}$}
\end{equation}
for $K$ given by \eqref{choice-K} and $\hat{\alp}$ given by \eqref{choice-of-alpha}.
\smallskip

{\textbf{4.}}
For a fixed ${\bm\zeta}=(\zeta_1, \zeta_2)\in \ol{Q_{\til{L}}\cap B_{r_1/2}({\bf q}_*)}\setminus \{{\bf q}_*\}$, set
\begin{equation}
\label{definition-d-zeta}
d_{\bm\zeta}:=|{\bm\zeta}-{\bf q}_*|.
\end{equation}
For $\hat{\alp}$ given by \eqref{choice-of-alpha}, we define
\begin{equation*}
\begin{split}
  &w^{({\bm\zeta})}({\bf z}):=\frac{1}{d_{\bm\zeta}^{1+\hat{\alp}}}U^{\til{L}}({\bm\zeta}+\frac {d_{\bm\zeta}}{2} {\bf z})
  \quad\tx{in $\{{\bf z}=(z_1, z_2)\in B_1({\bf 0}): {\bm\zeta}+\frac{d_{\bm\zeta}}{2} {\bf z}\in Q_{\til{L}}\cap B_{r_1}({\bf q}_*)\}$}=:\mcl{B}_{1}^{\bm\zeta},\\
  &\mathfrak{f}^{(\bm\zeta)}(z_2):=\frac{2}{d_{\bm\zeta}}
(F^{\til{L}}(\zeta_2+\frac{d_{\bm\zeta}}{2}z_2)-\zeta_1).
  \end{split}
\end{equation*}
Then we have
\begin{equation}
\label{wz-estimate-1}
 \sup_{{\bm\zeta}\in \ol{Q_{\til{L}}\cap B_{r_1/2}({\bf q}_*)}\setminus \{{\bf q}_*\}} \|w^{({\bm\zeta})}\|_{C^0(\ol{B_1^{\bm\zeta}})}\le 4K
\end{equation}
for $K$ given by \eqref{choice-K}. And, by \eqref{lbvp-U-L}, $w^{({\bm\zeta})}$ satisfies
\begin{equation*}
  \begin{split}
  &\sum_{i,j=1}^2 A_{ij}^{\vphi, \til{L}}({\bm\zeta}+\frac{d_{\bm\zeta}}{2}{\bf z})
  \der_{z_iz_j}w^{{\bm\zeta}}({\bf z})=0\quad\tx{in $B_1^{\bm\zeta}$},\\
  &\der_{z_2} w^{({\bm\zeta})}=0\quad\tx{on $\der B_1^{\bm\zeta}\cap\{{\bf z}: {\bm\zeta}+\frac{d_{\bm\zeta}}{2}{\bf z}\in \Sigma_{\rm{cutoff}}^{\til{L}}\}$}=:\der B_{1,{\rm cutoff}}^{\bm\zeta},\\
  &w^{(\bm\zeta)}(\mathfrak{f}^{(\bm\zeta)}(z_2), z_2)=\frac{1}{d_{\bm\zeta}^{1+\hat{\alp}}}
  G^{\til{L}}(\zeta_2+\frac{d_{\bm\zeta}}{2}z_2)
  =:\mathfrak{g}^{(\bm\zeta)}(z_2)
  \quad\tx{on $\der B_1^{\bm\zeta}\cap\{{\bf z}: {\bm\zeta}+\frac{d_{\bm\zeta}}{2}{\bf z}\in \Sigma_{\rm{sh}}^{\til{L}}\}=:\der B_{1,{\rm sh}}^{\bm\zeta}$}.
  \end{split}
\end{equation*}

For each ${\bm\zeta}=(\zeta_1, \zeta_2)\in \ol{Q_{\til{L}}\cap B_{r_1/2}({\bf q}_*)}\setminus \{{\bf q}_*\}$, set $\mcl{I}_{\rm sh}^{\bm\zeta}:=\{z_2\in \R: (\mathfrak{f}^{(\bm\zeta)}(z_2), z_2)\in\der B_{1,{\rm sh}}^{\bm\zeta} \}$. Then, a direct computation with using Lemma \ref{lemma-shock-polar2}, \eqref{estimate-phi-grad-cutoffpt}, \eqref{dG-dzeta} and \eqref{estimate-Gl-2} yields that
\begin{equation}
\label{wz-estimate-2}
  \|\mathfrak{g}^{(\bm\zeta)}\|
  _{C^{2,\hat{\alp}}(\ol{\mcl{I}_{\rm sh}^{\bm\zeta}})}\le CM_1\mathfrak{q}_{\gam}^2(\eps)
\end{equation}
provided that $\mcl{I}_{\rm sh}^{\bm\zeta}$ is nonempty. Then, by Lemma \ref{lemma-estimate-Fr}(b) and \eqref{estimate-coeff-scaling-2}, we can apply a standard Schauder estimate theory to obtain from \eqref{choice-K}, \eqref{wz-estimate-1} and \eqref{wz-estimate-2} that
\begin{equation*}
  \|  w^{({\bm\zeta})}\|_{2,\hat{\alp}, \mcl{B}_{1/2}^{\bm\zeta}}\le C(1+M_1\mathfrak{q}_{\gam}(\eps))\mathfrak{q}_{\gam}(\eps)
  \quad\tx{for all ${\bm\zeta}\in \ol{Q_{\til{L}}\cap B_{r_1/2}({\bf q}_*)}\setminus \{{\bf q}_*\}$},
\end{equation*}
and this yields the estimate
\begin{equation}\label{estimate-UL-C2}
  \|U^{\til{L}}\|_{2,\hat{\alp}, Q_{\til{L}}\cap B_{r_1/2}({\bf q}_*)}
  ^{(-1-\hat{\alp},\{{\bf q}_*\})}
  \le C(1+M_1\mathfrak{q}_{\gam}(\eps))\mathfrak{q}_{\gam}(\eps).
\end{equation}
In the estimates of $w^{({\bm\zeta})}$ and $U^{\til{L}}$, the estimate constants $C$ may be different from each other, but both constants can be chosen depending only on $d_0$.
\smallskip

{\textbf{5.}}
By \eqref{definition-scaled-functions}, \eqref{lbvp-u-tilde-L}, \eqref{definition-a-vector} and \eqref{definition-Phi}, we have
\begin{equation*}
\|u^{\til{L}}\|_{2,\hat{\alp}, Q_{\til{L}}\cap B_{r_1/2}({\bf q}_*)}^{(-1-\hat{\alp}, \{{\bf q}_*\})}
\le  \|U^{\til{L}}\|_{2,\hat{\alp}, Q_{\til{L}}\cap B_{r_1/2}({\bf q}_*)}
  ^{(-1-\hat{\alp},\{{\bf q}_*\})}
  +\frac{1}{\til{L}^{\beta}}|\til{\phi}_{\infty}
  (\til{f}(\zeta_2^{{\bf q}_*}),\zeta_2^{{\bf q}_*})|
  +\til{L}^{1-\beta}|{\mathbf{a}}|
\end{equation*}
where $\zeta_2^{{\bf q}_*}$ denotes $\zeta_2$-coordinate of ${\bf q}_*$.
Finally, we apply \eqref{estimate-gr-nr-cf}, \eqref{estimate-phi-grad-cutoffpt} and \eqref{estimate-UL-C2} to obtain that
\begin{equation}
\label{uL-corner-estimate}
\|u^{\til{L}}\|_{2,\hat{\alp}, Q_{\til{L}} \cap B_{r_1/2}({\bf q}_*)}^{(-1-\hat{\alp}, \{{\bf q}_*\})}\le
C(1+M_1\mathfrak{q}_{\gam}(\eps))\mathfrak{q}_{\gam}(\eps)
\end{equation}
for some constant $C>0$.

We fix $(\alp, r_0)$ to be $(\alp, r_0)=(\hat{\alp}, r_1)$ in the estimate \eqref{estimate-uL-reflection}, then combine the resultant estimate with \eqref{uL-corner-estimate} to conclude that
\begin{equation}
\label{uL-corner-estimate-final}
\|u^{\til{L}}\|_{2,\hat{\alp}, Q_{\til{L}} \cap \{\zeta_2>\frac{11}{20}\}}^{(-1-\hat{\alp}, \{{\bf q}_*\})}\le C(1+M_1\mathfrak{q}_{\gam}(\eps))\mathfrak{q}_{\gam}(\eps).
\end{equation}
The proof is completed.
\end{proof}
\end{lemma}

\begin{lemma}
\label{lemma-estimate-phi-corner}
Let $(\hat{\alp}, \hat{\eps}, \hat{\sigma})$ be from Lemma \ref{lemma-estiamte-uR-3}. Fix $\beta\in(0,1)$, and fix $\alp$ as $\alp=\hat{\alp}$
in \eqref{definition-iterset-shocks} and \eqref{definition-iterset-psi}. Let us assume the same conditions as Proposition \ref{wellposedness-lbvp-phi-cutoff}. In addition, assume that
\begin{equation*}
  0<\eps<\hat{\eps},\quad\tx{and}\quad M_2\mathfrak{q}_{\gam}(\eps)\le \hat{\sigma}
\end{equation*}
 so that Lemma \ref{lemma-estiamte-uR-3} holds. Then, there exists a sufficiently large constant $H_0>1$ depending on $d_0$ so that if $L\ge H_0$, and  if $\phi\in C^1(\ol{\Om_{f,L}})\cap C^2(\Om_{f,L})$ is a solution to the linear boundary value problem \eqref{lbvp-phi-cutoff}, then we have
\begin{equation}\label{estimate-phi-finalC2}
  \|\phi\|_{2,\hat{\alp}, \Om_{f,L}}^{(-\beta;-(1+\hat{\alp}),\{P_{2,L}^f\})}
  \le C(1+M_1\mathfrak{q}_{\gam}(\eps))\mathfrak{q}_{\gam}(\eps)
\end{equation}
for a constant $C>0$ depending only on $(\gam, B_0, d_0)$.

\begin{proof} The proof is divided into three steps.

{\textbf{1.}}
By \eqref{equiv-weight1}, we can fix a constant $H_0>0$ sufficiently large depending on $d_0>0$ so that if $L\ge H_0$, then $\til{L}$ given by \eqref{definition-Ltilde} satisfies $\til{L}\ge 10(1+h_3^*)$ for $h_3^*$ from Lemma \ref{lemma-estiamte-uR-3}. If $L\ge H_0$ and if $\phi\in C^1(\ol{\Om_{f,L}})\cap C^2(\Om_{f,L})$ is a solution to \eqref{lbvp-phi-cutoff}, then it follows from Lemmas \ref{lemma-estiamte-uR-1}--\ref{lemma-estiamte-uR-3} that $u^R$ given by \eqref{definition-scaled-functions} satisfies the following estimates:
\begin{equation}
\label{uR-estimate-total}
\begin{split}
&\sup_{4h_2^* \le R \le \frac 45 \til{L}} \|u^R\|_{2,\hat{\alp}, Q_R\cap\{\frac{11}{20}<\zeta_2<\frac{19}{20}\}}\le C(1+M_1\mathfrak{q}_{\gam}(\eps))\mathfrak{q}_{\gam}(\eps),\\
&\|u^{\til{L}}\|_{2,\hat{\alp}, Q_{\til{L}} \cap \{\zeta_2>\frac{11}{20}\}}^{(-1-\hat{\alp}, \{{\bf q}_*\})}
\le C(1+M_1\mathfrak{q}_{\gam}(\eps))\mathfrak{q}_{\gam}(\eps)
\end{split}
\end{equation}
for $h_2^*$ from Lemma \ref{lemma-estimate-Fr}. Here, the estimate constant $C$ depends on $d_0$.

For $R\in[4h_2^*, \frac 45 \til{L}]\cup\{\til{L}\}$,  let us define
 \begin{equation*}
\begin{split}
&\mcl{D}_R:=\left\{{\bm \eta}=(\eta_1, \eta_2):
 {\bm\zeta}=\frac{{\bm\eta}}{R}\in Q_R\cap \{\frac{3}{5}<\zeta_2<\frac{9}{10}\}\right\}\quad\quad\tx{for $R\in [4h_2^*, \frac 45 \til{L}]$},\\
&\mcl{D}_{\til{L}}:=\left\{{\bm \eta}=(\eta_1, \eta_2):
 {\bm\zeta}=\frac{{\bm\eta}}{\til{L}}\in Q_{\til{L}}\cap \{\zeta_2>\frac{3}{5}\}
\right\}\quad\quad\tx{for $R=\til{L}$}.
\end{split}
\end{equation*}
Then, for $\Lambda_{f,L}$ given by \eqref{definition-Lambda-fL}, we have
\begin{equation*}
\left(\Lambda_{f,L}\cap\{\eta_2>3h_2^*\}\right)\subset \bigcup_{R\in[4h_2^*, \frac 45\til{L}]\cup\{\til{L}\}} \mcl{D}_R .
\end{equation*}
So, for any given ${\bm\eta}\in \Lambda_{f,L}\cap\{\eta_2>3h_2^*\}$, there exists a constant $R\in [4h_2^*, \frac 45\til{L}]\cup\{\til{L}\}$ such that ${\bm\eta}=(\eta_1, \eta_2)\in \mcl{D}_{R}$. And, such a constant $R$ satisfies
\begin{equation}
\label{R-eta-relation}
\eta_2<R<2\eta_2.
\end{equation}

For $\til{\phi}({\bm\eta})$ given by \eqref{definition-tildes}, let us set
\begin{equation*}
  A_k({\bm\eta}):=(1+\eta_2)^{k-\beta}\sum_{l=0}^k
 |\der_{\eta_1}^l\der_{\eta_2}^{k-l}\til{\phi}({\bm\eta})|
 \quad\tx{for $k=0,1$}.
\end{equation*}
Then it directly follows from \eqref{definition-scaled-functions} that
\begin{equation}
\label{Ak-expression}
A_k({\bm\eta})=\left(\frac{1+R\zeta_2}{R}\right)^{k-\beta}
\sum_{l=0}^k |\der_{\zeta_1}^l\der_{\zeta_2}^{k-l}u^R({\bm\zeta})|
\quad\tx{for ${\bm\zeta}=\frac{\bm\eta}{R}$}.
\end{equation}
Since
\begin{equation}\label{equiv-weight4}
  \frac 12\le \frac{1+R\zeta_2}{R}\le \frac{1}{4h_2^*}+1
  \quad\tx{for $\zeta\in[\frac 12, 1]$ and $R\in [4h_2^*, \frac 45\til{L}]\cup\{\til{L}\}$},
\end{equation}
we obtain from \eqref{uR-estimate-total} and \eqref{Ak-expression} that
\begin{equation}\label{A1-estimate}
  \sup_{\bm\eta\in \Lambda_{f,L}\cap\{\eta_2\ge 3h_2^*\}}
  A_1({\bm\eta})\le C(1+M_1\mathfrak{q}_{\gam}(\eps))\mathfrak{q}_{\gam}(\eps),
\end{equation}
and this yields that
\begin{equation}
\label{phi-tilde-C1-estimate}
\|\til{\phi}\|_{1,\Lambda_{f,L}\cap\{\eta_2>3h_2^*\}}^{(-\beta)}\le C(1+M_1\mathfrak{q}_{\gam}(\eps))\mathfrak{q}_{\gam}(\eps).
\end{equation}
{\textbf{2.}}
For ${\bm\eta}=(\eta_1, \eta_2)$, ${\bm\eta}'=(\eta'_1, \eta'_2)\in \Lambda_{f,L}\cap \{\eta_2>3h_2^*\}$, let ${\bf x}=(x_1, x_2)$, ${\bf x}'=(x'_1, x'_2)$ be given by \eqref{definition-cov}, corresponding to ${\bm \eta}$, ${\bm\eta}'$, respectively. And, let us define
\begin{equation*}
\begin{split}
  &\delta_{\bm\eta}^{(2)}:=\min\{1+\eta_2, |{\bm\eta}-P_{2,L}^f|\},\quad \delta_{{\bm\eta}, {\bm\eta}'}^{(2)}:=\min\{ \delta_{\bm\eta}^{(2)}, \delta_{{\bm\eta}'}^{(2)}\},\\
  &\hat\delta_{{\bf x}}^{(2)}:=\min\{1+x_2, |{\bf x}-P_{2,L}^f|\},\quad \hat\delta_{{\bf x}, {\bf x}'}^{(2)}:=\min\{ \hat\delta_{\bf x}^{(2)}, \hat\delta_{{\bf x}'}^{(2)}\}.
  \end{split}
\end{equation*}
Similarly to \eqref{equiv-weight2}, one can directly check by using \eqref{equiv-weight1} that
\begin{equation}
\label{equiv-weight3}
\begin{split}
&\om_2\le \frac{1+\min\{\eta_2, \eta_2'\}}{1+\min\{x_2, x_2'\}}
\le \frac{1}{\om_2},\\
&\om_2  \le \frac{\delta_{\bm\eta}^{(2)}}{1+\eta_2}\cdot \frac {1+x_2}{\hat\delta_{{\bf x}}^{(2)}}\le \frac{1}{\om_2},\\
&\om_2  \le \frac{\delta_{{\bm\eta}, {\bm\eta}'}^{(2)}}{1+\max\{\eta_2, \eta_2'\}}\cdot \frac {1+\max\{x_2, x_2'\}}{\hat\delta_{{\bf x}, {\bf x}'}^{(2)}}\le \frac{1}{\om_2}
\end{split}
\end{equation}
for some constant $\om_2\in(0,1)$ depending only on $\tw$.

Set
\begin{equation*}
  A_2({\bm\eta}):=(1+\eta_2)^{k-\beta}\left(
  \frac{\delta_{\bm\eta}^{(2)}}{1+\eta_2}\right)^{1-\hat{\alp}}
  \sum_{l=0}^2
 |\der_{\eta_1}^l\der_{\eta_2}^{2-l}\til{\phi}({\bm\eta})|.
\end{equation*}
By \eqref{definition-scaled-functions}, we can rewrite $A_2({\bm\eta})$ as
\begin{equation*}
A_2({\bm\eta})=\left(\frac{1+R\zeta_2}{R}\right)^{2-\beta}
\left(
  \frac{\delta_{\bm\eta}^{(2)}}{1+\eta_2}\right)^{1-\hat{\alp}}
\sum_{l=0}^2
 |\der_{\zeta_1}^l\der_{\zeta_2}^{2-l}u^R({\bm\zeta})|.
\end{equation*}
By \eqref{definition-d-zeta} and \eqref{R-eta-relation}, we can estimate $\frac{\delta_{\bm\eta}^{(2)}}{1+\eta_2}$ as
\begin{equation}
\label{equiv-weight5}
  \frac{\delta_{\bm\eta}^{(2)}}{1+\eta_2}\le
  \begin{cases}
  1\quad&\mbox{if $R\le \frac 45 \til{L}$},\\
  4d_{\bm\zeta}\quad&\mbox{if $R=\til{L}$}
  \end{cases}
\end{equation}
for $d_{\bm\zeta}$ given by \eqref{definition-d-zeta}.
So we obtain from \eqref{uR-estimate-total} that
\begin{equation}\label{A2-estimate}
  \sup_{\bm\eta\in \Lambda_{f,L}\cap\{\eta_2\ge 3h_2^*\}}
  A_2({\bm\eta})\le C(1+M_1\mathfrak{q}_{\gam}(\eps))\mathfrak{q}_{\gam}(\eps).
\end{equation}

{\textbf{3.}} Let us fix two distinct points ${\bm\eta}=(\eta_1,\eta_2)$, ${\bm\eta'}=(\eta_1', \eta_2')\in \Lambda_{f,L}\cap\{\eta_2>3h_2^*\}$. Without loss of generality, we assume that $\eta_2\le \eta_2'$.  And, set
\begin{equation*}
\begin{split}
&A_{1+\hat{\alp}}({\bm\eta}, {\bm\eta'}):=
(1+\eta_2)^{1+\hat{\alp}-\beta}\sum_{k=1,2}
\frac{|\der_{\eta_k}\til{\phi}({\bm\eta})
-\der_{\eta_k}\til{\phi}({\bm\eta}')|}{|{\bm\eta}-{\bm\eta}'|^{\hat{\alp}}},\\
&A_{2+\hat{\alp}}({\bm\eta}, {\bm\eta'}):=(1+\eta_2)^{2+\hat{\alp}-\beta}
\frac{\delta^{(2)}_{{\bm\eta}, {\bm\eta'}}}{1+\eta_2'}
\sum_{j=0}^2
\frac{|\der_{\eta_1}^j\der_{\eta_2}^{2-j}\til{\phi}({\bm\eta})-\der_{\eta_1}^j\der_{\eta_2}^{2-j}\til{\phi}({\bm\eta}')|}{|{\bm\eta}-{\bm\eta}'|^{\hat{\alp}}}.
\end{split}
\end{equation*}

(Case 1) Suppose that $|{\bm\eta}-{\bm\eta}'|\le \frac{\eta_2'}{40}$, and ${\bm\eta}'\in \mcl{D}_R$ for dome $R\in [4h_2^*, \frac 45 \til{L}]\cup \{\til{L}\}$. Then, we have ${\bm\zeta}, {\bm\zeta}'\in Q_R$ for $\bm\zeta=\frac{{\bm\eta}}{R}$ and ${\bm\zeta}'=\frac{{\bm\eta}'}{R}$. Then, by \eqref{uR-estimate-total}, \eqref{equiv-weight4} and \eqref{equiv-weight5}, we obtain that
\begin{equation*}
A_{1+\hat{\alp}}({\bm\eta}, {\bm\eta'})+A_{2+\hat{\alp}}({\bm\eta}, {\bm\eta'})\le
C(1+M_1\mathfrak{q}_{\gam}(\eps))\mathfrak{q}_{\gam}(\eps).
\end{equation*}

(Case 2) Suppose that $|{\bm\eta}-{\bm\eta}'|\ge \frac{\eta_2'}{40}$. Then, by \eqref{uR-estimate-total}, we have
\begin{equation*}
  A_{1+\hat{\alp}}({\bm\eta}, {\bm\eta}')
  \le \left(\frac{1+\eta_2}{\frac{1}{40}\eta_2'}\right)^{\hat{\alp}}
  (A_1({\bm\eta})+A_1({\bm\eta}')).
\end{equation*}
Since $\frac{\delta^{(2)}_{{\bm\eta}, {\bm\eta'}}}{1+\eta_2'}\le \min\{ \frac{\delta_{\bm\eta}^{(2)}}{1+\eta_2}, \frac{\delta_{\bm\eta'}^{(2)}}{1+\eta_2'}\}$, and $\min\{ \frac{\delta_{\bm\eta}^{(2)}}{1+\eta_2}, \frac{\delta_{\bm\eta'}^{(2)}}{1+\eta_2'}\}\le 1$, we also have
\begin{equation*}
A_{2+\hat{\alp}}({\bm\eta}, {\bm\eta}')
  \le\left(\frac{1+\eta_2}{\frac{1}{40}\eta_2'}\right)^{\hat{\alp}}
  (A_2({\bm\eta})+A_2({\bm\eta}')).
\end{equation*}
Then, we obtain from \eqref{A1-estimate} and \eqref{A2-estimate} that
\begin{equation}
\label{A-holder-estimate-final}
\sup_{{\bm\eta, \bm\eta'\in \Lambda_{f,L}\cap\{\eta_2\ge 3h_2^*\}}\atop{\bm\eta\neq \bm\eta'}} A_{1+\hat{\alp}}({\bm\eta}, {\bm\eta}')+A_{2+\hat{\alp}}({\bm\eta}, {\bm\eta}') \le C(1+M_1\mathfrak{q}_{\gam}(\eps))\mathfrak{q}_{\gam}(\eps).
\end{equation}
In the estimates \eqref{phi-tilde-C1-estimate}, \eqref{A2-estimate} and \eqref{A-holder-estimate-final}, the estimate constants $C$ are different, but they all depend on $d_0>0$, but they are independent of $\beta\in(0,1)$ and $\til{L}$. Therefore, by  combining all the estimates \eqref{phi-tilde-C1-estimate}, \eqref{A2-estimate} and \eqref{A-holder-estimate-final}, we get
\begin{equation}
\label{phi-tilde-final-estimate}
\|\til{\phi}\|_{2,\hat{\alp},\Lambda_{f,L}\cap\{\eta_2>3h_2^*\}}
^{(-\beta;-(1+\hat{\alp}),\{P_{2,L}^f\})}\le C(1+M_1\mathfrak{q}_{\gam}(\eps))\mathfrak{q}_{\gam}(\eps).
\end{equation}

Finally, the estimate \eqref{estimate-phi-finalC2} can be derived by using Lemma \ref{lemma-estimate-nr-sym}, \eqref{equiv-weight3} and \eqref{phi-tilde-final-estimate}. This completes the proof.

\end{proof}
\end{lemma}

\begin{proof}[Proof of Proposition \ref{wellposedness-lbvp-phi-cutoff}]
Let $(\hat{\alp}, \hat{\eps}, \hat{\sigma})$ be from Lemma \ref{lemma-estiamte-uR-3}, and let $H_0$ be from Lemma \ref{lemma-estimate-phi-corner}. And, assume that
\begin{equation*}
  0<\eps<\hat{\eps},\quad M_2\mathfrak{q}_{\gam}(\eps)\le \hat{\sigma},\quad L\ge H_0
\end{equation*}
so that Lemma \ref{lemma-estimate-phi-corner} holds. Then, the proof of Proposition \ref{wellposedness-lbvp-phi-cutoff} is given by employing the method of continuity if we choose $L_{\rm wp}$ as $L_{\rm wp}=H_0$. An auxiliary boundary value problem for the method of continuity is given in the following.

For fixed functions $\mathfrak{F}\in C^{\hat{\alp}}_{((1-\hat{\alp}), \der \Om_{f,L})}(\Om_{f,L})$ and $\mathfrak{G}\in C^{\hat{\alp}}(\ol{\Gam^f_{{\rm cutoff}, L}})$,
with satisfying the compatibility conditions
\begin{equation}
\label{comp-cond-1}
  \mathfrak{F}=0\quad\tx{on $\der \Om_{f,L}\setminus \Gam^f_{{\rm cutoff},L}$},\quad
  \mathfrak{G}(P_{3,L}^f)=0,
\end{equation}
we consider the following auxiliary boundary value problem:
\begin{equation}
\label{aux-lbvp-1}
  \begin{cases}
  \Delta \Phi=\mathfrak{F}\quad&\mbox{in $\Om_{f,L}$}\\
  \Phi=0\quad&\mbox{on $\der \Om_{f,L}\setminus \Gam^f_{{\rm cutoff},L}$}\\
  \der_{{\bf n}_c}\Phi=\mathfrak{G}\quad&\mbox{on $\Gam^f_{{\rm cutoff}, L}$}
  \end{cases}.
\end{equation}
Since $\Om_{f,L}\subset \R^2$, we can apply Sobolev inequality and H\"{o}lder inequality to conclude from  Lax-Milgram theorem that \eqref{aux-lbvp-1} has a unique weak solution $\Phi\in H^1(\Om_{f,L})$. By local estimates of a weak solution, one can check that $\Phi$ is continuous up to the boundary in $\Om_{f,L}$, and that $\Phi$ is $C^{1,\hat{\alp}}$ up to the boundary away from $P_{2,L}^f$. Furthermore, local scaling arguments of $\Phi$ yield a priori weighted $C^{2,\hat{\alp}}$ estimate away from $P_{2,L}^{f}$. Finally, by adjusting the proof of Lemma \ref{lemma-estiamte-uR-3}, we can achieve a priori weighted $C^{2, \hat{\alp}}$ estimate of $\Phi$ near $P_{2,L}^f$. As a result, we obtain that
\begin{equation*}
  \|\Phi\|_{2,\hat{\alp}, \Om_{f,L}}^{(-(1+\hat{\alp}), \der \Om_{f,L})}\le C_L(\|\mathfrak{F}\|_{\hat{\alp}, \Om_{f,L}}^{(1-\hat{\alp}, \der \Om_{f,L})}+\|\mathfrak{G}\|_{\hat{\alp}, \Gam^f_{{\rm cutoff}, L}})
\end{equation*}
for some constant $C_L>0$ possibly depending on $L$. Therefore, the auxiliary boundary value problem \eqref{aux-lbvp-1} has a unique solution $\Phi\in C^{2,\hat{\alp}}_{-(1+\hat{\alp}), \der \Om_{f,L}}(\Om_{f,L})$.

For a fixed function $\mathfrak{H}\in C^{2,\hat{\alp}}(\ol{\Om_{f,L}})$,
with satisfying the compatibility conditions
\begin{equation*}
  \Delta \mathfrak{H}=0 \quad\tx{on $\der \Om_{f,L}\setminus \Gam^f_{{\rm cutoff},L}$},\quad
  \der_{{\bf n}_c}\mathfrak{H}(P_{3,L}^f)=0,
\end{equation*}
$\Psi$ solves
\begin{equation}
\label{aux-lbvp-2}
  \begin{cases}
  \Delta \Psi=0\quad&\mbox{in $\Om_{f,L}$}\\
  \Phi=\mathfrak{H}\quad&\mbox{on $\der \Om_{f,L}\setminus \Gam^f_{{\rm cutoff},L}$}\\
  \der_{{\bf n}_c}\Phi=0\quad&\mbox{on $\Gam^f_{{\rm cutoff}, L}$}
  \end{cases}
\end{equation}
if and only if $\hat{\Phi}:=\Psi-\mathfrak{H}$ solves
\begin{equation*}
  \begin{cases}
  \Delta \hat{\Phi}=\mathfrak{F}\quad&\mbox{in $\Om_{f,L}$}\\
  \hat{\Phi}=0\quad&\mbox{on $\der \Om_{f,L}\setminus \Gam^f_{{\rm cutoff},L}$}\\
  \der_{{\bf n}_c}\hat{\Phi}=\mathfrak{G}\quad&\mbox{on $\Gam^f_{{\rm cutoff}, L}$}
  \end{cases}
\end{equation*}
for $(\mathfrak{F}, \mathfrak{G})=-(\Delta \mathfrak{H},\der_{{\bf n
  }_c}\mathfrak{H})$ satisfying \eqref{comp-cond-1}.
Therefore, the well-posedness of \eqref{aux-lbvp-2} directly follows from the well-posedness of \eqref{aux-lbvp-1}.

Then, we can apply the method of continuity to conclude that,  for any $\vphi\in \mcl{I}^f_{M_2,L}$, the linear boundary value problem \eqref{lbvp-phi-cutoff} associated with $\vphi$ has a unique solution $\phi\in C^{2}(\Om_{f,L})\cap C^1(\ol{\Om_{f,L}})$. Then, Lemma \ref{lemma-estimate-phi-corner} directly yields the estimate \eqref{apriori-phi-lbvp-cutoff}.


\end{proof}

\subsection{Detached shocks in cut-off domains}
\label{subsection-proposition-pf}

Let $\hat{\eps}>0$ be from Lemma \ref{lemma-estiamte-uR-3}, and let us fix a {\emph{detached distance}} $d_0>0$. Let $H_0>1$ be  from Lemma \ref{lemma-estimate-phi-corner}. For $\ipsi$ given by \eqref{definition-ipsi} with $M_{\infty}\ge \frac{1}{\hat{\eps}}$, and $L\ge H_0$, we prove Proposition \ref{proposition-wp-fbpbd}, thus solve Problem \ref{fbp-bdd-dmns}.

Throughout \S \ref{subsection-proposition-pf}, we assume that $(M_1, M_2, \eps, L)$ satisfy the following conditions.
\begin{condition}
\label{condition2-for-parameters}
\begin{itemize}
\item[(i)] For $M_1^*$, $M_2^*$ from Lemmas \ref{lemma-f-iterset-nonempty}, \ref{lemma-vphi-iterset-nonempty}, respectively, assume that
    \begin{equation*}
      M_1\ge M_1^*,\quad M_2\ge M_2^*;
    \end{equation*}

\item[(ii)] For $\eps_2$, $\hat{\eps}$ from Lemmas \ref{lemma-shock-polar2}, \ref{lemma-estiamte-uR-3}, respectively, assume that
    \begin{equation*}
      0<\eps\le \min\{\eps_2, \hat{\eps}\};
    \end{equation*}

\item[(iii)] For $\sigma_2$ from Lemma \ref{lemma-distance-iterset}, assume that
    \begin{equation*}
      (M_1+1)\mathfrak{q}_{\gam}(\eps)
      \le \min\{\sigma_2, \frac{\cot \tw}{4}, \frac{\tan \tw}{4}\};
    \end{equation*}

\item[(iv)] For $\hat{\sigma}$ from Lemma \ref{lemma-estiamte-uR-3}, assume that
    \begin{equation*}
      M_2\mathfrak{q}_{\gam}(\eps)\le \hat{\sigma};
    \end{equation*}

\item[(v)] For $H_0>1$ from Lemma \ref{lemma-estimate-phi-corner}, assume that
    \begin{equation*}
      L\ge H_0.
    \end{equation*}
\end{itemize}
\end{condition}

Fix $\beta\in(0,1)$. Under Condition \ref{condition2-for-parameters}, let $\mcl{J}_{M_1, L}$ be given by \eqref{definition-iterset-shocks} with $\alp=\hat{\alp}$ for $\hat{\alp}$ from Lemma \ref{lemma-estiamte-uR-3}. And, for each $f\in \mcl{J}_{M_1,L}$, let $\mcl{I}_{M_2, L}^f$ be given by \eqref{definition-iterset-psi} with $\alp=\hat{\alp}$.
Then, all the lemmas and propositions stated up to \S \ref{subsection-fixedbd-nlbvp-psi} hold.

\begin{lemma}
\label{lemma-wp-fixed-nlbvp}
One can fix $M_2\in[M_2^*, \infty)$ depending on $d_0>0$ so that, for any $f\in \mcl{J}_{M_1,L}$, the nonlinear boundary value problem \eqref{fixed-nlbvp-bounded-phi} in $\Om_{f,L}$ has a unique solution $\phi \in \mcl{I}_{M_2, L}^{f}$.

\begin{proof}
By Proposition \ref{wellposedness-lbvp-phi-cutoff}, for each fixed $f\in \mcl{J}_{M_1,L}$ and $\vphi\in \mcl{I}^f_{M_2, L}$, the linear boundary value problem \eqref{lbvp-phi-cutoff} associated with $\vphi$ admits a unique solution $\phi\in C^2(\Om_{f,L})\cap C^1(\ol{\Om_{f,L}})$, and the solution $\phi$ satisfies the estimate \eqref{apriori-phi-lbvp-cutoff}. By Condition \ref{condition2-for-parameters}(iv) and the estimate \eqref{apriori-phi-lbvp-cutoff}, we have
\begin{equation}
\label{estimate-nlbvp-sol}
  \|\phi\|_{2,\hat{\alp}, \Om_{f,L}}^{(-\beta;-(1+\hat{\alp}), \{P_{2,L}^f\})}\le C\left(\sigma_2+\cot \tw+ \tan \tw\right)\mathfrak{q}_{\gam}(\eps)
\end{equation}
for the constant $C$ from \eqref{apriori-phi-lbvp-cutoff}.

Now, we choose the constant $M_2$ as
\begin{equation}
\label{choice-M2}
  M_2=1+M_2^*+C\left(\sigma_2+\cot \tw+ \tan \tw\right).
\end{equation}
Then we apply the Schauder fixed point theorem to show that the nonlinear boundary value problem \eqref{fixed-nlbvp-bounded-phi} has at least one solution in $\mcl{I}^f_{M_2L}$. Furthermore, the uniqueness of the solution to \eqref{fixed-nlbvp-bounded-phi} in $\mcl{I}^f_{M_2, L}$ can be proved by the maximum principle of nonlinear elliptic boundary value problem.
\end{proof}

\end{lemma}

Now, we are ready to prove Proposition \ref{proposition-wp-fbpbd}.	

\begin{proof}[Proof of Proposition \ref{proposition-wp-fbpbd}]
The proof is divided into three steps.

{\textbf{1.}} By Lemma \ref{lemma-wp-fixed-nlbvp}, for each $f\in \mcl{J}_{M_1,L}$, the nonlinear boundary value problem \eqref{fixed-nlbvp-bounded-phi} in $\Om_{f,L}$ admits a unique solution $\phi^{f}\in \mcl{I}_{M_2, L}^f$ for $M_2$ given by Lemma \ref{lemma-wp-fixed-nlbvp}.
For such $\phi^{f}$, let us set
\begin{equation}
\label{definition-psi-update}
  \psi^{f}({\bf x}):=\phi^{f}({\bf x})+\psi_0({\bf x})\quad\tx{in $\Om_{f,L}$}
\end{equation}
for $\psi_0$ given by \eqref{definition-background}.
For $\hat{\rho}$ from Lemma \ref{lemma-rho-expression}, we define $\hat{f}:[0,L]\rightarrow \R$ by
\begin{equation}
\label{definition-f-update}
  \hat{f}(x_2):=b_0-d_0+\int_0^{x_2}
  \frac{\psi^f_{x_1}(f(t),t)}
  {\psi^{f}_{x_2}(f(t),t)-\iu\hat{\rho}(|\nabla\psi^f|^2)(f(t),t)}\,dt.
\end{equation}
Then, we have $\hat{f}(0)=b_0-d_0$.
Due to the boundary condition $\psi=0$ on $\Gam_{\rm sym}$ stated in \eqref{fixed-nlbvp-bounded-psi},
\eqref{definition-f-update} yields that
\begin{equation*}
\hat{f}'(0)=\frac{\psi_{x_1}(f(0),0)}
{\psi^{f}_{x_2}(f(0),0)-\iu\hat{\rho}
(|\nabla\psi^f|^2)(f(0),0)}=0.
\end{equation*}
For $f_0$ defined by \eqref{definition-background}, a direct computation with using \eqref{definition-background} and \eqref{background-fbc} yields that
\begin{equation*}
\begin{split}
  &(\hat{f}-f_0)(x_2)\\
  &=\int_0^{x_2}
  \frac{
  \phi_{x_1}^{f}(\rx)-
  \frac{u^{\eps}_{\rm st}\kappa_w}
  {u^{\eps}_{\rm st}-\iu}
  \left(\left(\hat{\rho}(|\nabla\psi^f(\rx)|^2)
  -\hat{\rho}(|\nabla\psi_0|^2)\right)\iu
  -\phi_{x_2}^f(\rx)\right)}
  {\phi_{x_2}^f(\rx)+\rho_{\rm st}^{\eps}u_{\rm st}^{\eps}
  -\hat{\rho}(|\nabla\psi^f|^2(\rx))\iu}\Bigr|_{\rx=(f(t),t)}\,dt=:\int_0^{x_2} F(t)\,dt.
  \end{split}
  \end{equation*}
For $\eps=0$, we have $|\rho_{\rm st}^{\eps}(u_{\rm st}^{\eps}-\iu)|=((\gam-1)B_0)^{\frac{1}{\gam-1}}\sqrt{2B_0}>0$.
And, we have
\begin{equation*}
\begin{split}
  |\phi_{x_2}^f+\rho_{\rm st}^{\eps}u_{\rm st}^{\eps}
  -\hat{\rho}(|\nabla\psi^f|^2)\iu|
  &\ge |\rho_{\rm st}^{\eps}(u_{\rm st}^{\eps}-\iu)|
  -|\phi_{x_2}
  -\left(\hat{\rho}(|\nabla\psi^f|^2)-\hat{\rho}(|\nabla\psi_0|^2)
  \right)\iu|\\
  &\ge |\rho_{\rm st}^{\eps}(u_{\rm st}^{\eps}-\iu)|-CM_2\mathfrak{q}_{\gam}(\eps)
  \end{split}
\end{equation*}
for a constant $C>0$ depending only on $(\gam, B_0)$. Therefore, by Lemma \ref{lemma-shock-polar2}, one can fix a small constant $\eps^{\sharp}\in(0, \min\{\eps_2, \hat{\eps}\}]$ depending on $d_0$ so that if the parameter $\eps(=\frac{1}{M_{\infty}})$ satisfies
\begin{equation}\label{condition4-for-epsilon0}
  \eps\le \eps^{\sharp},
\end{equation}
then we get
\begin{equation}\label{estimate-denominator-f-diff}
  |(\phi_{x_2}^f(\rx)+\der_{x_2}\psi_0)
  -\hat{\rho}(|\nabla\psi^f(\rx)|^2)\iu|\ge \frac 14((\gam-1)B_0)^{\frac{1}{\gam-1}}\sqrt{2B_0}>0\quad\tx{in $\ol{\Om_{f,L}}$}.
\end{equation}
From \eqref{estimate-denominator-f-diff} and Definition \ref{definition-norms-phi-decay-corner}, we obtain that
\begin{equation*}
  |F(t)|\le C\|\phi^f\|_{2,\hat{\alp}, \Om_{f,L}}^{(-\beta;-(1+\hat{\alp}), \{P_{2,L}^f\})} (1+t)^{\beta-1}\quad\tx{for $0<t<L$},
\end{equation*}
and this yields that
\begin{equation*}
  \sup_{x_2\in(0,L)}(1+x_2)^{-\beta}|(\hat{f}-f_0)(x_2)|\le C\|\phi^f\|_{2,\hat{\alp}, \Om_{f,L}}^{(-\beta;-(1+\hat{\alp}), \{P_{2,L}^f\})}.
\end{equation*}
By \eqref{estimate-f-iterset}, we have
\begin{equation*}
|x_2-L|\le |{\bf x}-P_{2,L}^f|\le |x_2-L|\sqrt{1+(M_1+1)^2\mathfrak{q}_{\gam}^2(\eps)}
\end{equation*}
for any $\rx=(f(x_2),x_2)$, and $f\in \mcl{J}_{M_1,L}$. Therefore, it can be directly checked that
\begin{equation}\label{estimate-f-diff-final}
  \|\hat f-f_0\|_{2,\hat{\alp},(0,L)}^{(-\beta;-(1+\hat{\alp}), \{L\})}\le C\|\phi^f\|_{2,\hat{\alp}, \Om_{f,L}}^{(-\beta;-(1+\hat{\alp}), \{P_{2,L}^f\})}\le CM_2\mathfrak{q}_{\gam}(\eps)=: \mcl{C}_*\mathfrak{q}_{\gam}(\eps).
\end{equation}
for $C>0$ depending only on $(\gam, B_0,d_0)$.

We choose $M_1$ in \eqref{definition-iterset-shocks} as
\begin{equation}
\label{choice-M1}
M_1=\max\{2\mcl{C}_{*}, M_1^*\}
\end{equation}
for $M_1^*$ from Condition \ref{condition2-for-parameters}, then we have
\begin{equation*}
\hat{f}\in \mcl{J}_{M_1,L}.
\end{equation*}

We define a mapping $\mcl{H}: \mcl{J}_{M_1,L}\rightarrow \mcl{J}_{M_1,L}$ by
\begin{equation}
\label{iteration-mapping-H}
  \mcl{H}(f):=\hat{f}
\end{equation}
for $\hat f$ given by \eqref{definition-f-update}.
If $\mcl{H}(f_*)=f_*$ for some $f_*\in \mcl{J}_{M_1, L}$, then $(f_*, \psi^{f_*})$ solves Problem \ref{fbp-bdd-dmns} for $\psi^{f_*}$ given by \eqref{definition-psi-update}. Therefore, Proposition \ref{proposition-wp-fbpbd} is proved once we show that $\mcl{H}$ has a fixed point in $\mcl{J}_{M_1,L}$.
\smallskip

{\textbf{2.}}
By Remark \ref{remark-compact-embedding}, $\mcl{J}_{M_1, L}$ is compact and convex in $C^{2,\frac{\hat{\alp}}{2}}_{(-\beta;-(1+\frac{\hat{\alp}}{2}), \{L\})}((0,L))$.

Suppose that a sequence $\{f_n\}$ in $\mcl{J}_{M_1,L} $ satisfies that
\begin{equation}\label{convergence-shocks}
  \lim_{n\to \infty}\|f_n-f_{\infty}\|
  _{2,\frac{\hat{\alp}}{2},(0,L)}^{(-\beta;-(1+\frac{\hat{\alp}}{2}),\{L\})}
  =0
\end{equation}
for some function $f_{\infty}\in \mcl{J}_{M_1, L}$.

In the following, for each $n\in \mathbb{N}$, we define an invertible transformation $\mcl{T}_n: \ol{\Om_{f_n, L}}\rightarrow \ol{\Om_{f_{\infty}, L}}$ such that
\begin{equation*}
\mcl{T}_n(\Om_{f_n, L})=\Om_{f_{\infty}, L},\quad \mcl{T}_n(\Gam_{{\rm sh}, L}^{f_n})=\Gam_{{\rm sh}, L}^{f_{\infty}}.
\end{equation*}

Fix $n\in\mathbb{N}\cup\{\infty\}$. We draw a line from the point $P_{3, L}^{f_n}$ perpendicular to $\Gam_{b,L}^{f_n}$. Let $Q_n$ be the $x_1$-intercept of this line (Fig. \ref{figure_polar}).
\begin{figure}[htp]
	\centering
	\begin{psfrags}
		\psfrag{p2}[cc][][0.8][0]{$P_{2,L}^{f_n}$}
		\psfrag{p3}[cc][][0.8][0]{$\phantom{aaaa}{\bf{P_{3,L}^{f_n}}}$}
		\psfrag{p0}[cc][][0.8][0]{$\rb$}
		\psfrag{p1}[cc][][0.8][0]{$\lb$}
		\psfrag{sm}[cc][][0.8][0]{$\Gam_{\rm{sym}}$}
		\psfrag{w}[cc][][0.8][0]{$\Gam_{b,L}^{f_n}\phantom{aaaaaaaaaaa}$}
		\psfrag{cf}[cc][][0.8][0]{$\phantom{aa}\Gam_{{\rm{cutoff}},f}^L$}
		\psfrag{sh}[cc][][0.8][0]{$\Gam_{{\rm{sh}},L}^{f_n}$}
		\psfrag{Q}[cc][][0.8][0]{$Q_n$}
		\psfrag{om}[cc][][0.8][0]{$\Om_{f,L}$}
		\psfrag{Uf}[cc][][0.8][0]{$M_{\infty}=\frac{1}{\eps}$}
\psfrag{t2}[cc][][0.8][0]{$\theta=\frac{\pi}{2}+\tw$}
\psfrag{t1}[cc][][0.8][0]{$\theta=\pi$}
\psfrag{a2}[cc][][0.8][0]{$\eta_2$}
\psfrag{h}[cc][][0.8][0]{\phantom{aaaa}${\bf{\eta_2=\frac{h_1}{\sin \tw}}}$}
		\includegraphics[scale=0.6]{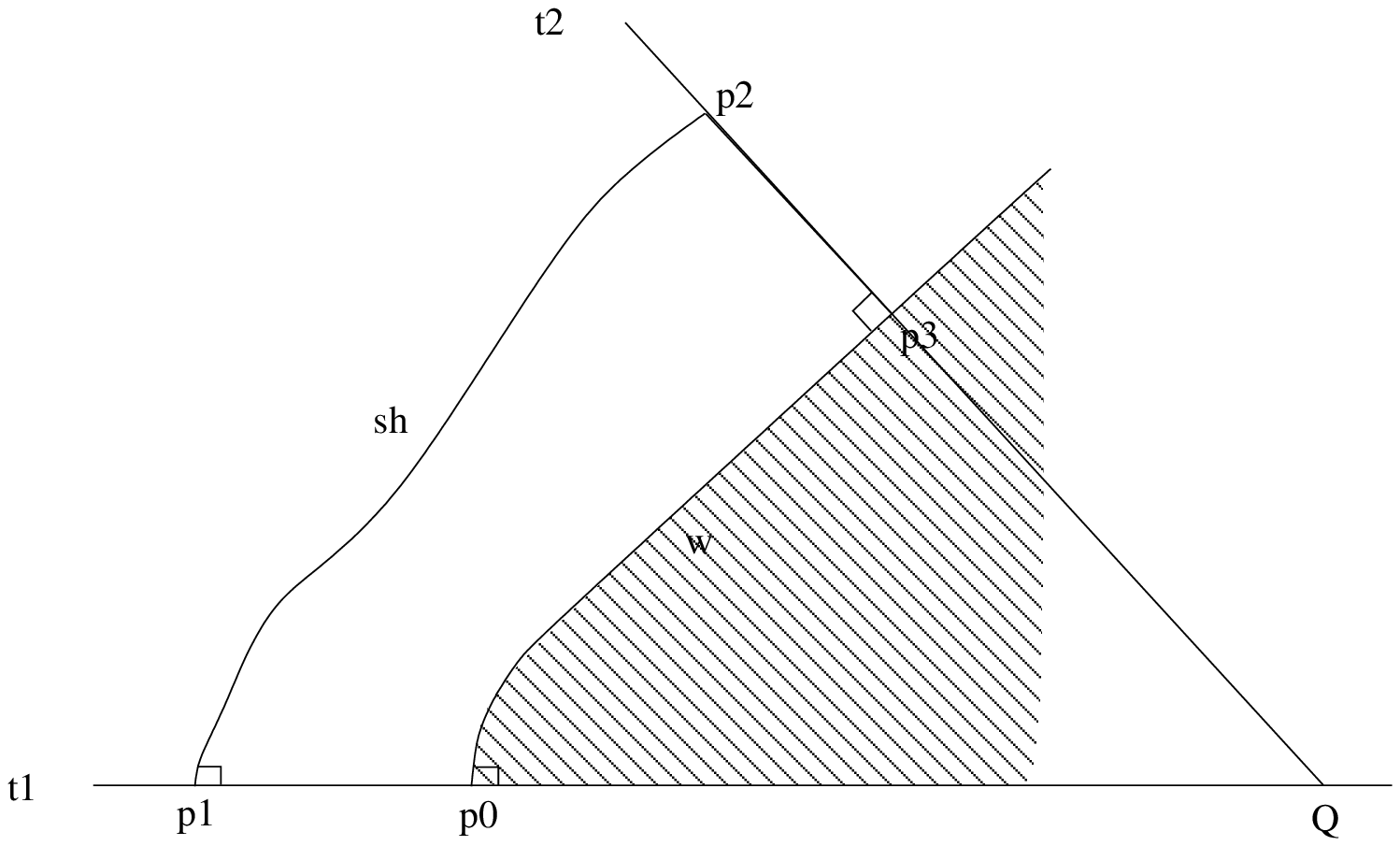}
		\caption{}\label{figure_polar}
	\end{psfrags}
\end{figure}
A direct computation yields that
\begin{equation*}
  Q_n:=\left(\frac{|P_{3,L}^{f_n}|}{\cos \tw},0\right).
\end{equation*}
Let $(r,\theta)$ be the polar coordinates centered at $Q_n$, that is,
\begin{equation*}
  {\rx}=Q_n+r(\cos\theta, \sin \theta).
\end{equation*}
In $(r,\theta)$-coordinates, we have $\Om_{f_n,L}\subset\{(r,\theta):r>0,\,\, \tw+\frac{\pi}{2}<\theta<\pi\}$.
We remark that, in $(r,\theta)$-coordinates,
we have $\der \Omega_{f_n,L}\cap \{\theta=\tw+\frac{\pi}{2}\}\neq \emptyset$ and
$\der \Omega_{f_n,L}\cap \{\theta=\tw+\pi\}\neq \emptyset$ for all $n\in \mathbb{N}\cup \{\infty\}$.
This property allows us to find coordinate transformations which map $\Om_{f_n,L}$ onto $\Om_{f_{\infty}, L}$.

Any point $\rx=(x_1,x_2)$ lying on $\Gam_{b,L}^{f_n}$ satisfies the equation $x_1=b(x_2)$ which can be rewritten in terms of $(r,\theta)$-coordinates as follows:
\begin{equation*}
  b(r\sin\theta)-r\cos \theta-\frac{|P_{3,L}^{f_n}|}{\cos \tw}=0\quad\tx{for $\frac{\pi}{2}+\tw<\theta<\pi$}.
\end{equation*}
By the property ($b_3$) stated in Definition \ref{definition-bluntbody-ftn}, we have
\begin{equation*}
\frac{\der}{\der r}\left(b(r\sin\theta)-r\cos \theta-\frac{|P_{3,L}^{f_n}|}{\cos \tw}\right)=b'(r\sin \theta)\sin \theta-\cos \theta>0\quad\tx{for $\frac{\pi}{2}+\tw\leq\theta\leq\pi$.}
\end{equation*}
So the implicit function theorem implies that there exists a unique $C^1$ function $b_n^{{\rm plr}}:[\frac{\pi}{2}+\tw,\pi]\rightarrow (0,\infty)$ such that
\begin{equation*}
  \Gam_{b,L}^{f_n}=\{r=b_n^{\rm plr}(\theta):\frac{\pi}{2}+\tw<\theta<\pi\}.
\end{equation*}

Similarly to $\Gam_{b,L}^{f_n}$, any point $\rx=(x_1,x_2)$ lying on $\Gam_{{\rm sh}, L}^{f_n}$ satisfies the equation $x_1=f_n(x_2)$ which can be rewritten in terms of $(r,\theta)$-coordinates as follows:
\begin{equation*}
  f_n(r\sin\theta)-r\cos\theta-\frac{|P_{3,L}^{f_n}|}{\cos \tw}=0 \quad\tx{for $\frac{\pi}{2}+\tw<\theta<\pi$}.
\end{equation*}
Since $\sin \theta>0$ and $\cot \theta<-\tan \tw$ for $\tw+\frac{\pi}{2}<\theta<\pi$, we obtain from \eqref{definition-iterset-shocks} and \eqref{shock-iterset-condtion1} that
\begin{equation*}
\begin{split}
  \frac{\der}{\der r}\left( f_n(r\sin\theta)-r\cos\theta-\frac{|P_{3,f_n}^L|}{\cos \tw}\right)
  &=f_n'(r\sin \theta)\sin \theta-\cos \theta\\
  &\ge \sin \theta(f_n'(r\sin \theta)+\tan\tw)\\
  &\ge \frac 34 \sin \theta \tan\tw>0\quad\quad\quad\quad\tx{on $\Gam_{{\rm sh}, L}^{f_n}$.}
  \end{split}
\end{equation*}
At $\theta=\pi$, we have
\begin{equation*}
\frac{\der}{\der r}\left( f_n(r\sin\theta)-r\cos\theta-\frac{|P_{3,f_n}^L|}{\cos \tw}\right)=1>0.
\end{equation*}
Therefore, by the implicit function theorem, there exists a unique $C^1$ function $f_n^{{\rm plr}}:[\frac{\pi}{2}+\tw,\pi]\rightarrow (0,\infty)$ such that
\begin{equation*}
  \Gam_{{\rm sh},L}^{f_n}=\{r=f_n^{\rm plr}(\theta):\tw+\frac{\pi}{2}<\theta<\pi\}.
\end{equation*}
It directly follows from \eqref{convergence-shocks} that $|P_{3, L}^{f_n}-P_{3, L}^{f_{\infty}}|$ converges to $0$ as $n$ tends to $\infty$, and this implies that
\begin{equation}
\label{convergence-C1-polar}
  \lim_{n\to \infty}
  \|b_n^{\rm plr}-b_{\infty}^{\rm plr}\|_{C^{1,\frac{\hat{\alp}}{2}}([\frac{\pi}{2}+\tw,\pi])}
  =0,\quad\tx{and}\quad
  \lim_{n\to \infty}
  \|f_n^{\rm plr}-f_{\infty}^{\rm plr}\|_{C^{1,\frac{\hat{\alp}}{2}}([\frac{\pi}{2}+\tw,\pi])}
  =0.
\end{equation}

Next, for $n\in \mathbb{N}\cup\{\infty\}$, we define two invertible mappings
\begin{equation*}
\begin{split}
  &\mcl{T}_n^{\rm plr}(r,\theta):=\left(\frac{f_{\infty}^{\rm plr}(\theta)-b_{\infty}^{\rm plr}(\theta)}{f_{n}^{\rm plr}(\theta)-b_{n}^{\rm plr}(\theta)}\left(r-b_n^{\rm plr}(\theta)\right)+b_{\infty}^{\rm plr}(\theta),\theta\right),\\
  &\mcl{P}_n(\rx):=\left(|\rx-Q_n|, \, \pi+\arctan\left(\frac{x_2}{x_1-\frac{|P_{3,L}^{f_n}|}{\sin \tw}}\right)\right),
  \end{split}
\end{equation*}
and, define a transformation mapping $\mcl{T}_n: \ol{\Om_{f_n,L}}\rightarrow \ol{\Om_{f_{\infty},L}}$ by
\begin{equation*}
  \mcl{T}_n(\rx):=\mcl{P}_{\infty}^{-1}\circ\mcl{T}_n^{\rm plr}\circ \mcl{P}_n(\rx).
\end{equation*}
Each $\mcl{T}_n:\Om_{f_n,L}\rightarrow \Om_{f_{\infty},L}$ is invertible, and \eqref{convergence-C1-polar} yields that
\begin{equation}
\label{convergence-C1-T}
  \lim_{n\to \infty} \|\mcl{T}_n-{\rm Id}\|_{C^{1,\frac{\hat{\alp}}{2}}(\ol{\Om_{f_n, L}})}
  =0,\quad\tx{and} \quad
  \lim_{n\to \infty} \|\mcl{T}^{-1}_n-{\rm Id}\|_{C^{1,\frac{\hat{\alp}}{2}}(\ol{\Om_{f_{\infty}, L}})}=0.
\end{equation}
Furthermore, $\{\mcl{T}_n^{-1}\}$ is uniformly bounded in $C^{2,\hat{\alp}}(K)$ for any compact subset $K$ of $\Om_{f_{\infty},L}$, and we have
\begin{equation}
\label{convergence-C2-T}
  \lim_{n\to \infty}\|\mcl{T}_n^{-1}-{\rm Id}\|_{C^{2,\frac{\hat{\alp}}{2}}(K)}=0.
\end{equation}
For each $n\in \mathbb{N}\cup\{\infty\}$, let us define
\begin{equation*}
  \hat{\psi}_n:=\psi^{f_n}\circ \mcl{T}_n^{-1},\quad
  \tx{and}\quad \hat{f}_n:=\mcl{H}(f_n)
\end{equation*}
for $\psi^{f_n}$ given by \eqref{definition-psi-update}.
The sequence $\{\hat{\psi}_n\}_{n\in \mathbb{N}}$ is bounded in $C^{1,\hat{\alp}}(\ol{\Om_{f_{\infty},L}})$, and bounded in $C^{2,\hat{\alp}}(K)$ for any compact subset $K$ of $\Om_{f_{\infty},L}$. Since each $\phi^{f_n}(=\psi^{f_n}-\psi_0)$ solves the nonlinear boundary value problem \eqref{fixed-nlbvp-bounded-phi} in $\Om_{f_n,L}$, it can be checked by using \eqref{convergence-C1-T}, \eqref{convergence-C2-T}, Arzel\`{a}-Ascoli theorem and the uniqueness of a classical solution to \eqref{fixed-nlbvp-bounded-phi} in $\Om_{f_{\infty},L}$ that the sequence $\{\hat{\psi}_n\}_{n\in \mathbb{N}}$ has a subsequence $\{\hat{\psi}_{n_j}\}$ satisfying that
\begin{equation*}
\lim_{n_j\to \infty}\|\hat{\psi}_{n_j}-\psi^{(f_{\infty})}\|
_{C^{1,\frac{\hat{\alp}}{2}}(\ol{\Om_{f_{\infty},L}})}=0,
\quad\tx{and}
\quad
\lim_{n_j\to \infty}\|\hat{\psi}_{n_j}-\psi^{f_{\infty}}\|
_{C^{2,\frac{\hat{\alp}}{2}}(K)}=0
\end{equation*}
for any compact subset $K$ of $\Om_{f_{\infty},L}$.
And, this implies that the sequence $\{\hat{f}_n\}$ has a subsequence $\{\hat{f}_{n_j}\}$ that converges to $\hat{f}_{\infty}$ in $C^{1,\frac{\hat{\alp}}{2}}([0,L])$.
Since the subsequence $\{\hat{f}_{n_j}\}$ is bounded in $C^{2,\hat{\alp}}_{(-\beta;-(1+\hat{\alp}),\{L\})}((0,L))$, one can further extract a subsequence $\{\hat{f}_{n_{j_k}}\}$ of $\{\hat{f}_{n_j}\}$ so that it converges to a function $\hat{f}_*\in \mcl{J}_{M_1,L}$ in $C^{2,\frac{\hat{\alp}}{2}}_{(-\beta;-(1+\frac{\hat{\alp}}{2}),\{L\})}((0,L))$.
From $C^{1,\frac{\hat{\alp}}{2}}$ convergence of $\{\hat f_{n_j}\}$ to $\hat{f}_{\infty}$, we get $\hat{f}_*=\hat{f}_{\infty}$. Furthermore, this shows that any subsequence of $\{\hat{f}_n\}$ contains a subsequence that converges to $\hat{f}_{\infty}$ in $C^{2,\frac{\hat{\alp}}{2}}_{(-\beta;-(1+\frac{\hat{\alp}}{2}),\{L\})}((0,L))$.
Therefore, we conclude that the iteration mapping $\mcl{H}:\mcl{J}_{M_1,L}\rightarrow \mcl{J}_{M_1,L}$ is continuous in $C^{2,\frac{\hat{\alp}}{2}}_{(-\beta;-(1+\frac{\hat{\alp}}{2}),\{L\})}((0,L))$.
\smallskip

{\textbf{3}.}
To complete the proof of Proposition \ref{proposition-wp-fbpbd}, it remains to fix a small constant $\eps_3\in(0, \eps_2]$, and a large constant $L_*\ge 1$ such that if $(\irho, \iu)\in D_{\infty}(\gam, B_0)$ satisfies $M_{\infty}=\frac{1}{\eps}\ge \frac{1}{\eps_3}$, and if $L\ge L_*$, then Condition \ref{condition2-for-parameters} (ii)--(v) are satisfied under the choices of $M_1$ and $M_2$ given by \eqref{choice-M1} and \eqref{choice-M2}, respectively.

Let $M_1$ and $M_2$ be given by \eqref{choice-M1} and \eqref{choice-M2}, respectively. For $\sigma_2, \hat{\sigma}$ from Lemmas \ref{lemma-distance-iterset}, \ref{lemma-estiamte-uR-3}, define
\begin{equation*}
\sigma_{\sharp}:=
\min\left\{
\frac{\sigma_2}{M_1+1},  \frac{\cot \tw}{4(M_1+1)}, \frac{\tan \tw}{4(M_1+1)}, \frac{\hat{\sigma}}{M_2}
\right\}.
\end{equation*}
By Lemma \ref{lemma-shock-polar2}, one can choose a small constant $\eps_{\sharp}\in(0, \eps_2]$ depending on $d_0$ so that if $\eps\in(0,\eps_{\sharp}]$, then
\begin{equation*}
  \mathfrak{q}_{\gam}(\eps)\le \sigma_{\sharp}.
\end{equation*}
For $\hat{\eps}$ from Lemma \ref{lemma-estiamte-uR-3}, we define
\begin{equation*}
  \eps_3:=\min\{\eps_{\sharp}, \hat{\eps}\}.
\end{equation*}
Finally, we choose $L_*$ as
\begin{equation*}
  L_*=H_0
\end{equation*}
for $H_0$ from Lemma \ref{lemma-estimate-phi-corner}. Under these choices of $(\eps_3, L_*)$, if $\eps\in(0, \eps_3]$ and $L\ge L_*$, then Condition \ref{condition2-for-parameters} (i)--(v) are all satisfied. Therefore all the arguments in the previous steps are valid so that if $\eps\in(0, \eps_3]$ and $L\ge L_*$, then, by Schauder fixed point theorem, the iteration mapping $\mcl{H}:\mcl{J}_{M_1,L}\rightarrow \mcl{J}_{M_1,L}$ has a fixed point $f_*\in \mcl{J}_{M_1,L}$ thus $(f_*, \psi^{f_*})$ solves Problem \ref{fbp-bdd-dmns} for $\psi^{f_*}$ given by \eqref{definition-psi-update}. By setting as $(\fshock, \psi):=(f_*, \psi^{f_*})$, the estimate \eqref{apriori-estimate-fbp-solution-bdd} holds with $\alp=\hat{\alp}$ due to Definition \ref{definition-of-shock-iterset} and Proposition \ref{wellposedness-lbvp-phi-cutoff}. Furthermore, since $f_*\in \mcl{J}_{M_1,L}$, it follows from Lemma \ref{lemma-distance-iterset}(a) that $f_*$ satisfies the estimate \eqref{uniform-lwrbd-detached-distance-bdd} with
$
\delta=\frac 34 d_0.
$
Finally, the subsonicity \eqref{subsonicity-fbp-bdd} can be directly checked by using Lemma \ref{lemma-coefficients}(b) and the fact that $\psi^{f_*}-\psi_0\in \mcl{I}_{M_2, L}^{f_*}$. The proof of Proposition \ref{proposition-wp-fbpbd} is completed.

\end{proof}

\section{The free boundary problem in $\R^2_+\setminus W_b^+$ (Problem \ref{problem-fbp-psi})}
\label{section-main-thm-pf}
\numberwithin{equation}{section}

Now we are ready to solve Problem \ref{problem-fbp-psi} by proving Theorem \ref{main-theorem-psi}, the main theorem of this paper. In proving Theorem \ref{main-theorem-psi}, the key ingredient is Proposition \ref{proposition-wp-fbpbd}, which is proved in the previous section.

\begin{proof}[Proof of Theorem \ref{main-theorem-psi}] The proof is given in four steps.

{\textbf{1.}} First of all, we fix $\bar{\eps}$ as
\begin{equation}\label{bar-eps-choice}
  \bar{\eps}=\eps_3
\end{equation}
for $\eps_3$ from Proposition \ref{proposition-wp-fbpbd}.
For the constant $L_*$ from Proposition \ref{proposition-wp-fbpbd}, set
\begin{equation*}
  L_n:=4(n+L_*)\quad\tx{for $n\in \mathbb{N}$}.
\end{equation*}
For each $n\in \mathbb{N}$,
Proposition \ref{proposition-wp-fbpbd} implies that if $(\irho, \iu)\in D_{\infty}(\gam, B_0)$ satisfies $\frac{1}{M_{\infty}}=\eps$ for $\eps\in(0, \eps_3]$, then Problem \ref{fbp-bdd-dmns} with $L=L_n$ has at least one solution $(f_n ,\psi_n)$ that satisfies the estimate
\begin{equation}
\label{uniform-estimate-seq}
 \|f_n-f_0\|_{2,\alp, (0, L_n)}^{(-\beta;-(1+\alp), \{L_n\})}+ \|\psi_n-\psi_0\|_{2,\alp, \Om_{f_n, L_n}}^{(-\beta;-(1+\alp),\{P_{2,L_n}^{f_n}\})}\le C\mathfrak{q}_{\gam}(\eps)
\end{equation}
for the constants $(C,\alp)$ from \eqref{apriori-estimate-fbp-solution-bdd}.
Since the constant $C$ in \eqref{uniform-estimate-seq} is the same for all $n\in \mathbb{N}$, by applying Arzel\`{a}-Ascoli theorem and the diagonal argument, one can extract a subsequence $\{f_{n_j}\}$ from $\{f_n\}$, and to find a function $f_*: \R_+\rightarrow \R$ so that
 \begin{equation}
 \label{subindex-limit}
  \lim_{n_j\to \infty} \|f_{n_j}-f_*\|_{2,\frac{{\alp}}{2}, (0, R)}=0 \quad{\tx{ for any constant $R>0$.}}
 \end{equation}
Furthermore, it follows from \eqref{free-bc} and \eqref{uniform-estimate-seq} that
\begin{equation}
\label{estimate-f-limiting}
  f_*(0)=b_0-d_0,\quad\tx{and}\quad \|f_*-f_0\|_{2,{\alp}, (0, L_n)}^{(-\beta;-(1+{\alp}), \{L_n\})}\le C\mathfrak{q}_{\gam}(\eps)\quad\tx{for all $n\in \mathbb{N}$}.
\end{equation}
\smallskip

{\textbf{2.}}
Fix $n\in \mathbb{N}$. For each $x_2\in(0, \frac{L_n}{2})$, we have
\begin{equation*}
\frac{\min\{|x_2-L_n|, 1+x_2\}}{1+x_2}\ge \frac 12.
\end{equation*}
Then it follows from Definitions \ref{definition-norms-w-decay} and \ref{definition-norms-f-decay-corner}  that
\begin{equation}
\label{Ln-estimate1}
\|f_*-f_0\|_{2,0, (0, \frac{L_n}{2})}^{(-\beta)}\le 2
\|f_*-f_0\|_{2,0, (0, L_n)}^{(-\beta;-(1+{\alp}), \{L_n\})}\le 2C\mathfrak{q}_{\gam}(\eps).
\end{equation}

For $x_2,x_2'\in (0, \frac{L_n}{2})$ with $0<x_2<x_2'<\frac{L_n}{2}$, set
\begin{equation*}
\begin{split}
&\Delta_{x_2,x_2'}^{(1)}:=\frac{\min\{|x_2-L_n|, |x_2'-L_n|, 1+x_2, 1+x_2'\}}{1+x_2'},\\
&a_1(x_2, x_2')  :=(1+x_2)^{2+\alp-\beta}
\Delta_{x_2,x_2'}^{(1)}\frac{|(f_*-f_0)''(x_2)-(f_*-f_0)''(x_2')|}{|x_2-x_2'|^{\alp}},\\
&a_2(x_2, x_2'):= (1+x_2)^{2+\alp-\beta}
\frac{|(f_*-f_0)''(x_2)-(f_*-f_0)''(x_2')|}{|x_2-x_2'|^{\alp}}.
\end{split}
\end{equation*}
If $|x_2-x_2'|\ge \frac 14(1+x_2')$, then we have
\begin{equation}
\label{estimate-adiff-1}
\begin{split}
  a_2(x_2, x_2')
  &\le 4^{\alp}\left((1+x_2)^{2-\beta}|(f_*-f_0)''(x_2)|
  +(1+x_2')^{2-\beta}|(f_*-f_0)''(x_2')|\right)\\
  &\le 2^{2\alp+1}\|f_*-f_0\|_{2,0, (0, \frac{L_n}{2})}^{(-\beta)}.
\end{split}
\end{equation}
Next, suppose that $|x_2-x_2'|< \frac 14(1+x_2')$. If $0<x_2<\frac{L_n}{10}$, then the assumption of $|x_2-x_2'|< \frac 14(1+x_2')$ implies that $0<x_2<x_2'<\frac{L_n}{4}$ from which we get
\begin{equation*}
\Delta_{x_2,x_2'}^{(1)}
=\frac{\min\{|x_2'-L_n|, 1+x_2\}}{1+x_2'}=
\frac{1+x_2}{1+x_2'}>\frac 34.
\end{equation*}
Then we have
\begin{equation}
\label{estimate-adiff-2}
a_2(x_2, x_2')=a_1(x_2, x_2')(\Delta_{x_2,x_2'}^{(1)})^{-1}
\le \frac 43\|f_*-f_0\|_{2,\alp, (0, L_n)}^{(-\beta, -(1+\alp), \{L_n\})}.
\end{equation}
Under the assumption of $|x_2-x_2'|< \frac 14(1+x_2')$, if $\frac{L_n}{10}\le x_2<x_2'\le \frac{L_n}{2}$, then a direct computation yields that
\begin{equation*}
\Delta_{x_2,x_2'}^{(1)}\ge
\frac{\min\{\frac{L_n}{2}, 1+\frac{L_n}{10}\}}{1+\frac{L_n}{2}}\ge \frac{1}{10}
\end{equation*}
so we get
\begin{equation}
\label{estimate-adiff-3}
a_2(x_2, x_2')=a_1(x_2, x_2')(\Delta_{x_2,x_2'}^{(1)})^{-1}
\le 10\|f_*-f_0\|_{2,\alp, (0, L_n)}^{(-\beta, -(1+\alp), \{L_n\})}.
\end{equation}
So we obtain from \eqref{estimate-f-limiting}--\eqref{estimate-adiff-3} that
\begin{equation}
\label{estimate-fstar-global}
\|f_*-f_0\|_{2,{\alp}, \R^+}^{(-\beta)}\le 10 C\mathfrak{q}_{\gam}(\eps)
\end{equation}
for the constants $(C,\alp)$ from \eqref{apriori-estimate-fbp-solution-bdd}.
For later use in the next step, we also point out that a simple adjustment of the previous argument yields the estimate
\begin{equation}\label{estimate-fn-global}
  \|f_n-f_0\|_{2,{\alp}, (0, \frac{L_n}{2})}^{(-\beta)}\le 10 C\mathfrak{q}_{\gam}(\eps)\quad\tx{for all $n\in \mathbb{N}$}.
\end{equation}
\smallskip

{\textbf{3.}} For each $n$, we remind readers that the point  $P_{3,L_n}^{f_n}$ is given by \eqref{definition-P3L}, Let us set
\begin{equation*}
 x_2^{(n)}:= {\tx{the $x_2$-coordinate of $P_{3,L_n}^{f_n}$}}.
\end{equation*}
A direct computation with using \eqref{definition-P3L}, \eqref{estimate-f-iterset} and \eqref{shock-iterset-condtion1} yields that
\begin{equation*}
\frac{(b_0-d_0+\frac 34 L_n\kappa_w)\kappa_w}{1+\kappa_w^2}\le  x_2^{(n)}\le \frac{(b_0-d_0+\frac 54 L_n \kappa_w)\kappa_w}{1+\kappa_w^2}.
\end{equation*}
Therefore, there exists a sufficiently large constant $N^*\in\mathbb{N}$ so that if $n\ge N^*$, then
\begin{equation}
\label{limit-of-eta-n}
\frac{\sin ^2\tw}{2}L_n\le  x_2^{(n)} \le \frac{3\sin ^2\tw}{2}L_n.
\end{equation}
For each $n\ge N^*$, we set
$\eta_n:=\frac{\sin ^2\tw}{2}L_n$,
and define a transformation mapping $T_n: \Om_{f_n,L_n}\rightarrow \R^2$ by
\begin{equation}
\label{definition-Tn-limiting}
  T_n(x_1,x_2)=\left(\frac{f_*(x_2)-b(x_2)}{f_n(x_2)-b(x_2)}(x_1-b(x_2))+b(x_2), x_2\right)=:(y_1,y_2).
\end{equation}
Each $T_n$ maps the set $\displaystyle{\{\rx\in \Om_{f_n,L_n}:x_2<\eta_n\}}$ onto the domain $\Om_{f_*}^{(n)}$ given by
\begin{equation*}
  \Om_{f_*}^{(n)}:=\{{\rm y}=(y_1,y_2)\in \R^2: f_*(y_2)<y_1<b(y_2),\,\,0<y_2<\eta_n\}.
\end{equation*}
For each $n\in \mathbb{N}$, we also have
\begin{equation*}
\begin{split}
  &T_n(\Gam_{{\rm sh}, L_n}^{f_n}\cap\{x_2<\eta_n\})=\{(f_*(y_2),y_2): 0<y_2<\eta_n\}=:\Gam_{{\rm sh}, n}^{f_*},\\
  &T_n(\Gam_{b, L_n}^{f_n}\cap\{x_2<\eta_n\})=\{(b(y_2),y_2): 0<y_2<\eta_n\}=:\Gam_{b,n}^{f_*}.
  \end{split}
\end{equation*}

By using Definition \ref{definition-bluntbody-ftn}, \eqref{definition-background}, \eqref{nonzero-detached-dist}, \eqref{subindex-limit}, \eqref{estimate-fstar-global} and \eqref{estimate-fn-global}, one can directly check the following properties:
\begin{itemize}
\item[(i)] $T_n$ is invertible in $\Om_{f_n, L_n}\cap\{x_2<\eta_n\}$;

\item[(ii)] There exists a constant $C>0$ depending only on $(\gam, B_0, d_0)$ such that
    \begin{equation*}
      \begin{split}
      &\|D_{\rm y}T_n^{-1}\|_{C^0(\ol{\Om_{f_*}^{(n)}})}\le C,\\
      &\sup_{{\rm y}=(y_1,y_2)\in \Om_{f_*}^{(n)}}|(1+y_2)^{2-\beta}D^2_{\rm y}T_n^{-1}({\rm y})|\le C,\\
      &\sup_{{\rm y}\neq {\rm y}'\in \Om_{f_*}^{(n)}}
      \left(1+\min\{y_2, y_2'\}\right)^{2+{\alp}-\beta}
      \frac{|D^2_{\rm y}T_n^{-1}({\rm y})-D^2_{\rm y}T_n^{-1}({\rm y}')|}{|{\rm y}-{\rm y}'|^{{\alp}}}\le C;
      \end{split}
    \end{equation*}
\item[(iii)]
\begin{equation*}
  \Om_{f_*}^{(N^*)}\subset \Om_{f_*}^{(N^*+1)} \subset \Om_{f_*}^{(N^*+2)}\subset \cdots \subset \Om_{f_*}^{(\infty)}:=\bigcup_{n=N^*}^{\infty} \Om_{f_*}^{(n)},
\end{equation*}
and the set $\Om_{f_*}^{(\infty)}$ is represented as
\begin{equation*}
    \Om_{f_*}^{(\infty)}=\{{\rm y}=(y_1,y_2)\in \R^2: f_*(y_2)<y_1<b(y_2),\,\,y_2>0\}.
\end{equation*}
For any compact subset $K$ of $\ol{\Om_{f_*}^{(\infty)}}$, we have
\begin{equation}
\label{limit-transformation}
  \lim_{n_j\to \infty}\|T_{n_j}^{-1}-{\rm Id}\|_{C^{2,\frac{{\alp}}{2}}(K)}=0
\end{equation}
where the sequence of subindex $\{n_j\}$ is same as the one from \eqref{subindex-limit}.
\end{itemize}
\smallskip

{\textbf{4.}} For each $n\ge N^*$, let us define a function $ \Psi_n:\Om_{f_*}^{(n)}\rightarrow \R$ by
\begin{equation*}
  \Psi_n({\rm y}):=(\psi_n-\psi_0)\circ T_n^{-1}({\rm y})+\psi_0({\rm y})
\end{equation*}
for the function $\psi_0$ given by \eqref{definition-background}.
By using \eqref{uniform-estimate-seq} and the two properties (i) and (ii) of $\{T_n\}_{n\ge N^*}$ stated in the previous step, one can find a constant $C>0$ depending only on $(\gam, B_0, d_0)$ so that, for each $n\ge N^*$,  $\Psi_n$ satisfies the estimate
\begin{equation}
\label{estimate-Psi-n}
  \|\Psi_n-\psi_0\|_{2,{\alp}, \Om_{f_*}^{(n)}}^{(-\beta)}\le C \mathfrak{q}_{\gam}(\eps).
\end{equation}
Therefore one can extract a subsequence $\{\Psi_{n_{j_l}}\}$ of $\{\Psi_{n_j}\}$ by using the Arzel\`{a}-Ascoli theorem and the diagonal argument, and find a function $\psi_*:\Om_{f_*}^{(\infty)}\rightarrow \R$ so that
\begin{equation}
\label{limit-Psi-nl}
  \lim_{n_l\to \infty}\|\Psi_{n_l}-\psi_*\|_{2,\frac{{\alp}}{2},K}=0\,\,\tx{for any compact subset $K$ of $\ol{\Om_{f_*}^{(\infty)}}$},
  \end{equation}
where the subsequence $\{\Psi_{n_j}\}$ is taken for the subindex sequence $\{n_j\}$ from \eqref{limit-transformation}.

And, it directly follows from \eqref{estimate-Psi-n} and \eqref{limit-Psi-nl} that
\begin{equation}
  \label{estimate-final-sol}
  \|\psi_*-\psi_0\|_{2,\hat{\alp}, \Om_{f_*}^{(\infty)}}^{(-\beta)}\le  C\mathfrak{q}_{\gam}(\eps).
  \end{equation}

It remains to check that $(f_*, \psi_*)$ solves Problem \ref{problem-fbp-psi}, and that it satisfies \eqref{shock-IC}--\eqref{apriori-estimate-fbp-solution} stated in Theorem \ref{main-theorem-psi}.
Since each $(f_n, \psi_n)$ solves
Problem \ref{fbp-bdd-dmns}, it follows from \eqref{subindex-limit}, \eqref{limit-transformation} and \eqref{limit-Psi-nl} that $(f_*, \psi_*)$ solves Problem \ref{problem-fbp-psi}.
Next, we check that   $(f_*, \psi_*)$ satisfies \eqref{shock-IC}--\eqref{apriori-estimate-fbp-solution}.
\begin{itemize}
\item[(i)]
It is already shown in \eqref{estimate-f-limiting} that $f_*$ satisfies \eqref{shock-IC};

\item[(ii)]
\eqref{estimate-final-sol} implies that
$\displaystyle{  \lim_{{y_2\to \infty}\atop{{\rm y}\in \Om_{f_*}^{(\infty)} }}|D_{\rm y}\psi_*(\rm y)-D_{\rm y}\psi_0|=0.}$
 This shows that $\psi_*$ satisfies  \eqref{Asymp-BCs-psi};

\item[(iii)]
For each $x_2\ge 0$,  it follows from  \eqref{uniform-lwrbd-detached-distance-bdd} that
\begin{equation*}
b(x_2)-f_*(x_2)=\lim_{n_j\to \infty} b(x_2)-f_{n_j}(x_2)\ge \frac 34d_0.
\end{equation*}
This shows that $f_*$ satisfies \eqref{uniform-lwrbd-detached-distance}  with choosing $\delta$ as $\delta=\frac 34 d_0$;

\item[(iv)]
It directly follows from \eqref{estimate-fstar-global} and \eqref{estimate-final-sol} that $(f_*, \psi_*)$ satisfies the estimate \eqref{apriori-estimate-fbp-solution} for $\hat{\alp}$ given by \eqref{choice-of-alpha}.
\end{itemize}

Therefore, the proof of Theorem \ref{main-theorem-psi} is completed.

\end{proof}

\section{Convexity of detached shocks (Proof of Theorem \ref{main-theorem-physical-var}($b$))}
\label{section-convexity}

For fixed constants $\gam>1$, $B_0>0$, $d_0>0$, and $\beta\in(0,1)$, let $(\eps_3, L_*)$ be from Proposition \ref{proposition-wp-fbpbd}. And, suppose that $(\irho, \iu)\in D_{\infty}(\gam, B_0)$ satisfies $M_{\infty}(=\frac{\iu}{\irho^{(\gam-1)/2}})=\frac{1}{\eps}$ for $\eps\in(0, \eps_3]$. As in the proof of Theorem \ref{main-theorem-psi} given in the previous section, let us set
\begin{equation*}
  L_n:=4(n+L_*)\quad\tx{for $n\in \mathbb{N}$}.
\end{equation*}
For each $n\in \mathbb{N}$, it follows from Proposition \ref{proposition-wp-fbpbd} that Problem  \ref{fbp-bdd-dmns} with $L=L_n$ has a solution $(f_{\rm sh}^{(n)}, \psi^{(n)})$ that satisfies the estimates \eqref{uniform-lwrbd-detached-distance-bdd} and \eqref{apriori-estimate-fbp-solution-bdd}.

\begin{proposition}\label{prop-convex}
There exists a constant $\hat{\eps}\in(0, \eps_3]$ depending only on $(\gam, B_0, d_0)$ so that if $M_{\infty}(=\frac{1}{\eps})\ge \frac{1}{\hat{\eps}}$, then we have
\begin{equation*}
  \frac{d^2}{dx_2^2}f_{\rm{sh}}^{(n)}(x_2)\ge 0
\end{equation*}
for  $0< x_2< L_n$.
\end{proposition}
Once we prove Proposition \ref{prop-convex}, the following result is directly obtained from passing to the limit $n\to \infty$.
\begin{theorem}
  \label{theorem-convexity-unbounded}
  There exists a constant $\hat{\eps}\in(0, \eps_3]$ depending only on $(\gam, B_0, d_0)$ so that if $M_{\infty}(=\frac{1}{\eps})\ge \frac{1}{\hat{\eps}}$, then the solution $(\fshock, \psi)$ to Problem \ref{problem-fbp-psi}, constructed in the proof of Theorem \ref{main-theorem-psi} in \S \ref{section-main-thm-pf} satisfies
  \begin{equation*}
    \frac{d^2}{dx_2^2}f_{\rm{sh}}(x_2)\ge 0\quad\tx{for  $x_2> 0$.}
  \end{equation*}
\end{theorem}
Once Theorem \ref{theorem-convexity-unbounded} is proved, then the proof of Theorem \ref{main-theorem-physical-var}($b$) is given by adapting the proof of Theorem \ref{main-theorem-physical-var}($a$), which is given at the end of \S \ref{section-stream-function}.
Therefore, the rest of this section is devoted to proving Proposition \ref{prop-convex}.

For the rest of the section, we assume that $n\in \mathbb{N}$ is fixed. And, let $(f_{\rm sh}, \psi, \Om, \Gam_{{\rm{sh}}}, \Gam_{b}, \Gam_{{\rm cutoff}})$ denote $(f_{\rm sh}^{(n)}, \psi^{(n)},\Om_{f_{\rm sh}^{(n)},L_n}, \Gam_{{\rm{sh}},L_n}^{f_{\rm sh}^{(n)}}, \Gam_{b,L_n}^{f_{\rm sh}^{(n)}}, \Gam_{{\rm cutoff}, L_n}^{f_{\rm sh}^{(n)}})$.  Also, let $(P_{2,L}^f, P_{3,L}^f, L)$ denote $(P_{2,L_n}^{f_{\rm sh}^{(n)}}, P_{3,L_n}^{f_{\rm sh}^{(n)}}, L_n)$. Here,  $(\Om_{f_{\rm sh}^L,L}, \Gam_{{\rm{sh}},L}^{f_{\rm sh}^L}, \Gam_{b,L}^{f_{\rm sh}^L}, \Gam_{{\rm cutoff}, L}^{f_{\rm sh}^L})$ and $(P_{2,L}^{f_{\rm sh}}, P_{3,L}^{f_{\rm sh}})$ are given by Definition \ref{definition-bounded-domain}.

For ${\rx}\in \Om$, let us set
\begin{equation}
\label{definition-phisical-var-bdd}
\begin{split}
& \rho(\rx):=\hat{\rho}(|\nabla\psi(\rx)|^2),\quad
{\bf u}(\rx)=(u_1, u_2)(\rx):=\frac{(\der_{x_2}\psi, -\der_{x_1}\psi)(\rx)}{\rho(\rx)},\\
&c(\rx):={\rho^{(\gam-1)/2}(\rx)}.
\end{split}
\end{equation}
for $\hat{\rho}$ from Lemma \ref{lemma-rho-expression}.

Lemmas \ref{lemma-velocity1-positivity}
--\ref{lemma-q-monotonicity} stated and proved in the following are essential to prove Proposition \ref{prop-convex}. And, a proof of Proposition \ref{prop-convex} is given at the end of this section.

\begin{lemma}
\label{lemma-velocity1-positivity}
For $P_0$ given by \eqref{definition-P0}, $u_1$ satisfies
\begin{equation*}
u_1({\rx})>0\quad\tx{in $\ol{\Om}\setminus \{P_0\}$}\quad\tx{and}\quad u_1(P_0)=0.
\end{equation*}

\begin{proof} The proof is divided into three steps.

{\textbf{1.}}
As shown in the proof of Theorem \ref{main-theorem-physical-var}(a) in \S 3, the boundary condition $\psi=\ipsi$ on $\Gam_{{\rm sh}}$ stated in \eqref{nlbvp-in-fbp-bdd}, and the free boundary condition \eqref{free-bc} imply that
\begin{equation}\label{4.1}
(\rho{\bf u}-\irho{\bf u}_{\infty})\cdot ({\bf u}-{\bf u}_{\infty})
=(\rho u_1-\rho_{\infty}u_{\infty})(u_1-u_{\infty})+\rho u_2^2=0\quad\tx{on $\ol{\Gam_{{\rm sh}}}$}
\end{equation}
for ${\bf u}_{\infty}=(u_{\infty},0)$.
By using Lemma \ref{lemma-rho-expression} and the estimate \eqref{apriori-estimate-fbp-solution-bdd},
one can directly check that
\begin{align}
\label{estimate-vel-bdd}
&\rho>\rho_{\rm sonic}(=\left(\frac{2(\gam-1)B_0}{\gam+1}\right)^{\frac{1}{\gam-1}}).
\end{align}
Since $\rho_{\rm sonic}>\irho$, \eqref{estimate-vel-bdd} implies that
\begin{equation}
\label{rho-relation}
 \rho> \irho>0\quad\tx{in $\ol{\Om}$}.
\end{equation}
Since $(\rho, {\bf u})$ satisfies \eqref{RH-potential} on $\ol{\shock}$ with $(\rho^+, {\bf u}^+)=(\rho, {\bf u})$, and $(\rho^-, {\bf u}^-)=(\irho, {\bf u}_{\infty})$, \eqref{rho-relation} implies that $u_1\le |{\bf u}|<\iu$ on $\ol{\shock}$. Then we obtain from \eqref{4.1} that $\displaystyle{\rho u_1-\irho\ui=-\frac{\rho u_2^2}{u_1-\ui}\ge 0}$ on $\ol{\Gam_{{\rm sh}}}$, from which it follows that
\begin{equation}\label{5.3x}
u_1>0\qquad\mbox{on }\ol{\shock}.
\end{equation}
\smallskip

{\textbf{2.}}
If we show that $\psi_{x_2}>0$ in $\Om$, then it follows from \eqref{definition-phisical-var-bdd} and \eqref{estimate-vel-bdd} that $u_1>0$ in $\Om$.

Let us set
\begin{equation*}
  w:=\psi_{x_2}.
\end{equation*}
By bootstrap arguments and standard Schauder estimate theory, it can be checked that $\psi$ is $C^{\infty}$ in $\Om\cup {\rm int} \Gam_{{\rm{cutoff}}}$.
For $\rho$ given by \eqref{definition-phisical-var-bdd}, $\psi$ satisfies the Bernoulli's law $G(\rho, |\nabla\psi|^2)=B_0$ for $G$ given by \eqref{E-system-str-form} in $\ol{\Om}$.
We differentiate $G(\rho, |\nabla\psi|^2)=B_0$ with respect to $x_2$ in $\Om$ to obtain that
\begin{equation}
\label{rho-deriv-expression}
\partial_{x_2}\rho
=\frac{u_2w_{x_1}-u_1w_{x_2}}{c^2-|{\bf u}|^2}.
\end{equation}
Next, we differentiate Eq. 
\eqref{equation-psi-nondiv}, and use \eqref{rho-deriv-expression} to get
\begin{equation}
\label{4.4a-pre}
\begin{split}
&(c^2-u_1^2)w_{x_1x_1}-2u_1u_2w_{x_1x_2}+(c^2-u_2^2)w_{x_2x_2}
+\frac{A_1\psi_{x_1x_1}+2A_2w_{x_1}+A_3w_{x_2}}{\rho(c^2-|{\bf u}|^2)}=0
\end{split}
\end{equation}
for
\begin{equation*}
  \begin{split}
  &A_1=u_2((\gam-1)c^2+2u_1^2)w_{x_1}-u_1((\gam-1)c^2+2(c^2-u_2^2))w_{x_2},\\
  &A_2=u_1(c^2-u_1^2+u_2^2)w_{x_1}-u_2(c^2+u_1^2-u_2^2)w_{x_2},\\
  &A_3=u_2((\gam-1)c^2-2(c^2-u_1^2))w_{x_1}-u_1((\gam-1)c^2-2u_2^2)w_{x_2}.
  \end{split}
\end{equation*}
Due to \eqref{subsonicity-fbp-bdd}, we can rewrite Eq. \eqref{equation-psi-nondiv} as
\begin{equation}
\label{psi-x1x1-w}
  \psi_{x_1x_1}=\frac{2u_1u_2 w_{x_1}-(c^2-u_2^2)w_{x_2}}{c^2-u_1^2}.
\end{equation}
By substituting this expression into \eqref{4.4a-pre}, we obtain an equation for $w$ in the form of
\begin{equation}\label{4.4a}
  (c^2-u_1^2)w_{x_1x_1}-2u_1u_2w_{x_1x_2}+(c^2-u_2^2)w_{x_2x_2}
  +a_1w_{x_1}+a_2w_{x_2}=0\quad\tx{in $\Om$}
\end{equation}
for $a_{1,2}=a_{1,2}(\nabla\psi, w_{x_1}, w_{x_2})$. It directly follows from \eqref{subsonicity-fbp-bdd} that Eq. \eqref{4.4a} is uniformly elliptic in $\Om$. Therefore, the minimum of $w$ over $\ol{\Om}$ is acquired on $\der \Om$.
\smallskip

{\textbf{3.}}
The equation ${\rm{div}}\left(\frac{\nabla\psi}{\hat{\rho}(|\nabla\psi|^2)}\right)=0$ is uniformly elliptic in $\Om$ due to \eqref{subsonicity-fbp-bdd}, thus it follows from \eqref{nlbvp-in-fbp-bdd}, the minimum principle and Hopf's lemma that $\psi\ge \min\{0,\ipsi|_{\shock}\}$ in $\ol{\Om}$.
Since $\shock\subset \{x_2\ge 0\}$, \eqref{definition-ipsi} implies that $\ipsi\ge 0$ on $\Gam_{\rm sh}$ therefore we have $\psi \ge 0$ in $\ol{\Om}$.
Since $\psi$ is not a constant in $\ol{\Om}$, we obtain from the Hopf's lemma and the boundary condition $\psi=0$ on $\Gam_{\rm sym}\cup \Gam_b$ that
 \begin{equation}
 \label{deriv-psi-wall-sym-bdry}
   \psi_{{\bm \nu}}< 0\quad\tx{and}\quad
    \psi_{{\bm \tau}}=0
    \quad\tx{on $(\Gam_{\rm sym}\cup\Gam_{b})\setminus \{P_0\}$}
 \end{equation}
for the outward unit normal ${\bm \nu}$ and ${\bm\tau}=\bm{\nu}^{\perp}$ on $(\Gam_{\rm sym}\cup\Gam_{b})\setminus \{P_0\}$. On $\Gam_{\rm sym}\setminus\{P_0\}$, we have ${\bm \nu}\cdot {\bf e}_{x_2}=-1<0$. On $\Gam_b\setminus \{P_0\}$, it follows from $(b_3)$ in Definition \ref{definition-bluntbody-ftn} that ${\bm \nu}\cdot {\bf e}_{x_2}=\frac{-b'}{\sqrt{1+(b')^2}}<0$. Then \eqref{deriv-psi-wall-sym-bdry} implies that
\begin{equation*}
  w=\nabla\psi\cdot{\bf e}_2=\psi_{\bm\nu}\, {\bm\nu}\cdot {\bf e}_2>0\quad \tx{on $(\Gam_{\rm sym}\cup\Gam_{b})\setminus \{P_0\}$}.
\end{equation*}

Since $\psi$ is $C^1$ up to $\der \Om$, the boundary condition $\psi=0$ on $\Gam_{\rm sym}\cup \Gam_b$ combined with \eqref{b-property1} stated in Remark \ref{remark-b-function} implies that
\begin{equation*}
  \nabla\psi(P_0)=0,
\end{equation*}
from which it follows that $w(P_0)=0$. Note that the flow speed becomes $0$ at $P_0$, that is, $P_0$ is a {\emph{stagnation point}}.

On $\Gam_{{\rm{cutoff}}}$,  the outward unit normal ${\bm\nu}$ is given by ${\bm\nu}=(\cos\theta_w,\sin\theta_w)$. By differentiating the boundary condition $\psi_{\bm\nu}=0$ on $\Gam_{{\rm{cutoff}}}$ in the tangential direction ${\bm\tau}=(-\sin\theta_w,\cos\theta_w)$, we obtain that
\begin{equation}\label{4.5a}
\psi_{\bm\nu \bm\tau}=
\sin \tw\cos\tw(-\psi_{x_{1}x_1}+w_{x_2})
+(\cos^2\theta_w-\sin^2\theta_w) w_{x_1}=0,\qquad\mbox{on }\Gam_{{\rm{cutoff}}}.
\end{equation}
By substituting the expression \eqref{psi-x1x1-w} into \eqref{4.5a}, we get
\begin{equation}
\label{bc-w-cutoff}
\beta_1w_{x_1}+\beta_2w_{x_2}=0 \quad\tx{on $\Gam_{\rm{cutoff}}$}
\end{equation}
for ${\bm\beta}=(\beta_1, \beta_2)$ given by
\begin{equation*}
  \begin{split}
  &\beta_1=(\cos^2\tw-\sin^2\tw)(c^2-u_1^2)-2u_1u_2\cos\tw\sin\tw,\\
  &\beta_2=\cos\tw\sin\tw\left((c^2-u_1^2)+(c^2-u_2^2)\right).
  \end{split}
\end{equation*}
By \eqref{subsonicity-fbp-bdd} and \eqref{estimate-vel-bdd}, we have
\begin{equation*}
  {\bm\beta}\cdot{\bm\nu}=\cos \tw (c^2-({\bf u}\cdot {\bm\nu})^2)\ge \frac{2(\gam-1)B_0\cos \tw}{\gam+1} \sigma>0\quad\tx{on $\Gam_{\rm{cutoff}}$}
\end{equation*}
for $\sigma>0$ from \eqref{subsonicity-fbp-bdd}. Therefore the boundary condition \eqref{bc-w-cutoff} for $w$ is uniformly oblique on $\Gam_{\rm{cutoff}}$.

By the strong minimum principle, the Hopf's lemma and \eqref{5.3x}, we conclude that
\begin{equation*}
  w>0\quad\tx{in $\ol{\Om}\setminus \{P_0\}$},\quad w(P_0)=0,
\end{equation*}
and this proves Lemma \ref{lemma-velocity1-positivity}.
\end{proof}
\end{lemma}


\begin{lemma}
\label{lemma-velocity2-positivity}
There exists a constant $\eps_4\in(0, \eps_3]$ depending only on $(\gam, B_0, d_0)$ so that, for any $L\ge L_*$ and $M_{\infty}(=\frac{1}{\eps})\ge \frac{1}{\eps_4}$, the vertical velocity $u_2(\rx)$ given by \eqref{definition-phisical-var-bdd} satisfies
\begin{equation*}
u_2({\rm x})>0\quad\tx{in $\ol{\Om}\setminus \ol{\Gam_{\rm sym}}$},\quad\tx{and}\quad u_2({\rm x})=0\quad\tx{on $\ol{\Gam_{\rm sym}}$}.
\end{equation*}

\begin{proof} The proof is divided into four steps.

{\textbf{1.}}
Let us set
\begin{equation*}
  W:=\psi_{x_1}.
\end{equation*}
Since $\rho>\rho_{\rm sonic}$ in $\ol{\Om}$, it suffices to show that $W<0$ in $\ol{\Om}\setminus \ol{\Gam_{\rm sym}}$, and $W=0$ on $\ol{\Gam_{\rm sym}}$ to prove
Lemma \ref{lemma-velocity2-positivity}.

Similarly to \eqref{4.4a} in the proof of Lemma \ref{lemma-velocity1-positivity}, one can check that $W$ satisfies the uniformly elliptic equation
\begin{equation*}
  (c^2-u_1^2)W_{x_1x_1}-2u_1u_2W_{x_1x_2}+(c^2-u_2^2)W_{x_2x_2}
  +b_1W_{x_1}+b_2W_{x_2}=0\quad\tx{in $\Om$}
\end{equation*}
for $b_{j}=b_{j}(\nabla\psi, W_{x_1}, W_{x_2})$, j=1,2.
\smallskip

{\textbf{2.}} By following the argument in the proof of Theorem \ref{main-theorem-physical-var}(a), one can check that
\begin{equation}\label{u-ui-diff}
  |{\bf u}|-|{\bf u}_{\infty}|\le -\mu\quad\tx{on $\shock$}
\end{equation}
for some constant $\mu>0$.
Then, the free boundary condition \eqref{free-bc} implies that the unit normal ${\bm\nu}_{\rm sh}$ on $\Gam_{\rm sh}$ with pointing interior to $\Om$ is given by
$  {\bm\nu}_{\rm sh}=\frac{{\bf u}_{\infty}-{\bf u}}{|{\bf u}_{\infty}-{\bf u}|}
$. Then the vector field ${\bm\tau}_{\rm sh}$ given by ${\bm\tau}_{\rm sh}:={\bm\nu}_{\rm sh}^{\perp}(=\frac{(u_2, \iu-u_1)}{|{\bf u}_{\infty}-{\bf u}|})$ yields a tangential vector field along $\shock$.

By using \eqref{definition-phisical-var-bdd}, we rewrite \eqref{4.1} as
\begin{equation*}
  (\psi_{x_2}-\irho\ui)(\frac{\psi_{x_2}}{\rho}-\ui)
  +\frac{\psi_{x_1}^2}{\rho}=0\quad \tx{on $\Gam_{\rm sh}$}.
\end{equation*}
By differentiating this expression along $\Gam_{\rm sh}$ in the direction of ${\bm\tau}_{\rm sh}$, and by using \eqref{definition-phisical-var-bdd}, we get
\begin{equation}\label{4.2}
\left(2u_1-(1+\frac{\rho_{\infty}}{\rho})u_{\infty}\right)
\partial_{\bm\tau_{\rm{sh}}}(\psi_{x_2})
-\left(|{\bf u}|^2-\frac{\rho_{\infty}}{\rho}u_1u_{\infty}\right)
\partial_{\bm\tau_{\rm{sh}}}\rho-2u_2\partial_{\bm\tau_{\rm{sh}}}W=0
\quad\tx{on $\Gam_{\rm sh}$}.
\end{equation}
Similarly to \eqref{rho-deriv-expression}, by differentiating the Bernoulli's law $G(\rho, |\nabla\psi|^2)=B_0$ on $\shock$ in the direction of ${\bm \tau_{\rm{sh}}}$ and using \eqref{definition-phisical-var-bdd}, we get
\begin{equation}\label{4.2a}
\partial_{\bm\tau_{\rm{sh}}}\rho=
\frac{u_2\partial_{\bm\tau_{\rm{sh}}}W
-u_1\partial_{\tau_{\rm{sh}}}(\psi_{x_2})}{c^2-|{\bf u}|^2}\quad\tx{on $\shock$.}
\end{equation}
Next, we find an expression $\der_{\bm\tau_{\rm sh}}(\psi_{x_2})$ in terms of $(W_{x_1}, W_{x_2})$. By the definition of ${\bm \tau}_{\rm sh}$, we have
\begin{equation}
\label{psi-x2-tau-deriv}
  \der_{\bm\tau_{\rm sh}}(\psi_{x_2})
  =\frac{u_2W_{x_2}+(\ui-u_1)\psi_{x_2x_2}}{|{\bf u}-{\bf u}_{\infty}|} \quad\tx{on $\shock$.}
\end{equation}
We substitute the expression
$
\psi_{x_2x_2}=\frac{2u_1u_2 W_{x_2}-(c^2-u_1^2)W_{x_1}}{c^2-u_2^2}
$, directly given from Eq. \eqref{equation-psi-nondiv}, into \eqref{psi-x2-tau-deriv} to get
\begin{equation*}
  \der_{\bm\tau_{\rm sh}}(\psi_{x_2})
  =\frac{1}{|{\bf u}-{\bf u}_{\infty}|}
  \left(u_2W_{x_2}+\frac{(\ui-u_1)(2u_1u_2 W_{x_2}-(c^2-u_1^2)W_{x_1})}{c^2-u_2^2}\right).
\end{equation*}
By inverting the relation
\begin{equation}
\label{derivative-on-shock}
  \begin{pmatrix}
  \der_{\bm\nu_{\rm sh}}\\
  \der_{\bm\tau_{\rm sh}}
  \end{pmatrix}
  =\frac{1}{|{\bf u}_{\infty}-{\bf u}|}\begin{pmatrix}
  \ui-u_1 & -u_2\\
  u_2 & \ui-u_1
  \end{pmatrix}
  \begin{pmatrix}
  \der_{x_1}\\
  \der_{x_2}
  \end{pmatrix},
\end{equation}
we get
\begin{equation*}
  W_{x_1}=\frac{(\ui-u_1)W_{\bm\nu_{\rm sh}}+u_2W_{\bm\tau_{\rm sh}}}{|{\bf u}-{\bf u}_{\infty}|},\quad
  W_{x_2}=\frac{-u_2W_{\bm\nu_{\rm sh}}+(\ui-u_1)W_{\bm\tau_{\rm sh}}}{|{\bf u}-{\bf u}_{\infty}|}.
\end{equation*}
So we can rewrite \eqref{psi-x2-tau-deriv} as
\begin{equation}
\label{psi-x2-tau-deriv-2}
  \der_{\bm\tau_{\rm sh}}(\psi_{x_2})=
  \frac{\alp_1W_{\bm\nu_{\rm sh}}+\alp_2W_{\bm\tau_{\rm sh}}}{|{\bf u}_{\infty}-{\bf u}|^2} \quad\tx{on $\shock$}
\end{equation}
with
\begin{equation*}
  \alp_1=-\frac{1}{c^2-u_2^2}\begin{pmatrix}\ui-u_1\\ -u_2\end{pmatrix}^T
  \begin{pmatrix}
  c^2-u_1^2 &u_1u_2\\
  u_1u_2 &c^2-u_2^2
  \end{pmatrix}
  \begin{pmatrix}\ui-u_1\\ -u_2\end{pmatrix}.
\end{equation*}
The term $\alp_2$ can be directly computed as well, but since $\alp_1$ is essential in proving Lemma \ref{lemma-velocity2-positivity}, we only give the exact value of $\alp_1$ here.
The eigenvalues of the matrix $\begin{pmatrix}
  c^2-u_1^2 &u_1u_2\\
  u_1u_2 &c^2-u_2^2
  \end{pmatrix}$ are $c^2$ and $c^2(1-M^2)$ for $M^2:=\frac{|{\bf u}|^2}{c^2}$. So it follows from Proposition \ref{proposition-wp-fbpbd}(iii) and \eqref{u-ui-diff} that, there exists a constant $\mu_0>0$ such that
\begin{equation}
\label{oblique-constant}
  \alp_1\le -\mu_0\quad\tx{on ${\shock}$.}
\end{equation}
We substitute the expression \eqref{4.2a} into \eqref{4.2}, then use the expression \eqref{psi-x2-tau-deriv-2} to rewrite \eqref{4.2} as
\begin{equation}
\label{bc-for-W}
  \beta_1W_{\bm\nu_{\rm sh}}+\beta_2W_{\bm \tau_{\rm sh}}=0\quad\tx{on $\Gam_{\rm sh}$}
\end{equation}
for $\beta_{1,2}=\beta_{1,2}(\nabla\psi)$. Particularly, the coefficient $\beta_1$ is given by
\begin{equation*}
  \beta_1=-\frac{\alp_1}{|{\bf u}-{\bf u}_{\infty}|}
  \left(\ui\left(1+\frac{\irho}{\rho}(1-\frac{u_1^2}{c^2-|{\bf u}|^2})
  \right)-u_1\left(2+\frac{|{\bf u}|^2}{c^2-|{\bf u}|^2}\right)\right).
\end{equation*}
By Lemma \ref{lemma-shock-polar2},
Proposition \ref{proposition-wp-fbpbd}(ii)--(iii), \eqref{definition-phisical-var-bdd},
\eqref{estimate-vel-bdd}, \eqref{u-ui-diff}, and \eqref{oblique-constant}, there exist constants $\mu_1>0$ and  $C>0$ with $C$ depending only on $(\gam, B_0, d_0)$ such that
\begin{equation*}
  \beta_1\ge \mu_1\left(\ui-C\mathfrak{q}_{\gam}(\eps)\right)\quad\tx{on $\Gam_{\rm sh}$}
\end{equation*}
for $\mathfrak{q}_{\gam}(\eps)$ from Lemma \ref{lemma-shock-polar2}.
Therefore there exists a constant $\eps_{4}\in(0, \eps_{3}]$ depending only on $(\gam, B_0, d_0)$ so that if $\eps\in(0,\eps_4]$, then we have $\beta_1\ge \frac{\mu_1\iu}{2}>0$ on $\shock$ from which it follows that the boundary condition \eqref{bc-for-W} is strictly oblique on $\Gam_{\rm sh}$.
\smallskip

{\textbf{3.}}
Similarly to \eqref{bc-w-cutoff}, one can derive from \eqref{4.5a} that
\begin{equation}
\label{bc-W-cutoff}
  \kappa_1 W_{x_1}+\kappa_2 W_{x_2}=0 \quad
  \tx{on $\Gam_{\rm cutoff}$ }
\end{equation}
for $\bm\kappa=(\kappa_1, \kappa_2)$ given by
\begin{equation*}
  \begin{split}
  &\kappa_1=-\left((c^2-u_1^2)+(c^2-u_2^2)\right)\cos\tw\sin\tw ,\\
  &\kappa_2=(c^2-u_2^2)(\cos^2\tw-\sin^2\tw)+2 u_1u_2\cos\tw\sin\tw.
  \end{split}
\end{equation*}
By Proposition \ref{proposition-wp-fbpbd}(iii) and the definition of $\rhosonic$ given in \eqref{definition-rho-sonic}, we have
\begin{equation*}
  {\bm\kappa}\cdot{\bm\nu}=\sin \tw (c^2-({\bf u}\cdot {\bm\nu})^2)\ge \frac{2(\gam-1)B_0\sin \tw}{\gam+1} \sigma>0\quad\tx{on $\Gam_{\rm{cutoff}}$}
\end{equation*}
for the outward unit normal ${\bm\nu}=(\cos\theta_w,\sin\theta_w)$ on $\Gam_{\rm{cutoff}}$, where $\sigma>0$ is from Proposition \ref{proposition-wp-fbpbd}(iii). Therefore the boundary condition \eqref{bc-W-cutoff} for $W$ is strictly oblique on $\Gam_{\rm{cutoff}}$.
\smallskip

{\textbf{4.}} On $\Gam_b\setminus\{(b_0,0)\}$, we differentiate the boundary condition $\psi=0$ along the tangential direction ${\bm\tau}_b:=\frac{(b'(x_2),1 )}{\sqrt{1+(b'(x_2))^2}}$, and use \eqref{definition-phisical-var-bdd} to get
$W=\frac{-\rho u_1}{b'(x_2)}$ on $\Gam_b$.
Then it follows from $(b_3)$ of Definition \ref{definition-bluntbody-ftn}, Lemma \ref{lemma-velocity1-positivity} and \eqref{rho-relation} that
$W<0$ on $\Gam_b\setminus \{(b_0,0)\}$.
Finally, we differentiate the boundary condition $\psi=0$ with respect to $x_1$ on $\Gam_{{\rm sym}}$ to obtain that
\begin{equation}
\label{W-estimate-1}
 W=0\quad\tx{on $\ol{\Gam_{\rm sym}}$.}
\end{equation}
Then, by the strong maximum principle and the Hopf's lemma, we conclude that
\begin{equation}
\label{W-estimate-2}
  W<0\quad\tx{in $\ol{\Om}\setminus \ol{\Gam_{\rm sym}}$}
\end{equation}
provided that $\eps\in (0, \eps_4]$. The estimates \eqref{W-estimate-1} and \eqref{W-estimate-2} combined with \eqref{estimate-vel-bdd} completes the proof of Lemma \ref{lemma-velocity2-positivity}.
\end{proof}
\end{lemma}

For ${\bf u}=(u_1, u_2)$ given by \eqref{definition-phisical-var-bdd}, let us set
\begin{equation}
\label{definition-q}
  q({\rm x}):=|{\bf u}(\rx)|(=
  \sqrt{u_1^2({\rm x})+u_2^2(\rm x)})\quad\tx{in $\ol{\Om}$},
\end{equation}
that is, $q({\rm x})$ is the local speed at ${\rm x}\in \ol{\Om}$. Then, the Mach number $M$ is written as $M=\frac qc$.
By Lemmas \ref{lemma-velocity1-positivity} and \ref{lemma-velocity2-positivity}, we have $q>0$ in $\ol{\Om}\setminus\{P_0\}$, thus we can define $\Theta(\rx)$ and $Q$ by
\begin{equation}
\label{u-into-q-Theta}
\Theta:=\arctan \frac{u_2}{u_1},\quad
Q:=\ln q\quad\tx{in $\ol{\Om}\setminus\{P_0\}$.}
\end{equation}

\begin{lemma}
\label{lemma-critical-val-speed}
One can further reduce $\eps_4$  from Lemma \ref{lemma-velocity2-positivity} so that if $M_{\infty}(=\frac{1}{\eps})\ge \frac{1}{\eps_4}$, then the local speed $q$ cannot attain its local extrema on $\Om\cup \ol{\Gam_{{\rm sh}}}\cup \ol{\Gam_{{\rm cutoff}}}$.

\begin{proof}
The proof is divided in six steps.

{\textbf{1.}}
Note that $(\rho, {\bf u})$ given by \eqref{definition-phisical-var-bdd} satisfy the equations stated in \eqref{E-system}. In order to compute a second order differential equation for $Q$ in $\Om$, we take the following steps:
\begin{itemize}
\item[(i)] By differentiating the Bernoulli's law
$
\frac 12 q^2+\mathfrak{h}(\rho)=B_0
$ in the direction of ${\bf u}$,
we get the expression
    \begin{equation}
    \label{deriv-Bernoulli}
      {\bf u}\cdot \nabla \rho=-\frac{q}{\mathfrak{h}'(\rho)}{\bf u}\cdot \nabla q\quad\tx{in $\Om$.}
    \end{equation}

\item[(ii)] Due to \eqref{estimate-vel-bdd}, we can  rewrite the first equation in \eqref{E-system} as $\nabla \cdot {\bf u}+\frac{{\bf u}\cdot \nabla \rho}{\rho}=0$. Then, we substitute the expression \eqref{deriv-Bernoulli} into  $\nabla \cdot {\bf u}+\frac{{\bf u}\cdot \nabla \rho}{\rho}=0$ to get
    \begin{equation}
    \label{q-equation1}
      \frac 1q\nabla \cdot {\bf u}-\frac{1}{c^2}{\bf u}\cdot \nabla q=0\quad\tx{in $\Om$}.
    \end{equation}

\item[(iii)] Next, we use the definition \eqref{u-into-q-Theta} to rewrite \eqref{q-equation1} and the second equation in \eqref{E-system} in terms of $(Q,\Theta)$ in $\Om$ as follows:
    \begin{equation*}
\begin{cases}
      \Theta_{x_1}\sin \Theta-\Theta_{x_2}\cos \Theta=(1-M^2)(Q_{x_1}\cos \Theta+Q_{x_2}\sin \Theta)\\
      \Theta_{x_1}\cos \Theta+\Theta_{x_2}\sin \Theta
      =Q_{x_2}\cos \Theta-Q_{x_1}\sin \Theta
    \end{cases}
    \end{equation*}
\end{itemize}
from which we obtain that
\begin{equation}\label{4.3}
\begin{cases}
\Theta_{x_1}=
-Q_{x_1} M^2\sin\Theta\cos\Theta+
Q_{x_2}(1-M^2\sin^2\Theta)\\
\Theta_{x_2}=-Q_{x_1}(1-M^2\cos^2\Theta)+
Q_{x_2} M^2\sin\Theta\cos\Theta
\end{cases}\quad\tx{in $\Om$.}
\end{equation}
Since $\Theta$ is $C^2$ in $\Om$, \eqref{4.3} yields that
\begin{equation}
\label{equation-for-Q}
\begin{split}
\sum_{i,j=1}^2 \der_{x_i}(a_{ij}\der_{x_j}Q)=0\quad\tx{in $\Om$}
\end{split}
\end{equation}
with
\begin{equation*}
  a_{11}=1-M^2\cos^2\Theta,\quad a_{12}=-a_{21}=M^2\sin\Theta\cos\Theta,\quad
  a_{22}=1-M^2\sin^2\Theta.
\end{equation*}
By Proposition \ref{proposition-wp-fbpbd}(iii), Eq.\eqref{equation-for-Q} is uniformly elliptic in $\Om$. And, it follows from Lemmas \ref{lemma-velocity1-positivity} and \ref{lemma-velocity2-positivity} that $q=0$ at ${\rx}=P_0$, and $q(\rx)>0$ in $\ol{\Om}\setminus \{P_0\}$. This implies that $Q$ is not a constant in $\Om$. Therefore, by the strong maximum principle,  $Q$ cannot attain its local extrema in $\Om$, or equivalently $q(=e^{Q})$ cannot attain its local extrema in $\Om$.
\smallskip

{\textbf{2.}} By \eqref{estimate-vel-bdd}, Lemmas \ref{lemma-velocity1-positivity} and  \ref{lemma-velocity2-positivity}, we have $\rho q>0$ on $\Gam_{{\rm{cutoff}}}$. So the boundary condition $\nabla\psi\cdot{\bf n}_c=0$ on $\Gam_{{\rm{cutoff}}}$, stated in \eqref{nlbvp-in-fbp-bdd}, implies that $\Theta-\tw=m\pi$ for some $m\in \mathbb{Z}$ so we have
\begin{equation*}
 \nabla\Theta\cdot {\bm\tau}_{\rm c}=0\quad\tx{on $\Gam_{{\rm{cutoff}}}$}
\end{equation*}
for a unit tangential ${\bm\tau}_{\rm c}$ along $\Gam_{{\rm{cutoff}}}$. Furthermore, we can represent ${\bm\tau}_{\rm c}$ as ${\bm\tau}_{\rm c}=(\cos \Theta, \sin \Theta)^{\perp}$ on $\Gam_{{\rm{cutoff}}}$. We substitute the expression \eqref{4.3} into the boundary condition for $\Theta$ stated right above, then apply Proposition \ref{proposition-wp-fbpbd}(iii) to get
\begin{equation}\label{4.18}
\nabla Q\cdot {\bf n}_c=0\quad\mbox{on }\Gam_{{\rm{cutoff}}},
\end{equation}
where ${\bf n}_c$ represents a unit normal
on $\Gam_{{\rm{cutoff}}}$. Therefore, by the Hopf's lemma, $q(=e^Q)$ cannot attain its local extrema on $\Gam_{{\rm{cutoff}}}$.
\smallskip

{\textbf{3.}} As pointed out eariler in the proof of Lemma \ref{lemma-velocity2-positivity},
\begin{equation}\label{definition-nu-shock}
  {\bm\nu}_{\rm sh}=\frac{{\bf u}_{\infty}-{\bf u}}{|{\bf u}_{\infty}-{\bf u}|}
\end{equation}
is the unit normal on $\shock$ with pointing interior to $\Om$. So ${\bm\tau}_{\rm sh}={\bm\nu}_{\rm sh}^{\perp}(=\frac{(u_2, \iu-u_1)}{|{\bf u}_{\infty}-{\bf u}|})$ yields a unit tangential along $\shock$.
By differentiating the Bernoulli's law
$\frac 12 q^2+\mathfrak{h}(\rho)=B_0$
in the tangential direction ${\bm\tau}_{\rm sh}$ along $\shock$, we get
\begin{equation}
\label{rho-tau-deriv}
  \rho_{{\bm\tau}_{\rm sh}}=-\rho M^2 Q_{{\bm\tau}_{\rm sh}}\quad\tx{on $\shock$}.
\end{equation}
Next, we rewrite \eqref{4.1} in terms of $(\rho, q, \Theta)$ as
\begin{equation}
\label{4.5}
(\rho q\cos\Theta-\rho_{\infty}u_{\infty})(q\cos\Theta-u_{\infty})
+\rho q^2 \sin^2\Theta
=0
\quad\tx{on $\shock$},
\end{equation}
and differentiate this boundary condition in the direction of ${\bm\tau}_{\rm sh}$ along $\shock$. Then, by using the expression \eqref{rho-tau-deriv}, one can derive that
\begin{equation}
\label{Theta-tau1}
 \Theta_{{\bm\tau}_{\rm sh}}=\frac{T_1 Q_{{\bm\tau}_{\rm sh}} }{
 (1+\frac{\irho}{\rho})\ui \sin \Theta}\quad\tx{on ${\rm int}\, \Gam_{\rm sh}$}
\end{equation}
for
\begin{equation*}
  T_1=(M^2-2)q+\ui(1-M^2)\cos \Theta+\frac{\irho\ui}{\rho}\cos \Theta.
\end{equation*}
Since $\sin \Theta>0$ on ${\rm int}\, \Gam_{\rm sh}$ due to Lemma \ref{lemma-velocity2-positivity}, the expression \eqref{Theta-tau1} is well defined.

Another way to compute $ \Theta_{{\bm\tau}_{\rm sh}}$ on $\shock$ is to use \eqref{derivative-on-shock} and \eqref{4.3}.
A lengthy but direct computation yields that
\begin{equation}
\label{Theta-tau2}
  \Theta_{{\bm\tau}_{\rm sh}}
  =\frac{1}{|{\bf u}-{\bf u}_{\infty}|^2}\left(a_{\nu} Q_{{\bm\nu}_{\rm sh}}+a_{\tau}Q_{{\bm\tau}_{\rm sh}}\right)\quad\tx{on ${\rm int}\,\Gam_{\rm sh}$}
\end{equation}
for
\begin{equation}
\label{definition-anu-atau}
  \begin{split}
  &a_{\nu}=-\frac{|{\bf u}_{\infty}-{\bf u}|^2}{c^2}(c^2-({\bf u}\cdot {\bm \nu}_{\rm sh})^2),\\
  &a_{\tau}=M^2 \iu \sin \Theta (\iu\cos \Theta-q).
  \end{split}
\end{equation}

From \eqref{Theta-tau1} and \eqref{Theta-tau2}, we obtain that
\begin{equation}
\label{bc-for-Q-on-shock}
{\bm\beta}_{\rm sh}\cdot \nabla Q=0 \quad\tx{on ${\rm int}\,\Gam_{\rm sh}$}.
\end{equation}
for
\begin{equation}
\label{definition-bc-Q-on-shock}
{\bm\beta}_{\rm sh}:=
-\frac{a_{\nu}}{|{\bf u}_{\infty}-{\bf u}|^2}{\bm\nu}_{\rm sh}
+
\left(\frac{T_1}{(1+\frac{\irho}{\rho})\ui\sin\Theta}
-\frac{a_{\tau}}{|{\bf u}_{\infty}-{\bf u}|^2}\right)
{\bm\tau}_{\rm sh}.
\end{equation}
Since
\begin{equation}\label{oblique-sh}
{\bm\beta}_{\rm sh}\cdot {\bm\nu}_{\rm sh}=  -\frac{a_{\nu}}{|{\bf u}_{\infty}-{\bf u}|^2}\ge 1-M^2\ge \sigma>0\quad\tx{on $\ol{\shock}$}
\end{equation}
due to Proposition \ref{proposition-wp-fbpbd}(iii), the boundary condition \eqref{bc-for-Q-on-shock} is strictly oblique on $\rm{int}\,\Gam_{\rm sh}$.
Therefore, the Hopf's lemma implies that $q(=e^Q)$ cannot attain its local extrema on $\rm{int}\,\Gam_{\rm sh}$.
\smallskip

{\textbf{4.}}
Suppose that $q$ attains its local minimum at $P_{2,L}^f$. Then, there exists a small constant $r>0$ so that $\der ({B_r(P_{2,L}^f)}\cap {\Om})\subset
(\shock \cup \der B_r(P_{2,L}^f)\cup \Gam_{{\rm{cutoff}}}) $, $P_0\not\in \ol{B_r(P_{2,L}^f)}\cap \ol{\Om}$ and
\begin{equation}
\label{local-min-Q}
Q(P_{2,L}^f)=\inf_{B_r(P_{2,L}^f)\cap \Om}Q.
\end{equation}
Since $Q$ is bounded in $\ol{B_r(P_{2,L}^f)}\cap \ol{\Om}$, there exists a constant $m>0$ so that
\begin{equation*}
0\le Q-Q(P_{2,L}^f)\le m\quad\tx{in $\ol{B_r(P_{2,L}^f)}\cap \ol{\Om}$.}
\end{equation*}
Let us set $Q^*(\rx):=Q(\rx)-Q(P_{2,L}^f)$, and define a unit vector field ${\bm\beta}$ on $\Gam_{{\rm{cutoff}}}\cup \shock$ by
\begin{equation*}
  {\bm\beta}:=\begin{cases}
  -{\bf n}_c\quad&\mbox{on $\Gam_{{\rm{cutoff}}}$}\\
  \frac{{\bm\beta}_{\rm sh}}{|{\bm\beta}_{\rm sh}|}
   \quad&\mbox{on $\shock$}
  \end{cases}
\end{equation*}
for ${\bm\beta}_{\rm sh}$ given by \eqref{definition-bc-Q-on-shock}.
By \eqref{equation-for-Q}, \eqref{4.18} and \eqref{bc-for-Q-on-shock}, $Q^*\in C^2({B_r(P_{2,L}^f)}\cap {\Om})\cap C^0(\ol{{B_r(P_{2,L}^f)}\cap {\Om}})$ satisfies
\begin{equation*}
  \begin{split}
  &\sum_{i,j=1}^2 \der_{x_i}(a_{ij}\der_{x_j}Q^*)=0\quad\tx{in ${B_r(P_{2,L}^f)}\cap {\Om}$},\\
  &{\bm\beta}\cdot \nabla Q^*=0
 \quad\mbox{on}\,\, B_r(P_{2,L}^f)\cap (\Gam_{{\rm{cutoff}}}\cup \,\shock)\setminus\{P_{2,L}^f\}.
  \end{split}
\end{equation*}
Let ${\bf n}$ be the inward unit normal on $(\Gam_{{\rm{cutoff}}}\cup \,\shock)\setminus\{P_{2,L}^f\}$. Then, it directly follows from the definition of ${\bm\beta}$ and \eqref{oblique-sh} that ${\bm\beta}$ satisfies
\begin{equation}
\label{standard-obliqueness}
  {\bm\beta}\cdot {\bf n}\ge \om_0\quad\tx{on $B_r(P_{2,L}^f)\cap (\Gam_{{\rm{cutoff}}}\cup \,\shock)\setminus\{P_{2,L}^f\}$}
\end{equation}
for some constant $\om_0>0$.

Set
\begin{equation*}
  {\bm\tau}_{c}:=(\sin \tw, -\cos \tw).
\end{equation*}

{\emph{Claim: ${\bm\beta}_{\rm sh}(P_{2,L}^f)\cdot {\bm \tau}_c>0.$}}
\smallskip

The claim will be verified in the next step. For now, we assume that the claim is true, then derive a contradiction from \eqref{local-min-Q} by showing that $Q^*(P_{2,L}^f)>0$.

Let $\Sigma_0$ be the line parallel to $(\cos \tw, \sin \tw)$ with passing through $P_{2,L}^f$. And, for $\sigma>0$, let $\Sigma_{\sigma}$ be the line parallel to ${\bf n}_{\sigma}:=(\cos (\tw-\sigma), \sin (\tw-\sigma))$ with passing through $P_{2,L}^f$. Then, ${\bm \tau}_c$ is a unit normal of $\Sigma_0$. Since $\Theta=\tw$ at $P_{2,L}^f$, \eqref{definition-nu-shock} easily yields that
\begin{equation}\label{corner-normal-relation1}
  {\bm\nu}_{\rm sh}\cdot {\bm \tau}_c=\frac{\iu\sin \tw}{|{\bf u}_{\infty}-{\bf u}|}>0\quad\tx{at $P_{2,L}^f$}.
\end{equation}
Therefore, one can choose a small constant $\sigma\in(0, \frac{\tw}{4})$ so that
\begin{equation}
\label{corner-obliqueness-sptcurve}
 \min\{ \frac{{\bm \beta}_{\rm sh}}{{|\bm \beta}_{\rm sh}|}\cdot {\bf n}_{\sigma}, \,\,{\bm\nu}_{\rm sh}\cdot {\bf n}_{\sigma}, -{\bf n}_c\cdot {\bf n}_{\sigma}\}\ge \lambda_0 \quad\tx{at $P_{2,L}^f$}
\end{equation}
for some constant $\lambda_0>0$. This shows that $\Sigma_{\sigma}$ yields a supporting line at $P_{2,L}^f$ with satisfying all the conditions to apply \cite[Lemma 4.4.2]{CF3}. Since $\shock$ is $C^{1,\alpha}$ up to $P_{2,L}^f$, and $\Gam_{\rm cutoff}$ is a line, it follows from \cite[Lemma 4.4.2]{CF3} that, by flattening $\shock$ near $P_{2,L}^f$ and using \eqref{standard-obliqueness} and \eqref{corner-obliqueness-sptcurve}, we can apply \cite[Lemma 2.2]{Lieberman} to obtain that $Q^*(P_{2, L}^f)>0$. But this contradicts to \eqref{local-min-Q}. Therefore, we conclude that $Q$ cannot attain its local minimum at $P_{2, L}^f$, and this implies that $q(=e^Q)$ cannot attain its local minimum at $P_{2,L}^f$. Also, one can similarly argue that $q$ cannot attain its local maximum at $P_{2,L}^f$ either.
\smallskip

{\textbf{5.}} It remains to verify the claim stated in the previous step. The claim is verified by a direct computation. Observe that
\begin{equation*}
  {\bm\nu}_{\rm sh}\cdot{\bm \tau}_c=\frac{\iu\sin \tw}{|{\bf u}_{\infty}-{\bf u}|},\quad
  {\bm \tau}_{\rm sh}\cdot {\bm \tau}_c=\frac{q-\iu\cos\tw}{|{\bf u}_{\infty}-{\bf u}|}\quad\tx{at $P_{2,L}^f$}
\end{equation*}
because $\Theta=\tw$ at $P_{2,L}^f$. Then, a straightforward computation with using \eqref{definition-anu-atau} and \eqref{definition-bc-Q-on-shock} yields that
\begin{equation*}
{\bm\beta}_{\rm sh}(P_{2,L}^f)\cdot {\bm \tau}_c=
\frac{1}{ |{\bf u}_{\infty}-{\bf u}|}\left(
\mcl{J}_1+\mcl{J}_2\right)
\end{equation*}
for
\begin{equation*}
  \begin{split}
  &\mcl{J}_1=q\tan \tw+\frac{(\iu\cos \tw-q)((2-M^2)q+M^2\cos \tw)}{(1+\frac{\irho}{\rho})\iu\sin \tw},\\
  &\mcl{J}_2=M^2\iu\sin \tw \left(\frac{(\iu\cos \tw-q)^2}{|{\bf u}_{\infty}-{\bf u}|^2}-\frac{({\bf u}\cdot{\bm\nu}_{\rm sh})^2}{q^2}\right).
  \end{split}
\end{equation*}

By Proposition \ref{proposition-wp-fbpbd}(ii) and \eqref{definition-phisical-var-bdd}, we have
\begin{equation*}
  \iu\cos \tw-q\ge \iu\cos \tw-C\mathfrak{q}_{\gam}(\eps)
\end{equation*}
for $\mathfrak{q}_{\gam}(\eps)$ from Lemma \ref{lemma-shock-polar2}. Therefore, if we reduce $\eps_4$ further, then we have $ \iu\cos \tw-q>0$ at $P_{2,L}^f$ provided that $M_{\infty}\ge \frac{1}{\eps_4}$. Then, we obtain that $\mcl{J}_1>0$ because $M^2<1$ at $P_{2,L}^f$ due to Proposition \ref{proposition-wp-fbpbd}(iii).

Since $\Theta=\tw$ at $P_{2,L}^f$, we have
\begin{equation*}
 \frac{({\bf u}\cdot{\bm\nu}_{\rm sh})^2}{q^2}
 =\frac{|(\cos \tw, \sin \tw)\cdot (\iu -q\cos \tw, -q\sin \tw)|^2}{|{\bf u}_{\infty}-{\bf u}|^2}
 =\frac{(\iu\cos \tw-q)^2}{|{\bf u}_{\infty}-{\bf u}|^2}.
\end{equation*}
This implies that $\mcl{J}_2=0$.
So we conclude that ${\bm\beta}_{\rm sh}(P_{2,L}^f)\cdot {\bm \tau}_c=
\frac{\mcl{J}_1}{ |{\bf u}_{\infty}-{\bf u}|}>0$ provided that $M_{\infty}\ge \frac{1}{\eps_4}$.
\smallskip

{\textbf{6.}} It remains to show that $q$ cannot attain its local minimum at neither $P_1$ nor $P_{3, L}^f$.

Along the boundary $\Gam_{{\rm sym}}\cup\Gam_{b}$, since $|{\bf u}|>0$ away from $P_0$, the slip boundary condition ${\bf u}\cdot {\bf n}_s=0$ for the inward unit normal ${\bf n}_s$ implies that $({\bf n}_s)^{\perp} \parallel (\cos\Theta,\sin\Theta)$ on $(\Gam_{{\rm sym}}\cup\Gam_{b})\setminus \{P_0\}$. This implies that $\Theta=0$ on $\Gam_{{\rm sym}}$, and that $\Theta=\tw$ on $\Gam_{b}\cap \{x_2\ge h_0\}$. So we have ${\bf n}_s^{\perp}\cdot \nabla\Theta=0$  on $\Gam_{{\rm sym}}\cup (\Gam_{b}\cap \{x_2\ge h_0\})$ for $h_0$ from Definition \ref{definition-bluntbody-ftn}. A direct computation with using \eqref{4.3} shows that it is equivalent to
\begin{equation}\label{4.19}
\nabla Q\cdot {\bf n}_s=0\quad\tx{on $\Gam_{{\rm sym}}\cup (\Gam_{b}\cap \{x_2\ge h_0\})$.}
\end{equation}

At $P_1$, we have ${\bm \beta}_{\rm sh}=\overrightarrow{(1,0)}$ because $f'(0)=0$. This can be directly checked from \eqref{nlbvp-in-fbp-bdd} and \eqref{free-bc}. Since $\overrightarrow{(1,0)}$ is tangential to $\Gam_{\rm sym}$, one can adjust the argument in Step 4 with using \eqref{4.19} to conclude that $q$ cannot attain its local minimum at $P_1$. For a similar reason, $q$ cannot attain its local minimum at $P_{3,L}^f$. Finally, one can easily adjust the argument in Step 4 and the argument right above to conclude that $q$ cannot attain its local maximum at the points $\{P_1, P_{2,L}^f, P_{3,L}^f\}$. This completes the proof of Lemma \ref{lemma-critical-val-speed}.
\end{proof}	
\end{lemma}

\begin{lemma}
\label{lemma-q-monotonicity}
For $\eps_4$ from Lemma \ref{lemma-critical-val-speed}, if $M_{\infty}(=\frac{1}{\eps})\ge \frac{1}{\eps_4}$, then $q$ monotonically increases along $\shock$ from $P_1$ to $P_{2,L}^f$.

\begin{proof}
By Lemmas \ref{lemma-velocity1-positivity} and \ref{lemma-velocity2-positivity}, we have $\Theta(P_{2,L}^f)>0$ and $\Theta(P_1)=0$. Then it follows from Lemma \ref{lemma-shock-polar-general} that
\begin{equation}
q(P_1)<q(P_{2,L}^f).
\end{equation}
Suppose that $q$ does not monotonically increase along $\shock$ from $P_1$ to $P_{2, L}^f$. Then, there exist two distinct points $P^*=(f(x_2^*), x_2^*)$ and $P^{\sharp}=(f(x_2^{\sharp}), x_2^{\sharp})$ satisfying the following two properties:
\begin{itemize}
\item[(i)] $0\le x_2^*<x_2^{\sharp}\le L$;
\item[(ii)] $q(P^*)>q(P^{\sharp}).$
\end{itemize}
Then, one can find at least one $x_2^A\in [0, x_2^{\sharp}]$ so that a point $A:=(f(x_2^{A}), x_2^A)$ satisfies that
\begin{equation*}
  q(A)=\max_{x_2\in [0, x_2^{\sharp}]} q(f(x_2), x_2).
\end{equation*}
In other words, $q(A)$ is the maximal local speed along $\shock$ between $P_1$ and $P^{\sharp}$. Next, we fix $x_2^B\in [x_2^A, L]$ so that a point $B:=(f(x_2^B), x_2^B)$ satisfies that
\begin{equation*}
  q(B)=\min_{x_2\in [x_2^A, L]} q(f(x_2), x_2),
\end{equation*}
that is, $q(B)$ is the minimal local speed along $\shock$ between $P_A$ and $P^{f}_{2,L}$. Finally, let us fix $x_2^{\hat{A}}\in[0, x_2^B]$ so that a point $\hat A:=(f(x_2^{\hat A}), x_2^{\hat A})$ satisfies that
\begin{equation*}
  q(\hat A)=\max_{x_2\in [0, x_2^{B}]} q(f(x_2), x_2).
\end{equation*}
Since $q(B)\le q(P^{\sharp})<q(P^*)\le q(\hat A)$, we have $q(\hat A)-q(B)>0$. Let us set
\begin{equation*}
  \mu:=\frac{q(\hat A)-q(B)}{4}.
\end{equation*}
By Proposition \ref{proposition-wp-fbpbd}(ii) and \eqref{definition-q}, there exists a constant $C_*>0$ satisfying that, for any $r>0$,
\begin{equation*}
\underset{B_{r}(\rx)\cap \Om}{\rm Osc}  q \le C_*\mathfrak{q}_{\gam}(\eps) r^{\alp}\quad\tx{for all $\rx\in \ol{\Om}$.}
\end{equation*}
So we have
\begin{equation}\label{q-osc-R}
  \underset{B_{R_*}(\rx)\cap \Om}{\rm Osc}  q \le \mu\quad\tx{for $R_*:=
  \left(\frac{\mu}{C_*\mathfrak{q}_{\gam}(\eps)}\right)
  ^{\frac{1}{\alp}}$}.
\end{equation}

Since $q\in C^{\alp}(\ol{\Om})$, one can fix a point $A_1\in \der (B_{R_*}(\hat A)\cap \Om)$ so that
\begin{equation*}
  q(A_1):=\max_{\ol{B_{R_*}(\hat A)\cap \Om}} q.
\end{equation*}
Note that $A_1\neq \hat{A}$ because a local extremum of $q$ cannot be attained on $\shock$ due to Lemma \ref{lemma-critical-val-speed}. Inductively, for each $n\ge 2$, one can fix a point $A_n\in \der (B_{R_*}(A_{n-1})\cap \Om)$ so that
\begin{equation*}
  q(A_n)=\max_{\ol{B_{R_*}(A_{n-1})\cap\Om}} q.
\end{equation*}
Define
\begin{equation*}
  N:=\min\{n: A_n\in \der \Om,\,\,n\ge 1\}.
\end{equation*}
Since $\ol{\Om}$ is compact, $N$ is finite. Furthermore, by \eqref{4.19} and Hopf's lemma, $A_n$ lies on $\Gam_{b}\cap \{0<x_2< h_0\}$, the curved part of $\Gam_{b}$. Note that $q(A_N)>0$ because $q(A_N)\ge q(P_1)$.

Similarly to the sequence $\{A_1, \cdots, A_N\}$, we construct a sequence of points by using the point $B$. We fix a point $B_1\in \der (B_{R_*}(B)\cap \Om)$ so that $\displaystyle{q(B_1)=\min_{\ol{B_{R_*}(B)\cap \Om}}} q$. Note that $B_1\neq B$ due to Lemma \ref{lemma-critical-val-speed}. Inductively, for each $m\ge 2$, fix $B_m\in \der (B_{R_*}(B_{m-1})\cap \Om)$ so that $\displaystyle{q(B_m)=\min_{\ol{B_{R_*}(B_{m-1})\cap \Om}}} q$. And, we define
\begin{equation*}
  M:=\min\{m: B_m\in \der \Om,\,\,m\ge 1\}.
\end{equation*}
Then, $M$ is finite, and we have $B_M\in \Gam_{b}\cap \{0\le x_2\le h_0\}$.

By the definition of $R_*$ given in \eqref{q-osc-R}, we have
\begin{equation*}
 q(A_n)-q(B_m)\ge q(\hat{A})-q(B)=4\mu \quad\tx{for $1\le n\le N$, $1\le m\le M$,}
\end{equation*}
from which it follows that
\begin{equation}
\label{location-of-B-on-W}
  B_{R_*}(A_n)\cap B_{R_*}(B_m)\cap \Om =\emptyset
  \quad\tx{for $1\le n\le N$, $1\le m\le M$.}
\end{equation}
Since $x_2^{\hat{A}}<x_2^B$, \eqref{location-of-B-on-W} implies that $B_M$ must lie on $\Gam_{b}\cap \{0<x_2<h_0\}$, that is, $B_M\neq P_0$.

Similarly to Step 6 in the proof of Lemma \ref{lemma-critical-val-speed}, along the boundary $\Gam_{b}$, the slip boundary condition ${\bf u}\cdot {\bf n}_s=0$ for the inward unit normal ${\bf n}_s$ implies that the vector field $(\cos \Theta, \sin \Theta)$ yields a unit tangential  on $\Gam_{b}$. Since $(\cos \Theta, \sin \Theta)$ points away from $P_0$ along $\Gam_{b}$, by the convexity of the blunt body $W_b$, stated in Definition \ref{definition-bluntbody-ftn} ($b_4$), we have
\begin{equation}
\label{convexity-along-w-Theta}
  (\cos \Theta, \sin \Theta)\cdot \nabla\Theta<0
  \quad\tx{on $\Gam_{b}\cap\{0<x_2< h_0\}$.}
\end{equation}
Note that we have $\nabla Q\cdot {\bf n}_s=(\cos \Theta, \sin \Theta)\cdot \nabla\Theta$ on $\Gam_{b}$ by a direct computation with using \eqref{4.3}, for the inward unit normal ${\bf n}_s$ on $\Gam_{b}$. Here, $Q$ is defined by $Q=\ln q$ as in the previous lemma. Therefore, \eqref{convexity-along-w-Theta} implies that
\begin{equation*}
\nabla Q\cdot {\bf n}_s<0\quad\tx{on $\Gam_{{\rm sym}}\cup (\Gam_{b}\cap \{0<x_2< h_0\})$,}
\end{equation*}
from which $(\nabla Q\cdot {\bf n}_s)(B_M)<0$ is given. But this contradicts to Hopf's lemma because $B_M$ is a local minimum point of $Q(=\ln q)$ in $\ol{B_{R_*}(B_{M-1})\cap \Om}$. Therefore, we conclude that $q$ monotonically increases along $\shock$ from $P_1$ to $P_{2,L}^f$.
\end{proof}
\end{lemma}

Finally, we are ready to prove Proposition \ref{prop-convex} by applying previous lemmas.

\begin{proof}[Proof of Proposition \ref{prop-convex}] The proof is divided into two steps.

{\textbf{1.}}
By \eqref{definition-nu-shock}, a unit tangential ${\bm \tau}_{\rm sh}$ along $\shock$ is given by
$
  {\bm\tau}_{\rm sh}=\frac{(u_2, \iu-u_1)}{|{\bf u}_{\infty}-{\bf u}|}.
$
Then the angle $\beta$ between ${\bm \tau}_{\rm sh}$ and the positive $x_1$-axis is given by
\begin{equation*}
  \beta=\arctan \left(\frac{\iu-u_1}{u_2}\right).
\end{equation*}
By Lemma \ref{lemma-velocity2-positivity} and \eqref{u-ui-diff}, we have $\beta\in(0, \frac{\pi}{2})$ on ${\rm int} \shock$.

{\emph{Claim: For the local speed $q=|{\bf u}|$, we have
\begin{equation*}
  \frac{d }{dq}\sin^2\beta <0\quad\tx{on ${\rm int}\,\shock$.}
\end{equation*}}}

We first assume that the claim holds true, then prove that $\fshock''(x_2)\ge 0$ for $0<x_2<L$. Since $\shock$ is the graph of $x_1=\fshock(x_2)$, the definition of $\beta$ implies that $\fshock'(x_2)=\cot \beta$ for $x_2\in[0, L]$. So the chain rule yields that
\begin{equation*}
  \fshock''(x_2)=\frac{d\cot\beta}{dx_2}=
  -\csc^2\beta\frac{d\beta}{dq}\frac{dq}{dx_2}\quad\tx{on ${\rm int}\,\shock$.}
\end{equation*}
By Lemma \ref{lemma-q-monotonicity} and the fact of $\beta\in(0, \frac{\pi}{2})$, if the claim stated above holds true, then we obtain that $\fshock''(x_2)\ge 0$ because of $\frac{d\beta}{dq}=\frac{1}{2\sin\beta\cos\beta}\frac{d}{dq}\sin^2\beta<0$ on ${\rm int} \shock$.

\smallskip

{\textbf{2.}} To complete the proof, it remains to verify the claim stated in Step 1. Direct computations with using \eqref{definition-q}, \eqref{u-into-q-Theta} and the definition of $\beta$ show that the Rankine-Hugoniot conditions $\rho{\bf u}\cdot {\bm \nu}_{\rm sh}=\irho {\bf u}_{\infty}\cdot {\bm\nu}_{\rm sh}$ and ${\bf u}\cdot {\bm \tau}_{\rm sh}={\bf u}_{\infty}\cdot {\bm \tau}_{\rm sh}$ on $\shock$ can be written as
\begin{equation*} q\sin(\Theta-\beta)=
-\frac{\rho_{\infty}}{\rho}u_{\infty}\sin\beta,\quad
q\cos(\Theta-\beta)=u_{\infty}\cos\beta.
\end{equation*}
Here, ${\bm \nu}_{\rm sh}$ represents the inward unit normal on $\shock$. From this expression, we obtain that
\begin{equation*}
\frac{\iu^2}{q^2}(1-\sin^2\beta)
+\frac{\irho^2\iu^2}{\rho^2q^2}\sin^2\beta=1\quad
\tx{on ${\rm int}\shock$}.
\end{equation*}
We solve this for $\sin^2\beta$ to get
\begin{equation}
\label{sin-beta-q-function}
  \sin^2\beta=\frac{q^2-\iu^2}{\iu^2(\frac{\irho^2}{\rho^2}-1)}
  \quad\tx{on ${\rm int} \,\shock$.}
\end{equation}
We differentiate the Bernoulli's law $\frac 12 q^2+\frac{\gam \rho^{\gam-1}}{\gam-1}=\frac 12 \iu^2+\frac{\gam \irho^{\gam-1}}{\gam-1}$, and use \eqref{definition-phisical-var-bdd} to get $\frac{d\rho}{dq}=-\frac{\rho q}{c^2}$. Then, we differentiate \eqref{sin-beta-q-function} with respect to $q$ to get
\begin{equation}
\label{derivative-sin-beta-q}
  \frac{d}{dq}\sin^2\beta=
  \frac{2q}{\iu^2(\frac{\irho^2}{\rho^2}-1)^2}
  \left(\frac{\irho^2}{\rho^2}-1
  -\frac{q^2-\iu^2}{c^2}\frac{\irho^2}{\rho^2}\right)
  \quad\tx{on ${\rm int}\, \shock$}.
\end{equation}
By rearranging the Bernoulli's law, we have $\frac{q^2-\iu^2}{c^2}=
\frac{2}{\gam-1}(\frac{\irho^{\gam-1}}{\rho^{\gam-1}}-1)$. We substitute this expression into \eqref{derivative-sin-beta-q} to finally obtain that
\begin{equation*}
  \frac{d}{dq}\sin^2\beta=
  \frac{2q}{\iu^2(\frac{\irho^2}{\rho^2}-1)^2}\mcl{F}(\frac{\irho}{\rho})
\end{equation*}
for $\mcl{F}(\eta)$ given by
\begin{equation*}
  \mcl{F}(\eta)=\frac{\gam+1}{\gam-1}\eta^2
  -\frac{2}{\gam-1}\eta^{\gam+1}-1.
\end{equation*}
By \eqref{rho-relation}, we have $0<\frac{\irho}{\rho}<1$. Since $\mcl{F}(1)=0$ and $\mcl{F}'(\eta)=\frac{2(\gam+1)}{\gam-1}\eta(1-\eta^{\gam-1})>0$ for $\gam>1$ and $0<\eta<1$, we have $\mcl{F}(\frac{\irho}{\rho})<0$ on $\shock$. And, this finally yields that $\frac{d}{dq}\sin^2\beta<0$ on ${\rm int}\,\shock$. This verifies the claim stated in Step 1. Therefore, the proof of Proposition \ref{prop-convex} is completed.

\end{proof}


\vspace{.25in}
\noindent
{\bf Acknowledgements:}
The research of Myoungjean Bae was supported in part by  Samsung Science and Technology Foundation
under Project Number SSTF-BA1502-02.
The research of Wei Xiang was supported in part by the Research Grants Council of the HKSAR, China (Project CityU 11303518, Project CityU 11332916, and Project CityU 11304817).


\end{document}